\documentclass[10pt]{amsart}

\usepackage{blkarray}
\usepackage{multirow}
\usepackage{amsmath,fixltx2e}
\usepackage{amsthm}
\usepackage{amssymb}
\usepackage{amsfonts}
\usepackage{mathtools}
\usepackage{MnSymbol}
\usepackage{graphicx}
\usepackage{hyperref}
\usepackage{tikz-cd} 
\allowdisplaybreaks

\newcommand{\overbar}[1]{\mkern 1.5mu\overline{\mkern-1.5mu#1\mkern-1.5mu}\mkern 1.5mu}

\newtheorem{remark}{Remark}[section]

\newtheorem{theorem}{Theorem}[section]
\newtheorem{lemma}[theorem]{Lemma}

\newtheorem{proposition}[theorem]{Proposition}
\newtheorem{corollary}[theorem]{Corollary}
\newtheorem{definition}[theorem]{Definition}

\numberwithin{equation}{section}

\usepackage[utf8]{inputenc}
\usepackage{graphicx}

\author[R. Gambheera]{Rusiru Gambheera}
\author[C.D. Popescu]{Cristian D. Popescu}
\address{Department of Mathematics, University of California Santa Barbara, CA 93106-3080, USA}
\address{Dept. of Mathematics, University of California at San Diego, San Diego, CA 92093, USA}
\email{rusiru@ucsb.edu}
\email{cpopescu@ucsd.edu}



\thanks{}

\date{}

\begin{document}

\begin{abstract} We consider a finite, abelian, CM extension $H/F$ of a totally real number field $F$, and construct a $\mathbb{Z}_p[[G(H_\infty/F)]]-$module $\nabla_S^T(H_\infty)_p$, where $p>2$ is a prime and $H_\infty$ is the cyclotomic $\Bbb  Z_p$--extension of $H$.  This is the Iwasawa theoretic analogue of a module introduced by Ritter and Weiss in \cite{Ritter-Weiss} and studied further by Dasgupta and Kakde in \cite{Dasgupta-Kakde}. Our main result states that the $\Bbb Z_p[[G(H_\infty/F]]^-$--module $\nabla_S^T(H_\infty)_p$ is of projective dimension $1$, is quadratically presented, and that its Fitting ideal is principal, generated by an equivariant $p$--adic $L$--function $\Theta_S^T(H_\infty/F)$. As a first application, we compute the Fitting ideal of an arithmetically interesting $\Bbb Z_p[[G(H_\infty/F)]]^-$--module $X_S^{T,-}$, which is a variant of the classical unramified Iwasawa module $X$ (the Galois group of the maximal abelian, unramified, pro-$p$ extension of $H_\infty$), extending earlier results of Greither--Kataoka--Kurihara \cite{Greither-Kataoka-Kurihara}.
These are all instances of what is now called an Equivariant Main Conjecture in the Iwasawa theory of totally real number fields, and refine the classical main conjecture, proved by Wiles in \cite{wiles}. As a final application, we give a short, Iwasawa theoretic proof of the minus part of the far-reaching Equivariant Tamagawa Number Conjecture for the Artin motive $h_{H/F}$, a result also obtained, independently and with different (Euler system) methods, by Bullack--Burns--Daoud--Seo \cite{Bullach-Burns-Daoud-Seo} and Dasgupta--Kakde--Silliman \cite{Dasgupta-Kakde-Silliman-ETNC}.
    \end{abstract}

\title[Equivariant Iwasawa Theory for Ritter--Weiss Modules and Applications]{ Equivariant Iwasawa Theory for Ritter--Weiss Modules\\ and Applications}
\maketitle

\section{Introduction}
\setcounter{equation}{0}

Let $H/F$ be a finite, abelian, CM extension of a totally real number field $F$, whose Galois group is denoted by $G$. We let $S$ and $T$ be two finite, nonempty, disjoint sets of places in $F$, such that $S$ contains the set $S_{\infty}(F)$ of archimedean places in $F$. When there is no risk of confusion, we denote the sets of places in $H$ above places in $S$ and $T$, also by $S$ and $T$, respectively. From the data $(H/F, S, T)$, one can construct a meromorphic $G$--equivariant Artin $L$--function 
$$\Theta_{S, H/F}^T: \Bbb C\longrightarrow \Bbb C[G],$$
which is holomorphic on $\Bbb C\setminus\{1\}$. Moreover, for this data, following the work \cite{Ritter-Weiss}  of Ritter and Weiss on generalized Tate sequences, Dasgupta and Kakde \cite{Dasgupta-Kakde} defined the arithmetically meaningful Ritter-Weiss $\Bbb Z[G]$--module $\nabla_S^T(H)$. In particular, when $S\cup T$ contains all the ramified primes,  there is a short exact sequence of $\Bbb Z[G]$--modules
\[\begin{tikzcd}
0\arrow{r} & Cl_S^T(H)\arrow{r} & \nabla_S^T(H)\arrow{r} & {\rm Div}^0_S(H)\arrow{r} & 0 
\end{tikzcd}
\]
where $Cl_S^T(H)$ is a certain $(S,T)$--ray class group of $H$ and ${\rm Div}_S^0(H)$ is the group of formal divisors of degree $0$ of $H$, supported at $S$. Moreover, the extension class of the exact sequence above is explicitly constructed out of the local and global fundamental classes in class--field theory.
\\

Now, let $p$ be an odd prime, let $H_{\infty}$ be the cyclotomic $\mathbb{Z}_p-$extension of $H$ and let $\mathcal{G}:=Gal(H_{\infty}/F)$. For the data $(H/F, S, T, p)$, we define an Iwasawa--Ritter-Weiss module $\nabla_S^T(H_{\infty})_p$, at the top of the cyclotomic tower $H_\infty/F$. This is a $\mathbb{Z}_p[[\mathcal{G}]]$--module, obtained by taking a projective limit under certain norm maps of the $p$--adic Ritter-Weiss modules $\nabla_S^T(H_n)_p:=\nabla_S^T(H_{n})\otimes \mathbb{Z}_p$ at each finite layer $H_n/F$ of $H_\infty/F$.

Our equivariant main conjecture, the main result of this paper, relates in a precise way the $0$--th Fitting ideal of the $\Bbb Z_p[[\mathcal G]]^-$--module $\nabla_S^T(H_{\infty})_p^-$  (where $\ast^-$ denotes the $(-1)$--eigenspace of $\ast$ under the action of the unique complex conjugation automorphism in $\mathcal{G}$) to an equivariant $p$--adic $L$-function and shows that the module in question has certain desirable homological properties. The precise statement is the following.

\begin{theorem}\label{FUIMC}
For the data $(H/F, S, T, p)$ as above, let $S_\infty$ be the set of infinite places in $F$, $S_{ram}$ the set of primes in $F$ that ramify in $H_\infty/F$ and $S_p$ the set of primes in $F$ above $p$. Suppose that the disjoint sets $S$ and $T$ satisfy the following additional properties: 

\begin{equation*} S_{\infty}\cup S_p\subseteq S,\quad  S_{ram}\subseteq S\cup T,\qquad  T\not\subseteq S_{ram}\end{equation*}
Then, we have an equality of $\Bbb Z_p[[\mathcal G]]^-$--ideals  $${\rm Fitt}_{\mathbb{Z}_p[[\mathcal{G}]]^-}(\nabla_S^T(H_{\infty})_p^-)=(\Theta_S^T(H_{\infty}/F)).$$
Moreover, there is a $k\in\Bbb Z_{>0}$, such that we have a short exact sequence of $\Bbb Z_p[[\mathcal G]]^-$--modules
\[0\xrightarrow[]{} (\mathbb{Z}_p[[\mathcal{G}]]^-)^k \xrightarrow[]{} (\mathbb{Z}_p[[\mathcal{G}]]^-)^k\xrightarrow[]{}  \nabla_S^T(H_{\infty})_p^-\xrightarrow[]{}0. \] 
 Consequently, the $\Bbb Z_p[[\mathcal G]]^-$--module $\nabla_S^T(H_\infty)_p^-$ is quadratically presented and of projective dimension $1$.

\end{theorem}

In the theorem above, ${\rm Fitt}$ is the $0$--th Fitting ideal. Further, $\Theta_S^T(H_{\infty}/F)\in \mathbb{Z}_p[[\mathcal{G}]]^-$ is the associated equivariant $p$--adic $L$--function, obtained by taking the projective limit of the special values $\Theta_{S,H_n/F}^T(0)$ of the corresponding global equivariant $L$--functions at the finite levels $H_n/F$ of the Iwasawa tower $H_\infty/F$. \\

As mentioned earlier, the module of interest $\nabla_S^T(H_{\infty})_p^-$ is constructed as a projective limit with respect to certain norm maps between the Ritter--Weiss modules at the finite levels in the Iwasawa tower $H_\infty/H$. This is fundamentally different from the Iwasawa--Selmer module $Sel_S^T(H_\infty)_p^-$ we considered in \cite{gambheera-popescu},  which was constructed by using the inclusion maps between the Selmer modules at the finite levels. Therefore, the above theorem does not follow directly from the equivariant main conjecture proved in \cite{gambheera-popescu}.
\\

As a first application of the main theorem, we compute the $0$--th Fitting ideal of the $(-1)$--eigenspace of Iwasawa module $X_S^{T}:=\varprojlim_n Cl_S^T(H_n)_p$ over the equivariant Iwasawa algebra $\mathbb{Z}_p[[\mathcal{G}]]^-$. Here, the projective limit is taken with respect to the norm maps of $(S,T)$--ray class groups at the finite levels of the Iwasawa tower. One can view the following theorem as progress towards the still open problem of computing the Fitting ideal of the minus part of the classical unramified Iwasawa module $X:=\varprojlim_n Cl(H_n)_p$ (the Galois group of the maximal pro--$p$ abelian extension of $H_\infty$ which is unramified everywhere.)

\begin{theorem}\label{S,T-modified} Let $p$ be an odd prime, $S_\infty$ be the set of infinite places in $F$ and $S_{ram}$ be the set of primes in $F$ that ramify in the extension $H_\infty/F$. Let $S_p$ be the primes in $F$ that are above $p$. Suppose that $S$ and $T$ satisfy the following properties. 

\begin{equation*} S_{\infty}\cup S_p\subseteq S,\quad  S_{ram}\subseteq S\cup T,\qquad  T\not\subseteq S_{ram}\end{equation*}
Then, we have an equality of $\Bbb Z_p[[\mathcal G]]^-$--ideals
$$Fitt_{\mathbb{Z}_p[[\mathcal{G}]]^-}(X_S^{T,-})=(\Theta_T^{S\cap S_{ram}}(H_{\infty}/F))\prod_{v\in (S\cap S_{ram})}(N(A)\Delta B^{r_B-2}\,\mid\,\mathcal{G}_v=A\times B ,  A- torsion)$$
Above, $\mathcal{G}_v$ is the decomposition group of the prime $v$ in the extension $H_{\infty}/F$ and $A$ and $B$ vary through all the subgroups of $\mathcal G_v$, such that $A$ is torsion (therefore finite) and $\mathcal{G}_v=A\times B$. Further, for any $A$ and $B$ as above, $N(A)=\sum_{g\in A}g$, and $r_B$ and $\Delta B$ are the minimum number of generators of $B$ and the augmentation ideal of $\Bbb Z[B]$, respectively.
\end{theorem}

The proof of the above result hinges upon the short exact sequence of $\mathbb{Z}_p[[\mathcal{G}]]^--$modules
$$0\xrightarrow[]{} X_S^{T,-}\xrightarrow[]{} \nabla_S^T(H_{\infty})_p^-\xrightarrow[]{} {\rm Div}_S(H_{\infty})_p^-\xrightarrow[]{} 0,$$
where ${\rm Div}_S(H_{\infty})$ is the free abelian group generated by the primes in $H_\infty$ above primes in $S$. Once we establish the above exact sequence, we observe that since Theorem \ref{FUIMC} computes ${\rm Fitt}_{\mathbb{Z}_p[[\mathcal{G}]]^-}(\nabla_S^T(H_\infty)_p^-)$, Kataokas's theory of shifted Fitting ideals \cite{Kataoka} reduces the problem to computing the first shifted Fitting ideal  
${\rm Fitt}^{[1]}_{\Bbb Z_p[[\mathcal G]]^-}({\rm Div}_S(H_{\infty})_p^-)$. In \cite{Greither-Kataoka-Kurihara}, Greither, Kurihara and Kataoka compute ${\rm Fitt}_{\mathbb{Z}_p[[\mathcal{G}]]^-}^{[1]}({\rm Div}_S(H_{\infty})_p^-)$ when the $p-$parts of the inertia groups of the primes in $S$ are cyclic. In the Appendix below, we generalize this result and calculate it for all sets $S$ as above, which permits us to prove the Theorem above. \\

As a second and final application of the main theorem, in \S5 below (see Theorem \ref{ETNC}) we give a new proof of the minus part ${\rm ETNC}(H/F)^-$, away from the prime $2$, of the far reaching Equivariant Tamagawa Number Conjecture 
${\rm ETNC}(H/F)$ for the Artin motive associated to $H/F$, with coefficients in the group ring $\Bbb Z[G(H/F)]$, originally formulated by Burns and Flach in \cite{Burns-Flach}. Proofs of this statement were also given independently by Bullach--Burns--Daoud--Seo \cite{Bullach-Burns-Daoud-Seo}, away from 2, and by Dasgupta--Kakde--Silliman \cite{Dasgupta-Kakde-Silliman-ETNC}, at all primes. Both of these proofs relied on a statement about the scarcity of Euler Systems. Our proof does not involve Euler Systems. Instead, it uses our main Theorem \ref{FUIMC} and Iwasawa co-descent, followed by an application of the method of Taylor--Wiles primes. As stated in \cite{Dasgupta-Kakde-Silliman-ETNC}, it is known (especially due to work of Burns) that ${\rm ETNC}(H/F)^-$ implies, for example, the minus parts of Rubin's integral refinement and Gross's $p$--adic refinement of Stark's conjecture for the extension $H/F$. (See Conjectures 3.2.9 and 5.4.1 in \cite{Popescu-Stark}.)
\\

{\bf Geometric motivation and analogies.} The main motivation for this paper comes from the recent work of Bley and Popescu \cite{bley-popescu}, where the authors proved an Equivariant Main Conjecture along any rank one, sign-normalized Drinfeld modular (geometric) Iwasawa tower of a general function field of characteristic $p$. For comparison purposes, we describe very briefly their results below, referring the reader to loc.cit. for more details, and draw some analogies with the set up of this paper.

Let $k$ be any function field of characteristic $p$ and let $\nu_{\infty}$ be a fixed place of $k$, which will be called the place at infinity from now on. Let $A\subseteq k$ be the ring of elements integral away from $\nu_{\infty}$. Also, fix an ideal $\mathfrak{f}$ and a maximal ideal $\mathfrak{p}$ in $A$, such that $\mathfrak{p}\nmid\mathfrak{f}$. Now, for each $n\geq 0$, define $L_n$ to be the ray-class field of $k$ of conductor $\mathfrak{fp}^{n+1}$ in which $\nu_{\infty}$ splits completely. The extension $L_n/L_0$ is
essentially generated by the $\mathfrak p$--power torsion points of a certain type of rank $1$, sign-normalized
Drinfeld module (a so--called Drinfeld--Hayes module) defined on $A$. So, we obtain a geometric (Drinfeld modular) Iwasawa tower $L_{\infty}/k$. 

Let $G_{n}=Gal(L_{n}/k)$,  $G_{\infty}=Gal(L_{\infty}/k)$ and $\mathbb{Z}_p[[G_{\infty}]]=\varprojlim \mathbb{Z}_p[[G_{n}]]$, where the transitionn maps in the projective system are induced by Galois restriction. One word of warning is in order here: the Iwasawa algebra $\Bbb Z_p[[G_\infty]]$ is not Noetherian in this context. In fact, one has an isomorphism of topological rings
$$\Bbb Z_p[[G_\infty]]\simeq \Bbb Z_p[H][[T_1, T_2, T_3, \dots]],$$
for some finite abelian group $H$ and countably many independent variables $T_1, T_2, \dots $.
\\

Now, consider two finite sets of places $S$ and $\Sigma$ of $k$, where
$$S:=\{\mathfrak{p}\}\cup \{\nu \mid  \nu\in{\rm MSpec}(A),\, \nu\, |\, \mathfrak{f}\}$$
and $\Sigma$ is nonempty and disjoint from $S$. Note that $S$ contains the ramification locus of $L_\infty/k$. For the data  $(L_{\infty}/k, S, \Sigma)$, building upon earlier work of Greither and Popescu \cite{Greither-Popescu-Picard}, Bley and Popescu constructed an ``Iwasawa--Ritter--Weiss type" $\mathbb{Z}_p[[G_{\infty}]]$--module $\nabla_S^{(\infty)}=\nabla_{S, \Sigma}^{(\infty)}$ (which is independent of $\Sigma$) of arithmetic interest. Further, they constructed a geometric version $\Theta_{S,\Sigma}^{(\infty)}(u)\in \mathbb{Z}_p[[G_{\infty}]][[u]]$ of the Greither--Popescu equivariant $p$--adic $L$--function and proved the following equivariant main conjecture type result.

\begin{theorem}[Bley-Popescu]
    For the data $(L_\infty/k,S, \Sigma)$, the $\mathbb{Z}_p[[G_{\infty}]]-$module $\nabla_S^{(\infty)}$ is finitely generated, torsion and the following equalities hold.
\begin{enumerate}
    \item ${\rm pd}_{\mathbb{Z}_p[[G_{\infty}]]}(\nabla_S^{(\infty)})=1$
    \item ${\rm Fitt}_{\mathbb{Z}_p[[G_{\infty}]]}(\nabla_S^{(\infty)})=(\Theta_{S,\Sigma}^{(\infty)}(1))  $
\end{enumerate}
\end{theorem}

Our main theorem can be viewed as a perfect, number field analogue of the above characteristic $p$ function field, non--Noetherian, geometric Iwasawa theoretic result. The difference is that in the number field setting the Iwasawa tower $H_\infty/H$ is obtained by adjoining the $p$--power torsion points of the group--scheme $\Bbb G_m$ (under the usual embedding $\Bbb Z\to {\rm End}(\Bbb G_m)$), while in the function field setting one obtains the Iwasawa tower $L_\infty/L$ by adjoining the $\mathfrak p$--power torsion points of the group--scheme $\Bbb G_a$ (under the embedding $A\to {\rm End}(\Bbb G_a)$ given by the Drinfeld--Hayes module in question.)
\\\\

The paper is organized as follows. In Chapter 2, we prove a slight generalization of a theorem of Greither and Kurihara regarding the behavior of Fitting ideals under projective limits. Then, we introduce Kataoka's shifted Fitting ideals. The relevant shifted Fitting ideal computations are done in the Appendix. 

In Chapter 3, we introduce the Ritter-Weiss modules, discuss their arithmetic significance and how they relate to other well known number theoretic objects. We also discuss how these modules behave in field extensions, which is important in our subsequent Iwasawa theoretic considerations.

In Chaper 4, we define our Iwasawa theoretic setup and prove the main theorem, using the material in Chapter 3, the link between the Ritter-Weiss module and the Burns--Kurihara--Sano Selmer modules and the related $p$--adic Burns--Kurihara--Sano Conjecture, proved in \cite{gambheera-popescu}. In the last part of Chapter 4, we use the main theorem to give our first application with the help of computations done in the Appendix. 

Finally, in Chapter 5, we give the second application of our main theorem: a new proof of the minus part ${\rm ETNC}(H/F)^-$, away from $2$, of the Equivariant Tamagawa Number Conjecture ${\rm ETNC}(H/F)$ for the Artin motive associated to $H/F$, with coefficients in $\Bbb Z[Gal(H/F)]$.

\section{Algebraic Preliminaries}

\subsection{Fitting ideals and projective limits}\label{Fitt-proj-section}

In this section, we prove a slight generalization of a result of Greither and Kurihara (see Theorem 2.1 in \cite{Greither-Kurihara}) on the behavior of Fitting ideals under projective limits. For further generalizations of this useful result, even in the non--Noetherian setting, the reader can consult the recent work of Popescu--Yin \cite{Popescu-Yin}. For the basic definitions and properties of Fitting ideals, the reader can consult \cite{Greither-Kurihara}, \cite{Kataoka} and \cite{Greither-Kataoka-Kurihara}.

The following result (Theorem 2.1 in \cite{Greither-Kurihara})  asserts that a certain compatibility between projective limits and Fitting ideals holds in certain Iwasawa theoretic contexts.

\begin{theorem}[Greither-Kurihara]\label{Greither-Kuri}
    Let $\Lambda:=\mathcal{O}[[T]]$, where $\mathcal{O}$ is the ring of integers of a finite extension of $\mathbb{Q}_p$. Let $R:=\Lambda[G]$, where $G$ is a finite, abelian $p-$group and let $R_n:=R/((1+T)^{p^n}-1)R$. Assume that $(A_n)_n$ is a projective system of modules over the projective system of rings $(R_n)_n$ in the obvious sense, satisfying the following two properties.
\begin{enumerate}
\item The transition maps of $(A_n)_n$ are surjective for all $n\gg 0$.\\
\item  $A:=\underset{n}{\varprojlim}\text{  } A_n$ is a finitely generated, torsion module over $\Lambda$.
\end{enumerate}
Then, we have an equality of $R$--ideals ${\rm Fitt}_R(A)=\underset{n}{\varprojlim}\text{  }{\rm Fitt}_{R_n}A_n$, under the usual identification $R\simeq\varprojlim\limits_n R_n.$ 

\end{theorem}

Now, we state and sketch the proof of a slight generalization of the above Theorem and prove a corollary of the more general result which will be used in Chapter 5. First, we make a remark which clarifies why the next theorem is a generalization of the previous one.

\begin{remark}
Let us observe that when $G$ is a $p-$group, $\Lambda[G]$ is a local ring with the unique maximal ideal $(\pi, T, g-1; g\in G)$ where $\pi$ is a uniformizer of $\mathcal{O}$. Since $\Lambda$ has Krull dimension 2, if $f\in \Lambda$ is nonzero, the Krull dimension of $(\Lambda/f)$ is 1, hence so is that of $(\Lambda/f)[G]\cong \Lambda[G]/(f)$, which is a finite extension of $\Lambda/f$. 
\end{remark}

\begin{theorem}\label{GK-generalization}
Let $(A_n)_n$ be a projective system of modules over the projective system of compact local rings $(R_n)_n$ in the obvious sense, such that the following properties are satisfied.
\begin{enumerate}
\item The transition maps $\pi_n : A_{n+1}\xrightarrow{} A_n$ and $\pi_n': R_{n+1}\xrightarrow[]{} R_n$ are surjective for all $n\gg 0$.
\item $A:=\underset{n}{\varprojlim}\text{  } A_n$ is finitely generated over $R:=\underset{n}{\varprojlim}\text{  }R_n$.
\item There exists $f\in R$ such that $f\cdot A=0$ and $R/fR$ is local Noetherian, of Krull dimension at most 1.
\end{enumerate}
Then, we have an equality of $R$--ideals ${\rm Fitt}_R(A)=\varprojlim\limits_n{\rm Fitt}_{R_n}A_n$.
\end{theorem}
\begin{proof}(Sketch)
In loc.cit., Greither and Kurihara proved their theorem in 8 steps. 
All steps, except for step 5 can be translated identically to the more general situation described above. Certain changes need to be made in completing step 5 and this is what we will describe in detail below.

We remind the reader that in Steps 1--4, in loc.cit. it is proved that there exists $m\in\Bbb Z_{\geq 0}$ such that one can construct a projective system of exact sequences of $R_n$--modules
$$\{0\to B_n\to R_n^m\to A_n\to 0\}_{n>>0},$$
whose transition maps are compatible with the original transition maps at the $A_n$ and $R_n$ levels. Step 5 asks for  the existence of $r\in\mathbb{N}$, such that $B_n$ is generated by $r$ elements over $R_n$, for all $n\gg 0$. We will give a detailed proof of this statement below.

Since $f\cdot A=0$, we have $f\cdot R_n^m\subseteq B_n \subseteq R_n^m$. So, if we can find some $r_0$ such that $B_n':=B_n/(f\cdot R_n^m)$ can be generated by $r_0$ elements, for all $n\gg0$, we can set  $r:=r_0+m$ and we are done. Now, for $n>>0$, $B_n'$ is an $R$--submodule of $(R_n/fR_n)^m$. Let $B_n''$ be the preimage of $B_n'$ in $(R/fR)^m$ via the canonical projection
$$\pi_n: (R/fR)^m\twoheadrightarrow (R_n/fR_n)^m.$$
Since $\pi_n$ is surjective, it suffices to show that $B_n''$ is generated by $r_0$ elements over $R/fR$, for some $r_0$ and all $n\gg0$. However, hypothesis (3) in the theorem combined with \cite{Shalev} implies the existence of a constant $d$ such that all ideals of $R/fR$ can be generated by $d$ elements. By an easy argument (induction on $m$), this shows that all submodules of $(R/fR)^m$ can be generated by $r_0:=md$ elements. This completes the proof of Step 5 and hence that of the theorem. 
\end{proof}
The concrete Iwasawa theoretic context in which we will apply the above theorem is the following.
We let $H/F$ be an extension of number fields. Let $H_{\infty}/H$ be a $\mathbb{Z}_p$-extension, for some prime $p$ and we let $H_n$ denote its $n$--th layer, for all $n$. Suppose that $H_n/F$ is Galois, abelian and let $G_n:=Gal(H_n/F)$, for all $n$. Therefore, for each $n$, $G_n\cong G'\times G_{p,n}$ where $G_{p,n}$ is the Sylow $p$ subgroup of $G_n$ and $G'$ is its maximal subgroup of order not divisible by $p$. Note that $G'$ does not depend on $n$ and it is isomorphic to the maximal subgroup of $Gal(H/F)$ of order not divisible by $p$. Then, we obtain topological group isomorphisms 
$$\mathcal{G}:=Gal(H_{\infty}/F):=\underset{n}{\varprojlim}\text{ }G_n\cong G'\times G_{p,\infty}, \quad\text{where }  G_{p,\infty}:=\underset{n}{\varprojlim}\text{ }G_{p,n}.$$
Let $G_{[p]}$ be the torsion subgroup of $G_{p,\infty}$. Since $G_{p,\infty}$ is a finitely generated $\Bbb Z_p$ module of rank one, there is a subgroup $\Gamma$ of $G_{p,\infty}$, such that 
\begin{equation}\label{group-iso} G_{p,\infty}=G_{[p]}\times\Gamma, \qquad\Gamma\cong \Bbb Z_p, \qquad\mathcal G\cong G'\times G_{[p]}\times \Gamma.\end{equation}\\

Let $\widehat{G'}:={\rm Hom}(G', \overline{\Bbb Q_p}^\times)$. As usual, we define an equivalence relation $\sim$ on $\widehat{G'}$ by 
$$\chi\sim \chi', \quad\text{ if there exists }\sigma\in G_{\Bbb Q_p}, \text{ such that } \chi'=\sigma\circ\chi,$$
and denote by $[\chi]$ the equivalence class of $\chi$. Further, for any character $\chi\in\widehat{G'}$, define the rings $$R^{\chi}_n:=\mathbb{Z}_p[\chi][G_{p,n}], \quad R^{\chi}_{\infty}:=\mathbb{Z}_p[\chi][[G_{p,\infty}]]\cong \mathbb{Z}_p[\chi][G_{[p]}][[\Gamma]],$$
where $\Bbb Z_p[\chi]$ is the ring extension of $\Bbb Z_p$ generated by the values of $\chi$ in $\overline {\Bbb Q_p}.$ It is well known that if we let $[\widehat {G'}]:=\widehat{G'}/\sim$ denote the set of equivalence classes, then  we obtain isomorphisms of $\Bbb Z_p$--algebras
$$R_n:=\mathbb{Z}_p[G_n]\cong \bigoplus_{[\chi]\in[\widehat{G'}]}R^{\chi}_n, \qquad R_\infty:=\mathbb{Z}_p[[\mathcal{G}]]\cong \bigoplus_{[\chi]\in[\widehat{G'}]}R^{\chi}_{\infty}.$$ 
Correspondingly,  if $A_n$ is a $\Bbb Z_p[G_n]$--module and $A$ is a $\mathbb{Z}_p[[\mathcal{G}]]$--module, then we get direct sum decompositions of $\Bbb Z_p[G_n]$--modules and  $\Bbb Z_p[[\mathcal G]]$--modules, respectively
$$A_n\cong \bigoplus_{[\chi]\in[\hat{G'}]}A_n^\chi, \quad A\cong \bigoplus_{[\chi]\in[\hat{G'}]}A^\chi, \quad \text{ where }\quad  A_n^\chi:=A_n\otimes_{R_n}R_{n}^\chi,\quad A^\chi:=A\otimes_{R_\infty}R_{\infty}^\chi,$$
and an arbitrary $\sigma'\in G'$ acts on $A_n^\chi$ and $A^\chi$ via multiplication by $\chi(\sigma')$. \\

In this context, the Theorem above has the following consequence.

\begin{corollary}\label{GK}
With  notations as above, let $(A_n)_n$ be a projective system of modules over the projective system of rings $(\mathbb{Z}_p[G_n])_n$, such that the following conditions hold.
\begin{enumerate}
    \item The transition maps $\pi_n : A_{n+1}\xrightarrow{} A_n$ are surjective for $n\gg 0$.\\
\item $A:=\underset{n}{\varprojlim}\text{  } A_n$ is finitely generated and torsion over $\Lambda:=\mathbb{Z}_p[[\Gamma]]$.
\end{enumerate}
Then, we have an equaliy of $\Bbb Z_p[[\mathcal G]]$--ideals ${\rm Fitt}_{\mathbb{Z}_p[[\mathcal{G}]]}(A)=\underset{n}{\varprojlim}\text{  }{\rm Fitt}_{\mathbb{Z}_p[[G_n]]}A_n$.
\end{corollary}
\begin{proof}
Fix an arbitrary character $\chi\in\widehat{G'}$. Note that, for $n\gg 0$, the maps $\pi_n^\chi:A_{n+1}^\chi\to A_n^\chi$, induced by $\pi_n$ are surjective. The same is true for the ring morphisms  $\pi_n^{\ast, \chi}:R_{n+1}^\chi\to R_n^\chi$, induced by Galois restriction. Also, note that since $A^\chi=\varprojlim_n A_n^\chi$
is a direct summand of $A$, it is finitely generated and torsion over the Noetherian, integral domains 
$$\Lambda\simeq\Bbb Z_p[[T]], \qquad \Lambda(\chi)\simeq  Z_p[\chi][[T]].$$ Moreover, if one picks $f\in\Lambda\setminus\{0\}$, such that $f\cdot A=0$, then $f\cdot A^\chi=0$. Also,  $R_\infty^\chi/f\simeq (R_\infty^\chi/f)$ is a local, Noetherian ring of Krull dimension at most $1$, according to the remark above and the observation that
$$R_\infty^\chi\simeq \Bbb Z_p[\chi][[T]][G_{[p]}].$$ 
Consequently, one can apply the Theorem above to the projective systems of modules and rings $(A_n^\chi)_n$ and $(R_n^\chi)_\chi$ to obtain an equality of $R_\infty^\chi$--ideals
$${\rm Fitt}_{R_\infty^\chi}(A^\chi)=\underset{n}{\varprojlim}\text{  }{\rm Fitt}_{R_n^\chi}A_n^\chi,$$
for all $\chi\in\widehat{G'}$. Now, the desired equality of $R_\infty$--ideals is obtained by taking the direct sum of the above equalities, for all $[\chi]\in[\widehat{G'}]$.
\end{proof}

\begin{remark} Note that the standard ring isomorphism $R_\infty^\chi\simeq\Bbb Z_p[\chi][[T]][G_{[p]}]$, taking a topological generator $\gamma$ of $\Gamma$ to $(1+T)$, does not, in general, induce ring isomorphisms $R_n^\chi\simeq \Bbb Z_p[\chi][[T]]/((1+T)^{p^n}-1)[G_{[p]}]$. This is the main reason why the original Greither--Kurihara Theorem is not applicable to the general Iwasawa theoretic context described above and a generalization of that result was needed. Such isomorphisms would be induced if, for example, $\mathcal G\simeq\Gamma\times G$, which is a very special situation.  
\end{remark}

We end this subsection with the following technical lemma which is to be used in Section 5.

\begin{lemma}\label{unit-lifting-lemma}
Assume that we have a surjective group homomorphism of abelian groups $\phi:G_2\xrightarrow[]{} G_1$, whose kernel is a $p-$group. Then, the induced group morphism $\phi':\mathbb{Z}_p[G_2]^{\times}\xrightarrow[]{}\mathbb{Z}_p[G_1]^{\times}$ is surjective.
\end{lemma}
\begin{proof}
    Since $ker(\phi)$ is a $p-$group, $G_i\simeq G_{i,p}\times G'$, for some group $G'$ of order not divisible by $p$ and $p$--groups $G_{i,p}$, for $i=1,2$. Then, we have the following well known isomorphisms of rings, for $i=1,2$.
    $$\mathbb{Z}_p[G_i]\cong \bigoplus_{[\chi]\in[\widehat{G'}]}\mathbb{Z}_p[\chi][G_{i,p}] $$
    Here $G_{i,p}$ is the Sylow $p$ subgroup of $G_i$. This induces the following isomorphism of groups. 
     $$\mathbb{Z}_p[G_i]^{\times}\cong \bigoplus_{[\chi]\in[\widehat{G'}]}\mathbb{Z}_p[\chi][G_{i,p}]^{\times}. $$
     Under the above isomorphisms,  $\phi'=\oplus_\chi\phi_\chi'$, where $\phi_{\chi}':\mathbb{Z}_p[\chi][G_{2,p}]^{\times}\xrightarrow[]{}\mathbb{Z}_p[\chi][G_{1,p}]^{\times}$ is induced by $\phi$ restricted to $G_{2,p}$, which is surjective. Hence, it is enough to show that $\phi_{\chi}'$ is surjective, for each $\chi$. Now, for $i=1,2$ observe that $\mathbb{Z}_p[\chi][G_{i,p}]$ is a local ring with residue field is $\mathbb{Z}_p[\chi]/m_{\chi}$ where $m_{\chi}$ is the maximal ideal of the local field $\mathbb{Z}_p[\chi]$. So, we have the following commutative diagram.

\[
\begin{tikzcd}
  \mathbb{Z}_p[\chi][G_{2,p}]\arrow{r}{\psi_2}\arrow{d}{\phi_{\chi}} &  \mathbb{Z}_p[\chi]/m_{\chi}\arrow[equal]{d} \\%
   \mathbb{Z}_p[\chi][G_{1,p}]\arrow{r}{\psi_1} & \mathbb{Z}_p[\chi]/m_{\chi}
\end{tikzcd}
\]
Here $\phi_{\chi}$ is the obvious map induced by $\phi$ and the horizontal maps are induced by the augmentation maps. Let $u\in \mathbb{Z}_p[\chi][G_{1,p}]^{\times}$ and $\Tilde{u}\in \mathbb{Z}_p[\chi][G_{2,p}]$ be a lift of $u$. Then, $\psi_1(u)\neq 0$. Therefore, we should have $\psi_2(\Tilde{u})\neq 0$. Hence, $\Tilde{u}\in \mathbb{Z}_p[\chi][G_{2,p}]^{\times}$. This proves the surjectivity of $\phi'_{\chi}$ as desired.
\end{proof}

\subsection{Shifted Fitting ideals}

In this section, we define shifted Fitting ideals, introduced recently by Kataoka in \cite{Kataoka} and studied further by Greither, Kataoka and Kurihara in \cite{Greither-Kataoka-Kurihara}. The concrete computations of shifted Fitting ideals needed in this paper can be found in the Appendix. 

Before we begin, we want to remind the reader that if $R$ is a commutative ring and $M$ is an $R$--module, then $M$ is said to be torsion if for every $m\in M$ there exists a {\it non zero--divisor $r\in R$,} such that $r\cdot m=0.$ Equivalently, if $Q(R)$ denotes the total ring of fractions of $R$, then $Q(R)\otimes_R M=0$.

Also, note that if $M$ is torsion and ${\rm pd}_R(M)\leq 1$, then ${\rm pd}_R(M)=1$, unless $M=0$. For such a module $M$ (torsion, of ${\rm pd}_R=1$) which is also finitely presented, it is not very difficult to see that its $0$--th Fitting ideal ${\rm Fitt}_R(M)$ is a projective $R$--module of rank $1$, i.e. an invertible fractional ideal in $Q(R)$. Indeed, this is trivially true if $M=0$, because ${\rm Fitt}_R(0)=R$. If $M\ne 0$, it is not hard to see (from the definitions) that if 
$$0\to P_2\to P_1\to M\to 0$$
is a resolution of $M$, with $P_1$ and $P_2$ finitely generated and projective, then there is an $R$--module isomorphism
\begin{equation}\label{Fitt-formula}{\rm Fitt}_R(M)\simeq {\rm det}_R(P_1)^{(-1)}\otimes_R{\rm det_R}(P_2),\end{equation}
which shows that ${\rm Fitt}_R(M)$ is a projective $R$--module of rank $1$.
\\

The following result (see Theorem 1.7 in \cite{Greither-Kataoka-Kurihara}) leads to the definition of the $n$--th shifted Fitting ideal of a finitely presented, torsion module over a commutative ring, independent of any choices.

\begin{theorem}[Kataoka]
  Let $M$ be a finitely presented torsion module over the commutative ring $R$. Take a resolution in the category of $R$--modules
  $$0\xrightarrow[]{} N \xrightarrow[]{} X_1\xrightarrow[]{} ... \xrightarrow[]{} X_n \xrightarrow[]{} M \xrightarrow[]{} 0$$
where all the modules are finitely presented and torsion over $R$ and ${\rm pd}_R(X_i)\leq 1$, for all $i=1, \dots, n$. If we define the $n$--th shifted ideal of $M$ over $R$ to be the fractional ideal 
$${\rm Fitt}_R^{[n]}(M):=(\prod_{i=1}^{n}{\rm Fitt}_R(X_i)^{(-1)^i})\cdot {\rm Fitt}_R(N),$$
then this definition is independent of the choice of resolution.
\end{theorem}

The following corollary is a computationally useful direct consequence of the above theorem.

\begin{corollary}\label{Shifted-Fitting}
    A short exact sequence of finitely presented, torsion modules over a commutative ring $R$
    $$0\xrightarrow[]{} M \xrightarrow[]{} M'\xrightarrow[]{} M''\xrightarrow[]{}0,$$
such that $pd_R(M')\leq 1$, leads to an equality of fractional $R$--ideals
$${\rm Fitt}_R(M)={\rm Fitt}_R(M')\cdot {\rm Fitt}_R^{[1]}(M'').$$
\end{corollary}

\subsection{Quadratically presented modules and their associated transposed modules.} An important role in what follows (as it was the case with \cite{Dasgupta-Kakde} and \cite{gambheera-popescu}) will be played by what we now call quadratically presented modules. In this section, we remind the reader their definition and some of their basic properties.

\begin{definition}\label{def-quad-pres} A module $M$ over a commutative ring $R$ is called quadratically presented if there exists an exact sequence of $R$--modules
$$R^k\xrightarrow[]{\theta} R^k\xrightarrow[]{}M \xrightarrow[]{}0,$$
for some $k\geq 1$. In this case, $k$ is called the rank of the quadratic presentation.
\end{definition}
\noindent Note that if $M$ has a quadratic presentation as in the definition above, then its Fitting ideal is principal, generated by the determinant of $\theta$.
$${\rm Fitt}_R(M)=\det\theta\cdot R.$$
\medskip

A slightly weaker notion is that of ``locally quadratically presented'' modules.
\begin{definition}\label{def-loc-quad-pres}
A module $M$ over a commutative ring $R$ is called locally quadratically presented if there exists an exact sequence of $R$--modules 
$$P_1\xrightarrow[]{\rho} P_0\xrightarrow[]{}M \xrightarrow[]{}0,$$
with $P_0$ and $P_1$ projective, finitely generated $R$--modules of equal local ranks, meaning that for all maximal ideals $\mathfrak m\in\rm{MSpec}(R)$ (and therefore, for all prime ideals $\mathfrak m$) we have  ${\rm rk}_{R_{\mathfrak m}}(P_0)_{\mathfrak m}={\rm rk}_{R_{\mathfrak m}}(P_1)_{\mathfrak m}$.
\end{definition}
\noindent Note that if $M$ has a local quadratic presentation as in the last definition, then its Fitting ideal is locally principal, namely
$${\rm Fitt}_R(M)_{\mathfrak m}=\det(\rho_{\mathfrak m})\cdot R_{\mathfrak m}, \qquad\text{ for all }\mathfrak m\in{\rm Spec}(R).$$
Indeed, this is obtained by localizing at $\mathfrak m$ the presentation in the last definition and observing that the result is a quadratic presentation of $M_{\mathfrak m}$ over $R_{\mathfrak m}$. Also, recall that Fitting ideals commute with localization, or any extension of scalars, for that matter.
\medskip

Next, we remind the reader the notion of a transpose of a module, due to Auslander. (See \cite{Jannsen} for more details.) In what follows, if $N$ is a module over a commutative ring $R$, we let $$N^\ast:={\rm Hom}_R(N, R),$$
be its dual, endowed with the usual $R$--module structure.

\begin{definition}\label{def-transpose}
    Let $M$ be a module over a commutative ring $R$, endowed with a ``projective presentation''
    $$P_1\xrightarrow[]{\phi} P_2 \xrightarrow[]{} M \xrightarrow[]{}0.$$
This means that $P_1$ and $P_2$ are projective $R$--modules and the sequence above is exact in the category of $R$--modules. Then, the transpose of $M$ with respect to the given presentation is given by
$$M^{tr}={\rm coker}(P_2^* \xrightarrow[]{\phi^*} P_1^*),$$
where $\phi^\ast$ is the map at the level of duals naturally induced by $\phi.$
\end{definition}

Observe that $M^{tr}$ depends on the choice of projective presentation. (It is in fact known that, up to projective equivalence, it does not depend on any choices.) However, there are instances where the Fitting ideal of $M^{tr}$ happens to coincide with that of $M$ and therefore it is independent of the choice of presentation, as explained in the next result.

\begin{proposition}\label{Fitt-transp}
    Let $M$ be a locally quadratically presented $R-$module. Then, for the transpose $M^{tr}$ of $M$ with respect to any locally quadratic presentation, 
    we have
    $${\rm Fitt}_R(M^{tr})={\rm Fitt}_R(M).$$
\end{proposition}
\begin{proof} Pick any locally quadratic presentation for $M$ as in Definition \ref{def-loc-quad-pres} and let $$M^{tr}={\rm coker}(P_0^\ast\xrightarrow[]{\rho^\ast}P_1^\ast).$$
Then, $P_i^\ast$ is a projective $R$--modules, of the same local ranks as $P_i$, for all $i=1,2$. Therefore $M^{tr}$ is itself locally quadratically presented and we have equalities of $R_{\mathfrak m}$--ideals, for all $\mathfrak m\in{\rm Spec}(R)$:
\[
\begin{aligned}
{\rm Fitt}_R(M^{tr})_{\mathfrak m}={\rm Fitt}_{R_{\mathfrak m}}((M^{tr})_{\mathfrak m})&={\rm det}(\rho^\ast_{\mathfrak m})R_{\mathfrak m}\\
&={\rm det}(\rho_{\mathfrak m})R_{\mathfrak m}={\rm Fitt}_{R_{\mathfrak m}}(M_{\mathfrak m})={\rm Fitt}_R(M)_{\mathfrak m}.
\end{aligned}
\]
The first and last equalities above follow from the fact that Fitting ideals commute with localization. The third equality follows from the fact that taking duals and localizing commute and, consequently, the square matrices associated to $\rho_{\mathfrak m}^\ast$ and $\rho_{\mathfrak m}$ are transposed of one another.
The equality between the first and last term above shows that the ideals ${\rm Fitt}_R(M^{tr})$ and ${\rm Fitt}_R(M)$ are indeed equal in $R$.
\end{proof}

The following is a slightly modified version of Lemma 21 in \cite{Dasgupta-Kakde-Silliman-ETNC}, providing a situation where quadratically presented modules over certain quotients of a group ring $\Bbb Z[G]$ can occur. We provide a slightly more conceptual proof than the one in loc.cit. In what follows, if $M$ is a $\Bbb Z[G]$--module and $R$ is a $\Bbb Z[G]$--algebra, then $M_R:=M\otimes_{\Bbb Z[G]}R$. The same notational convention applies to morphisms of $\Bbb Z[G]$--modules.

\begin{lemma} Let $G$ be a finite, abelian group and let $R$ be a quotient ring of $\Bbb Z[G]$ which is $\Bbb Z$--free. Let
$$P\overset{\iota}\longrightarrow F\longrightarrow X\longrightarrow 0$$
be an exact sequence of $\Bbb Z[G]$--modules, with $F$ free of rank $n$ and $P$ projective of rank $n$. If ${\rm Fitt}_R(X_R)=xR$, where $x$ is not a zero--divisor in $R$, the following hold.
\begin{enumerate}
\item $\iota_R$ is injective and therefore ${\rm pd}_R(X_R)=1$.
\item $P_R$ is $R$--free of rank $n$ and therefore $X_R$ is quadratically presented. 
\end{enumerate}
\end{lemma}
\begin{proof} First, note that, as a quotient of $\Bbb Q[G]$, the ring $R_{\Bbb Q}:=R\otimes_{\Bbb Z}\Bbb Q$ is a finite, non--empty direct sum of fields (cyclotomic extensions of $\Bbb Q$). Also, ${\rm Fitt}_{R_{\Bbb Q}}(X_{R_{\Bbb Q}})=xR_{\Bbb Q}$ and, since $R$ has no $\Bbb Z$--torsion,  $x$ remains a non zero--divisor in $R_{\Bbb Q}$. These facts, combined with the equality ${\rm rk}_{R_{\Bbb Q}}(P_{R_{\Bbb Q}})={\rm rk}_{R_{\Bbb Q}}(F_{R_{\Bbb Q}})=n$, imply that 
$$\ker(i_R)\otimes_{\Bbb Z}\Bbb Q=0,$$
which shows that $\ker(i_R)=0$, as $\ker(i_R)$ (as a submodule of a projective $R$--module) has no $\Bbb Z$--torsion.

Second, note that the injectivity of $\iota_R$ and the $R$--freeness of $F_R$ give isomorphisms of $R$--modules
$$R\simeq {\rm Fitt}_R(X_R)\simeq {\rm det}_R(P_R)\otimes_R{\rm det}_R(F_R)^{(-1)}\simeq {\rm det}_R(P_R). $$
(See \eqref{Fitt-formula} for the second isomorphism above.) On the other hand, a well--known theorem of Swan (see \cite{Swan}) characterizing finitely generated, projective $\Bbb Z[G]$--modules shows that we have a $\Bbb Z[G]$--module isomorphism 
$$P\simeq \Bbb Z[G]^{n-1}\oplus \mathfrak A,$$
where $\mathfrak A$ is an invertible ideal of $\Bbb Z[G]$. Consequently, we obtain isomorphisms of $R$--modules
$$R\simeq {\rm det}_R(P_R)\simeq  \mathfrak A_R,$$
which shows that $\mathfrak A_R$ is $R$--free of rank $1$ and therefore $P_R$ is $R$-free of rank $n$.
\end{proof}

\section{The Ritter-Weiss modules}

\subsection{Definitions and Main Properties}
\setcounter{equation}{0}
In this section we define the Ritter-Weiss modules via their locally quadratic presentations. 
First, we recall some notations and definitions from \cite{Dasgupta-Kakde}. However, the reader should also consult \S5 of \cite{gambheera-popescu}, as the notation used below is identical to that in loc.cit. Let $H/F$ be an abelian extension of number fields with Galois group $G$.  We denote the set of all places in $F$ which are ramified in $H/F$ by $S_{ram}(H/F)$. Also, $S_\infty$ denotes the set of archimedean places in $F$. We fix two disjoint, finite sets of places $S$ and $T$ of $F$, such that $S_\infty\subseteq S$ and define
$$\mathcal O_{H,S,T}^{\times}:= \{x\in H^\times ; {\rm ord}_w(x)= 0, \text{ for all }w\not\in S_H, \quad {\rm ord}_{w}(x-1)>0, \text{ for all } w\in T_H\},$$
where for a finite prime $w$ in $H$, we let ${\rm ord}_w$ be the normalized valuation of $H$ associated to $w$. Here, $S_H$ and $T_H$ are the sets of places in $H$ sitting above places in $S$ and $T$, respectively. Further, we define
$$H_T^{\times}:= \{x\in H^{\times} ; {\rm ord}_{w}(x-1)>0,\ \text{for all } w\in T_H\}.$$

Now, the $(S,T)-$ray class group $Cl_S^T(H)$ of $H$ is the finite, abelian group defined by the following exact sequence of $\Bbb Z[G]$--modules:
\[
 \begin{tikzcd}
0\arrow{r} &\mathcal O_{H, S,T}^\times\arrow{r} & H_T^{\times}\arrow{r}\arrow{r}{{\text div}_{\overline{S\cup T}}} & Y_{\overline{S\cup T}}(H)\arrow{r} & Cl^T_S(H)\arrow{r} &  0.
\end{tikzcd}
\]
where $Y_{\overline{S\cup T}}(H):=\bigoplus_{w\notin S\cup T}\mathbb{Z}\cdot w$ is the free $\mathbb{Z}-$module of divisors supported at the places of $H$ outside $S_H\cup T_H$, the map ${\text div}_{\overline{S\cup T}}(\ast):=\sum_{w\notin S_H\cup T_H}{\rm ord}_w(\ast)\cdot w$ is the usual $(S_H\cup T_H)$--depleted divisor map and the right--most non--zero map is the divisor--class map.

Observe now that if we let $M^\ast={\rm Hom}_{\Bbb Z}(M, \Bbb Z)$ denote the $\Bbb Z$--dual of a $\Bbb Z[G]$--module, viewed as a $\Bbb Z[G]$--module with the usual co--variant $G$--action, then the map ${\text div}_{\overline{S\cup T}}$
induces an injective morphism of $\Bbb Z[G]$--modules at the level of $\Bbb Z$--duals:
$${\text div}_{\overline{S\cup T}}^\ast: Y_{\overline{S\cup T}}(H)^\ast\to (H_T^\times)^\ast.$$
Now, it is easy to define first the Selmer modules introduced by Burns--Kurihara--Sano in \cite{Burns-Kurihara-Sano} and studied extensively in \cite{Dasgupta-Kakde} and \cite{gambheera-popescu}.
\begin{definition}\label{sel-def} 
    The Selmer $\Bbb Z[G]$--module ${\rm Sel}_S^T(H)$ for the data $(H/F, S, T)$ is given by
$${\rm Sel}_S^T(H):= (H_T^{\times})^\ast/{\rm Im}({\rm div}_{\overline{S\cup T}}^\ast)\,\cong\,(H_T^{\times})^\ast/Y_{\overline{S\cup T}}(H)^\ast.$$
\end{definition}
\bigskip

Next, we let $S$, $S'$ and $T$ be finite sets of places in $F$, satisfying the following properties. 
\\

\begin{itemize}
\item[] \underline{\textbf{Properties $P(H/F, S, T, S')$}}
\item $S_{\infty}\subseteq S$.
  \item $S\subseteq S'$ and $S'\cap T=\emptyset.$
  \item $S_{ram}(H/F)\subseteq S\cup T$.
  \item $T\not\subseteq S_{ram}(H/F)$ and $H_T^\times$ is torsion free.
  \item  $Cl_{S'}^T(H)=1$   
  \item $\bigcup_{w\in S'_H} G_w=G$, where $G_w$ is the decomposition group of $w$ in the extension $H/F$.
\end{itemize}
\medskip

For every place $v$ of $F$, we fix a place $w$ of $H$, sitting above $v$. We let $H_w$ denote the completion of $H$ at $w$ and $G_w$ the decomposition subgroup of $w$ in $G$. If $v$ is non--archimedean, $\mathcal O_w$ is the ring of integers in $H_w$, $U_w$ is the subgroup of principal units in $\mathcal O_w^\times$, and $H_w^{ab}$ is a fixed maximal abelian extension of $H_w$. 
\\

If $K$ is a local or global field, $W(K)$ denotes its Weil group and $W(K)^c$ denotes the topological closure of the commutator subgroup of $W(K)$. If $E/K$ is a Galois extension of local or global fields, we let $W(K^{ab}/E):=W(K)/W(E)^c$. Recall that in the global and local case, respectively, we have canonical topological group isomorphisms $C_K\simeq W(K^{ab}/K)$ and $K^\times \simeq W(K^{ab}/K)$ which, when composed with the canonical map $W(K^{ab}/K)\to G_K^{ab}$, give the local and global reciprocity maps, respectively. Here, as usual, $C_K$ denotes the id\`ele class group of the global field $K$.
\\

On the local side, following Ritter and Weiss \cite{Ritter-Weiss}, we start with the obvious exact sequence of Weil groups:
\begin{equation}\label{Weil-SES}
 0\xrightarrow[]{} W(H_w^{ab}/H_w)\cong H_w^{\times} \xrightarrow[]{} W(H_w^{ab}/F_v) \xrightarrow[]{} G_w\xrightarrow[]{} 0,
 \end{equation}
to which we apply the translation functor $t$ (an equivalence of categories) of Proposition 1 in loc.cit. and obtain an exact sequence of $\Bbb Z[G_w]$--modules 
\begin{equation}\label{Weil-SES-t} 0\xrightarrow[]{} H_w^{\times}\xrightarrow[]{}V_w\xrightarrow[]{}\Delta G_w\xrightarrow[]{}0,\end{equation}
where $\Delta G_w$ is the augmentation ideal of $\mathbb{Z}[G_w]$.
\begin{remark} It is important to note that since the class $u_{H_w/F_v}\in H^2(G_w, H_w^\times)$ of exact sequence \eqref{Weil-SES} above is the local fundamental class, by the properties of the translation functor $t$ in loc.cit., the class $\alpha_w\in{Ext}^1_{\Bbb Z[G_w]}(\Delta G_w, H_w^\times)$ of exact sequence \eqref{Weil-SES-t} is the unique class which maps to $u_{H_w/F_v}$ via the obvious local coboundary isomorphism 
$${\rm Ext}_{\Bbb Z[G_w]}^1(\Delta G_w,H_w^\times)=H^1(G_w, {\rm Hom}_{\Bbb Z}(\Delta G_w, H_w^\times))\overset {\delta_w^1}\longrightarrow H^2(G_w, H_w^\times).$$
\end{remark}

Next, we define the $\Bbb Z[G_w]$--module $W_w$, in the case when $v$ is non--archimedean, via the following commutative diagram, whose rows are short exact sequences of $\Bbb Z[G_w]$--modules.
\[
\begin{tikzcd}
0\arrow{r} & O_w^{\times}\arrow{r}\arrow[hook]{d} & V_w \arrow{r}\arrow[equal]{d} & W_w \arrow{r}\arrow[two heads]{d}{j} & 0 \\%
0\arrow{r} & H_w^\times \arrow{r}\arrow[two heads]{d}{ord_w} & V_w \arrow{r} & \Delta G_w\arrow{r} & 0 \\%
& \mathbb{Z} 
\end{tikzcd}
\]
When applying the snake lemma to the diagram above, we obtain the following short exact sequence.
\begin{equation}\label{ZWG-SES}
0\xrightarrow[]{}\mathbb{Z}\xrightarrow[]{i}W_w\xrightarrow[]{j}\Delta G_w\xrightarrow[]{}0
\end{equation}
\\

On the global side, we start with the short exact sequence of global Weil groups
\begin{equation}\label{Weil-SES-global}
 0\xrightarrow[]{} W(H^{ab}/H)\cong C_H \xrightarrow[]{} W(H^{ab}/F) \xrightarrow[]{} G\xrightarrow[]{} 0,
 \end{equation}
to which we apply the functor $t$ in loc.cit. to obtain an exact sequence of $G$--modules
\begin{equation}\label{Weil-SES-t-global} 0\xrightarrow[]{}C_H\xrightarrow[]{}D\xrightarrow[]{}\Delta G\xrightarrow[]{}0.\end{equation}
As in the remark above, the class $\alpha\in {\rm Ext}_G^1(\Delta G,C_H)$ of the second exact sequence above is unique with the property that its image via the coboundary isomorphism  
$${\rm Ext}_G^1(\Delta G,C_H)=H^1(G, {\rm Hom}_{\Bbb Z}(\Delta G,C_H))\overset {\delta^1}\longrightarrow H^2(G,C_H)$$
is the global fundamental class $u_{H/F}\in H^2(G,C_H)$, which is the class of \eqref{Weil-SES-global}. 
\\\\

In what follows, if $X$ is a set of primes in $F$ and $M_w$ is a $\Bbb Z[G_w]$--module, for the fixed $w$ in $H$ sitting above $v$ in $X$, we define the $\Bbb Z[G]$--module:
\[\prod_{v\in X}^{\sim} M_w:=\prod_{v\in X}^{} {\rm Ind}_{G_w}^GM_w.
\]
With this notation, we define the following $\Bbb Z[G]$--modules:
\[J:=\prod_{v\notin S\cup T}^{\sim} O_w^{\times}\times \prod_{v\in S}^{\sim} H_w^{\times}\times \prod_{v\in T}^{\sim} U_w, \qquad
J':=\prod_{v\notin S'\cup T}^{\sim} O_w^{\times}\times \prod_{v\in S'}^{\sim} H_w^{\times}\times \prod_{v\in T}^{\sim} U_w
\]
\[V:=\prod_{v\notin S'\cup T}^{\sim} O_w^{\times}\times \prod_{v\in S'}^{\sim} V_w\times \prod_{v\in T}^{\sim} U_w\]
\[W:=\prod_{v\in S'\setminus S}^{\sim} W_w\times \prod_{v\in S}^{\sim} \Delta G_w, \qquad W':=\prod_{v\in S'}^{\sim} \Delta G_w. \]
Via \eqref{Weil-SES-t} and \eqref{ZWG-SES}, we obtain the following obvious commutative diagram with exact rows
\begin{equation} \label{J'VW'-SES}
\begin{tikzcd}
0\arrow{r} & J\arrow{r}\arrow{d} & V\arrow{r}\arrow[]{d}{=} & W \arrow{r}\arrow[]{d} & 0 \\
0\arrow{r} & J' \arrow{r} & V \arrow{r} & W'\arrow{r} & 0. \\
\end{tikzcd}
\end{equation}
By Theorem 1 in \cite{Ritter-Weiss}, we have the following commutative diagram with exact rows and columns: 
\begin{equation}\label{Bigger-diagram}  
\begin{tikzcd}
& 0\arrow{d} & 0\arrow{d} & 0\arrow{d} \\%
& O_{H,S,T}^{\times}\arrow{d} & V^{\theta}\arrow{d} & W^{\theta}\arrow{d} \\%
0\arrow{r} & J\arrow{r}\arrow{d}{\theta_J} & V \arrow{r}\arrow{d}{\theta_V} & W \arrow{r}\arrow{d}{\theta_W} & 0 \\%
0\arrow{r} & C_H\arrow{r}\arrow[two heads]{d} & D \arrow{r}\arrow{d} & \Delta G\arrow{r}\arrow{d} & 0 \\%
& Cl_S^T(H) & 0 & 0
\end{tikzcd}
\end{equation}
where the vertical maps $\theta$ are defined explicitly in \S5 of \cite{gambheera-popescu}.
By the snake lemma, we obtain a short exact sequence of $\mathbb{Z}[G]-$modules:
\begin{equation}\label{Long-OVW-SES}
0\xrightarrow[]{} O_{H,S,T}^{\times}\xrightarrow[]{} V^{\theta}\xrightarrow[]{\xi} W^{\theta}\xrightarrow[]{} Cl_S^T(H)\xrightarrow[]{} 0.
\end{equation}
Similarly, using $J'$ and $W'$ instead of $J$ and $W$, we obtain the following short exact sequence of $\mathbb{Z}[G]-$modules:
\begin{equation}\label{OVW-SES}
0\xrightarrow[]{} O_{H,S',T}^{\times}\xrightarrow[]{} V^{\theta}\xrightarrow[]{} W'^{\theta}\xrightarrow[]{} 0.
\end{equation}
We recall some more definitions. 
\[B:= \prod_{v\in S'}\mathbb{Z}[G],\qquad 
Z:=\prod_{v\in S}^{\sim} \mathbb{Z}.  \]
These modules fit into the following commutative diagram of $G$--modules. (See \S5 in \cite{gambheera-popescu} for the exact definitions of the transition maps $\theta_\ast$.)
\begin{equation}\label{Big-diagram} 
\begin{tikzcd}
& 0\arrow{d} & 0\arrow{d} & 0\arrow{d} \\%
& W^{\theta}\arrow{d} & B^{\theta}\arrow{d} & Z^{\theta}\arrow{d} \\%
0\arrow{r} & W\arrow{r}{\gamma}\arrow{d}{\theta_W} & B \arrow{r}\arrow{d}{\theta_B} & Z \arrow{r}\arrow{d}{\theta_Z} & 0 \\%
0\arrow{r} & \Delta G\arrow{r}\arrow{d} & \mathbb{Z}[G]\arrow{r}\arrow{d} & \mathbb{Z}\arrow{r}\arrow{d} & 0 \\%
& 0 & 0 & 0  
\end{tikzcd}
\end{equation}
Again, an application of the snake lemma gives the following short exact sequence of $\Bbb Z[G]$--modules:
\begin{equation}\label{WBZ-SES}
0\xrightarrow[]{} W^{\theta} \xrightarrow[]{\gamma} B^{\theta} \xrightarrow[]{} Z^{\theta} \xrightarrow[]{} 0
\end{equation}

\begin{definition} The Ritter--Weiss module $\nabla_S^T(H)$ is defined by \begin{equation}\label{Nabla-definition} 
\nabla_S^T(H):={\rm coker}(\rho:\, V^{\theta}\xrightarrow[]{\xi} W^{\theta} \xrightarrow[]{\gamma} B^{\theta})  \end{equation}
\end{definition}
It turns out that the above definition is independent of the choice of the auxiliary set of primes $S'$. (See \cite{Dasgupta-Kakde} for details.) Also, since $\gamma$ is injective, we get the following short exact sequence of $\Bbb Z[G]$--modules:
\begin{equation}\label{VBNabla-SES}
0\xrightarrow[]{} O_{H,S,T}^{\times}\xrightarrow[]{} V^{\theta} \xrightarrow[]{\rho} B^{\theta}\xrightarrow[]{}\nabla_S^T(H)\xrightarrow[]{}0
\end{equation}

The following result from Appendix A in \cite{Dasgupta-Kakde} shows that, under certain conditions, the above definition gives a locally quadratic presentation of the Ritter-Weiss module. 

\begin{proposition}\label{nabla-quad} 
     Let $R$ be a commutative $\mathbb{Z}[G]-$algebra, such that for all primes $v\in T\cap S_{ram}(H/F)$, the rational prime $l$ below $v$ is invertible in $R$.  
     Define $$V_R^{\theta}:=V^{\theta}\otimes_{\mathbb{Z}[G]}R, \quad B_R^{\theta}:=B^{\theta}\otimes_{\mathbb{Z}[G]}R, \quad \nabla_S^T(H)_R:=\nabla_S^T(H)\otimes_{\mathbb{Z}[G]}R.$$ 
     Then, the following hold: 
\begin{enumerate}
\item   $V_R^{\theta}$ and $B_R^\theta$ are projective $R$--modules of constant local ranks $(|S'|-1)$. 
\item The $R$--module presentation of $\nabla_S^T(H)_R$ obtained by tensoring \eqref{VBNabla-SES} with $R$:
$$V_R^{\theta}\xrightarrow[]{\rho_R} B_R^{\theta}\xrightarrow[]{}\nabla_S^T(H)_R\xrightarrow[]{}0$$
is a locally quadratic presentation.
\end{enumerate}
\end{proposition}     
 \bigskip
 
\begin{remark}[Freeness]\label{freeness-remark} The $\Bbb Z[G]$--module $B^\theta$ is in fact free of rank $|S'|-1$, as it follows from the definition of the map $\theta_B$. In fact, if we pick once and for all a prime $v_\infty\in S_\infty(F)$, then this leads to an obvious canonical splitting of $\theta_B$ and a canonical isomorphism 
$$B^\theta\simeq \prod_{v\in S'\setminus\{v_\infty\}}\Bbb Z[G],$$
and therefore a standard basis for $B^\theta$.
(See \cite{Dasgupta-Kakde} and isomorphism (16) in \cite{Dasgupta-Kakde-Silliman-ETNC}.) 

Less obvious is the fact that if $S_\infty\cup S_{ram}(H/F)\subseteq S$, $H$ is CM, $F$ is totally real and $R$ is an algebra over $\Bbb Z[G]/(1+j)$, where $j$ is the complex conjugation automorphism of $H$, then the $R$--module $V_R^\theta$ is free of rank $|S'|-1$. (See Theorem 1 in \cite{Dasgupta-Kakde-Silliman-ETNC}.) Kurihara has in fact conjectured in \cite{Kurihara} that $V^\theta$ is a free $\Bbb Z[G]$--module of rank $(|S'|-1)$, but this remains unproved at the moment.

The freeness of these modules will only play a role in \S5.
\end{remark}
\bigskip

Now, combining the sequences \ref{Long-OVW-SES} and \ref{WBZ-SES}, we get the following sequence.
\begin{equation}\label{Nabla-SES}
0\xrightarrow[]{} Cl_S^T(H) \xrightarrow[]{} \nabla_S^T(H)\xrightarrow[]{} Z^{\theta}\xrightarrow[]{} 0
\end{equation}
Note that since $Z=\prod_{v\in S}\Bbb Z[G/G_v]$ can be identified (as a $G$--module) with the set of divisors in $H$ supported at primes in $S_H$ and $\theta_Z: Z\to\Bbb Z$ is simply the dvisor degree map, $Z^\theta$ is the set of divisors of degree $0$ in $H$ supported at $S_H$. This is a module usually denoted by $X_S$ in classical texts.
\subsection{A link between the Ritter--Weiss and the Selmer modules}
It turns out that there is an intimate link between the Ritter-Weiss modules defined above and the Selmer modules given in Definition \ref{sel-def}. This link, described in the Theorem below,  will allow us to use our results on Selmer modules in \cite{gambheera-popescu}  to study the Ritter--Weiss modules in what follows.
\begin{theorem}[Dasgupta--Kakde]\label{nabla-sel-theorem}
    Let $(H/F,G,S,T,R)$ be as in Proposition \ref{nabla-quad}. Then, the Selmer module $Sel_S^T(H)_R:=Sel_S^T(H)\otimes_{\mathbb{Z}[G]}R$ has the following locally quadratic presentation
    $$(B_R^{\theta})^*\xrightarrow[]{\rho_R^\ast}(V_R^{\theta})^*\xrightarrow{}Sel_S^T(H)_R\xrightarrow[]{}0.$$
    Consequently, we have an isomorphism of $R$--modules and an equality of $R$--ideals, respectively
    $$Sel_S^T(H)_R\cong \nabla_S^T(H)_R^{tr},\qquad {\rm Fitt}_R(\nabla_S^T(H)_R)={\rm Fitt}_R(Sel_S^T(H)_R),$$
    where the transposed is taken with respect to the quadratic presentation in Proposition \ref{nabla-quad}.
\end{theorem}
\begin{proof}
See Appendix A in \cite{Dasgupta-Kakde} for the isomorphism  $Sel_S^T(H)_R\cong \nabla_S^T(H)_R^{tr}$ and then apply Proposition \ref{Fitt-transp}.
\end{proof}

\subsection{Transition maps for Ritter--Weiss modules} In this section, we consider number fields $F\subseteq K_1\subseteq K_2$, with $K_i/F$ abelian, of Galois group $G_i$, for $i=1,2$, and sets of primes $S, S', T$ satisfying the conditions in the previous sections for both $K_1/F$ and $K_2/F$. Our goal will be to establish a canonical restriction map $\nabla_S^T(K_2)\to \nabla_S^T(K_1)$, compatible in a very precise sense with the exact sequences of type \eqref{Nabla-SES} at the $K_2$ and $K_1$ levels, respectively.

First, we start with a local construction of transition maps compatible with the exact sequences of type \eqref{Weil-SES-t} at the $K_2$ and $K_1$ levels. For that purpose, we fix a finite place $u$ of $F$ and, above it a place $v$ of $K_1$ and, above $v$, a place $w$ of $K_2$. Now, let $F_u$, $K_{1,v}$ and $K_{2,w}$ be completions of the corresponding fields with respect to those places. Then we have the decomposition groups of $u$ in $K_1/F$ and $K_2/F$ are $G_v:=Gal(K_{1,v}/F_u)$ and $G_w:=Gal(K_{2,w}/F_u)$.

Now, let $\pi: G_2\to G_1$ be the Galois restriction group morphism. This induces morphisms of rings and modules over those rings $\pi:\Bbb Z[G_2]\to \Bbb Z[G_1]$, $\pi: \mathbb{Z}[G_w]\xrightarrow[]{} \mathbb{Z}[G_v]$, and $\pi:\Delta G_w \xrightarrow[]{} \Delta G_v$.
\begin{proposition}\label{compatibility}
There exist a $\mathbb{Z}[G_w]$-module morphism $f_w$ such that the following diagram commutes 
\[\begin{tikzcd}
0\arrow{r} & K_{2,w}^{\times}\arrow{r}\arrow[]{d}{N_w} & V_w^2 \arrow{r}\arrow[]{d}{f_w} & \Delta G_w \arrow{r} \arrow[]{d}{\pi} & 0 \\%
0\arrow{r} & K_{1,v}^{\times}\arrow{r} & V_v^1 \arrow{r} & \Delta G_v \arrow{r} & 0,
\end{tikzcd}
\]
where the left vertical map is he norm map.
\end{proposition}
\begin{proof}
By local class field theory, we have a commutative diagram whose rows are the exact sequences \ref{Weil-SES} at the $K_2$ and $K_1$ levels, respectively:  
\begin{equation}\label{KWG-diagram} 
\begin{tikzcd}
0\arrow{r} & K_{2,w}^{\times}\arrow{r}\arrow[]{d}{N_w} & W(K_{2,w}^{ab}/F_u) \arrow{r}\arrow[]{d}{res} & G_w \arrow{r} \arrow[]{d}{\pi} & 0 \\%
0\arrow{r} & K_{1,v}^{\times}\arrow{r} & W(K_{1,v}^{ab}/F_u) \arrow{r} & G_v \arrow{r} & 0
\end{tikzcd}
\end{equation}
where the middle vertical map is Galois restriction. Now, by applying the functor $t$ of \cite{Ritter-Weiss} to the rows in the diagram above, we obtain the desired result.
\end{proof}
Now, we need to glue these local commutative diagrams to obtain similar global commutative diagrams. A very useful technical tool in that direction is the following.\\

Suppose that $X$ is a $G_w$--module and $Y$ is a $G_v$--module (hence, is also a $G_w$--module via Galois restriction). Let $\xi_w: M\xrightarrow[]{} N$ be a $G_w$- module morphism. Then, the following global map 
$$\xi: \mathbb{Z}[G_2]\otimes_{\mathbb{Z}[G_w]} X \xrightarrow[]{} \mathbb{Z}[G_1]\otimes_{\mathbb{Z}[G_v]} Y, \qquad  \xi(g\otimes x)=\pi(g)\otimes \xi_w(x),\quad\text{ for all }g\in G_2, x\in M$$ 
is a well--defined $G_2$--module morphism.
Further, observe that we have a commutative diagram
\[\begin{tikzcd}
\mathbb{Z}[G_2]\otimes_{\mathbb{Z}[G_w]} X \arrow{r}{\xi}\arrow[leftarrow]{d}{} &  \mathbb{Z}[G_1]\otimes_{\mathbb{Z}[G_v]} Y \arrow[leftarrow]{d}{} \\%
X \arrow{r}{\xi_w} & Y
\end{tikzcd}
\]
where the vertical maps are the canonical embeddings.
\\

The following two lemmas give an explicit description of the global map $f$ for some concrete local maps $f_w$ of interest to us in what follows.

\begin{lemma}{\label{Galois restriction}}
If $X=\mathbb{Z}[G_w]$ , $Y=\mathbb{Z}[G_v]$ and $\xi_w=\pi$ is the local Galois restriction, then $\xi$ is the global Galois restriction (which we also call $\pi$).
\end{lemma}
\begin{proof}
We know that, $\mathbb{Z}[G_2]\otimes_{\mathbb{Z}[G_w]} \mathbb{Z}[G_w]\cong \mathbb{Z}[G_2]$ and $\mathbb{Z}[G_1]\otimes_{\mathbb{Z}[G_v]} \mathbb{Z}[G_v]\cong \mathbb{Z}[G_1]$. Now, under these isomorphisms, observe that, for all $g\in G_2$, we have $\xi(g)=\xi(g\otimes 1)=\pi(g)\otimes \pi(1)=\pi(g)\otimes 1=\pi(g)$. This completes the proof.
\end{proof}

\begin{lemma}
If $X=K_{2,w}^{\times}$, $Y=K_{1,v}^{\times}$ and $\xi_w=N_w$ is the local norm map, then the induced global map $\xi$ gives a commutative diagram 
\[\begin{tikzcd}
\mathbb{Z}[G_2]\otimes_{\mathbb{Z}[G_w]} K_{2,w}^{\times} \arrow{r}{\xi}\arrow[leftarrow]{d}{i_2} &  \mathbb{Z}[G_1]\otimes_{\mathbb{Z}[G_v]} K_{1,v}^{\times} \arrow[leftarrow]{d}{i_1} \\%
K_2^{\times} \arrow{r}{N} & K_1^{\times}
\end{tikzcd}
\]
where $i_1$ and $i_2$ are the canonical diagonal embeddings and $N$ is the global norm map.
\end{lemma}
\begin{proof}
Let $\{\Tilde{\rho_i}\mid i\}$ be coset representatives for $G_2/G_w$ in $G_2$ and $\{\rho_j\mid j\}$ be coset representatives of $G_1/G_v$ in $G_1$, such that 
$\{\pi(\tilde{\rho_i})\mid i\}=\{\rho_j\mid j\}$. Fo an arbitrary $x\in K_2^{\times}$, we have
\begin{eqnarray*}
\xi(i_2(x))&=&\xi(\sum_{i} \Tilde{\rho_i}\otimes \Tilde{\rho_i}^{-1}x)\\
&=&\sum_{i} \pi(\Tilde{\rho_i})\otimes N_w(\Tilde{\rho_i}^{-1}x)\\
&=&\sum_{j}\sum_{\pi(\Tilde{\rho_i})=\rho_j}(\rho_j\otimes N_w(\Tilde{\rho_i}^{-1}x))
\\
&=&\sum_{j}(\rho_j\otimes\prod_{\pi(\Tilde{\rho_i})\rho_j} N_w(\Tilde{\rho_i}^{-1}x)).
\end{eqnarray*}
For a fixed $j$, let $\alpha= \prod_{\pi(\Tilde{\rho_i})=\rho_j} N_w(\Tilde{\rho_i}^{-1}x)$. Now, let $\theta\in G_2$ such that $\pi(\theta)=\rho_j$. Observe that we can always choose the $\Tilde{\rho_i}$'s such that $\{\Tilde{\rho_i}\mid \pi(\Tilde{\rho_i})=\rho_j\}=\theta\cdot \{\Tilde{\rho_i}\mid \pi(\Tilde{\rho_i})=1\}$. For such a choice, we have  
$$\alpha=\prod_{\pi(\Tilde{\rho_i})=1} N_w(\Tilde{\rho_i}^{-1}\theta^{-1} x)=\prod_{\pi(\Tilde{\rho_i})=1}\left(\prod_{g\in\ker(\pi|_{G_w})}g\Tilde{\rho_i}^{-1}(\theta^{-1} x)\right)=N(\theta^{-1}x)=\rho_j^{-1}N(x).$$
Consequently, we have  $$\xi(i_2(x))=\sum_{j}(\rho_j\otimes \rho_j^{-1}N(x))=i_1(N(x)),$$
which concludes the proof of the Lemma.
\end{proof}
Now, by gluing the local data from Proposition \ref{compatibility} according to the machinery given in the last two Lemmas, we obtain a three dimensional commutative diagram whose top and bottom are the commutative diagrams \eqref{J'VW'-SES} at the $K_2$ and $K_1$ levels, respectively.
\begin{equation}\label{JVW-diagram}
 \begin{tikzcd}[row sep=1.5em, column sep = 1.5em]
  0\arrow[rr] && J_2\arrow[rr] \arrow[dr, swap] \ar{dd} &&
    V_2 \arrow[rr]\arrow[dd] \ar[equal]{dr}  &&
    W_2 \arrow[rr]\arrow[dd] \arrow[dr] && 0\\
    & 0\arrow[rr] && J_2'\arrow[rr]  \ar{dd} &&
    V_2 \arrow[rr]\ar{dd}   &&
    W_2' \arrow[rr]\arrow[dd]  && 0 \\
    0\arrow[rr] && J_1\arrow[rr] \arrow[dr, swap]  &&
    V_1 \arrow[rr] \ar[equal]{dr}  &&
    W_1 \arrow[rr] \arrow[dr] && 0\\
    & 0\arrow[rr] && J_1'\arrow[rr]   &&
    V_1 \arrow[rr]   &&
    W_1' \arrow[rr] && 0
    \end{tikzcd}    
\end{equation}
The left vertical maps (which we refer to as ``Norm") are induced by the local norm maps. The middle vertical maps (which we refer to as $f$) are induced by the local map $f_w$ in Proposition \ref{compatibility}. The right vertical maps (which we refer to as ``$\pi$") are induced by the Galois restriction maps $\pi$ and the maps $f_w$.\\\\

Next, we prove a global analog of Proposition \ref{compatibility}. 

\begin{proposition}\label{global-compatibility}
If $N:C_{K_2}\to C_{K_1}$ denotes the norm map at the level of id\`ele classes, then there exists a $G_2$-module morphism $d$ such that the following diagram commutes.
\[\begin{tikzcd}
0\arrow{r} & C_{K_2}\arrow{r}\arrow[]{d}{N} & D_2\arrow{r}\arrow[]{d}{d} & \Delta G_2 \arrow{r} \arrow[]{d}{\pi} & 0 \\%
0\arrow{r} & C_{K_1}\arrow{r} & D_1 \arrow{r} & \Delta G_1 \arrow{r} & 0.
\end{tikzcd}
\]
\end{proposition}
\begin{proof}  Global class field theory (see \cite{Tate-CFT} and also see the exact sequences (i) on page 168 of \cite{Ritter-Weiss}) gives us a commutative diagram in the category of topological groups
\begin{equation}\label{CWG-diagram}    
\begin{tikzcd}
0\arrow{r} & W(K_{2}^{ab}/K_{2})\cong C_{K_2}\arrow{r}\arrow[]{d}{N} & W(K_{2}^{ab}/F) \arrow{r}\arrow[]{d}{res} & G_2 \arrow{r} \arrow[]{d}{\pi} & 0 \\%
0\arrow{r} & W(K_{1}^{ab}/K_{1})\cong C_{K_1}\arrow{r} & W(K_{1}^{ab}/F) \arrow{r} & G_1 \arrow{r} & 0,
\end{tikzcd}
\end{equation}
where $N$ is the norm map at the level of id\'ele classes. Now, by applying the functor $t$ of \cite{Ritter-Weiss} to the commutative diagram above, we get the desired result.
\end{proof}

\begin{proposition}\label{commutative-cuboid}
We have a commutative diagram of $G_2$--modules with exact horizontal edges
\[
 \begin{tikzcd}[row sep=1.5em, column sep = 1.5em]
  0\arrow[rr] && J_2\arrow[rr] \ar{dr}{\theta_{J_2}} \ar{dd} &&
    V_2 \arrow[rr]\arrow[dd] \ar{dr}{\theta_2}  &&
    W_2 \arrow[rr]\arrow[dd] \ar{dr}{\theta_{W_2}} && 0\\
    & 0\arrow[rr] && C_{K_2} \arrow[rr]  \ar{dd} &&
    D_2 \arrow[rr]\ar{dd}   &&
    \Delta G_2 \arrow[rr]\arrow[dd]  && 0 \\
    0\arrow[rr] && J_1\arrow[rr] \ar{dr}{\theta_{J_1}}  &&
    V_1 \arrow[rr] \ar{dr}{\theta_1}  &&
    W_1 \arrow[rr] \ar{dr}{\theta_{W_1}} && 0\\
    & 0\arrow[rr] && C_{K_1} \arrow[rr]   &&
    D_1 \arrow[rr]   &&
    \Delta G_1 \arrow[rr] && 0
    \end{tikzcd}
\]
whose top and bottom are given by \eqref{Bigger-diagram} at the $K_2$ and $K_1$ levels, respectively, whose back is given by \eqref{JVW-diagram}, and whose front is given by Proposition \ref{global-compatibility}.
\end{proposition}
\begin{proof} The compatibility between local and global class field theory, gives natural maps connecting diagrams \ref{KWG-diagram} and \ref{CWG-diagram}, leading to the following commutative diagram. (See the diagram at the bottom of page 168 in \cite{Ritter-Weiss}.)
\[
 \begin{tikzcd}[row sep=1.5em, column sep = 1.0em]
  0\arrow[rr] && K_{2,w}^{\times}\arrow[rr] \ar{dr}{} \ar{dd} &&
    W(K_{2,w}^{ab}/F_u) \arrow[rr]\arrow[dd] \ar{dr}{}  &&
    G_w \arrow[rr]\arrow[dd] \ar{dr}{} && 0\\
    & 0\arrow[rr] && C_{K_2} \arrow[rr]  \ar{dd} &&
    W(K_2^{ab}/F) \arrow[rr]\ar{dd}   &&
    G_2 \arrow[rr]\arrow[dd]  && 0 \\
    0\arrow[rr] && K_{1,v}^{\times}\arrow[rr] \ar{dr}{}  &&
    W(K_{1,v}^{ab}/F_u) \arrow[rr] \ar{dr}{}  &&
    G_v \arrow[rr] \ar{dr}{} && 0\\
    & 0\arrow[rr] && C_{K_1} \arrow[rr]   &&
    W(K_1^{ab}/F) \arrow[rr]   &&
    G_1 \arrow[rr] && 0
    \end{tikzcd}
\]
 When applying the functor $t$ of \cite{Ritter-Weiss} to the above diagram, we get a commutative diagram which connects the diagrams in Proposition \ref{compatibility} and Proposition \ref{global-compatibility}. Then, by gluing the local diagrams (inner faces) appropriately, by using the machinery introduced in the last two Lemmas, we get the following commutative diagram.
\[
 \begin{tikzcd}[row sep=1.5em, column sep = 1.5em]
  0\arrow[rr] && J_2'\arrow[rr] \ar{dr}{\theta_{J_2'}} \ar{dd} &&
    V_2 \arrow[rr]\arrow[dd] \ar{dr}{\theta_2}  &&
    W_2' \arrow[rr]\arrow[dd] \ar{dr}{\theta_{W_2'}} && 0\\
    & 0\arrow[rr] && C_{K_2} \arrow[rr]  \ar{dd} &&
    D_2 \arrow[rr]\ar{dd}   &&
    \Delta G_2 \arrow[rr]\arrow[dd]  && 0 \\
    0\arrow[rr] && J_1'\arrow[rr] \ar{dr}{\theta_{J_1'}}  &&
    V_1 \arrow[rr] \ar{dr}{\theta_1}  &&
    W_1' \arrow[rr] \ar{dr}{\theta_{W_1'}} && 0\\
    & 0\arrow[rr] && C_{K_1} \arrow[rr]   &&
    D_1 \arrow[rr]   &&
    \Delta G_1 \arrow[rr] && 0
    \end{tikzcd}
\]
Now, by connecting the above diagram with diagram \ref{JVW-diagram}, we obtain the desired result.
\end{proof}

When applying the snake lemma to the top and bottom of the diagram in the last Proposition, we obtain the following commutative diagram, which is a morphism between sequences (138) in \cite{Dasgupta-Kakde} at levels $K_2$ and $K_1$, inducing the appropriate norm maps at the level of global units and class--groups. 
\begin{equation}\label{OVW-diagram}
\begin{tikzcd}
0\arrow{r} & {\mathcal O}_{K_2,S,T}^{\times}\arrow{r}\arrow{d}{N} & V_2^{\theta} \arrow{r}\arrow{d} & W_2^{\theta}\arrow{r}\arrow{d} & Cl_S^T(K_2) \arrow{r}\arrow{d}{N} & 0 \\%
0\arrow{r} & {\mathcal O}_{K_1,S,T}^{\times}\arrow{r} & V_1^{\theta} \arrow{r} & W_1^{\theta}\arrow{r} & Cl_S^T(K_1) \arrow{r} & 0
\end{tikzcd}
\end{equation}

Now, we look at maps $\gamma:W\xrightarrow[]{} B$ as in diagram \ref{Big-diagram} at levels $K_1$ and $K_2$, respectively, and then construct maps between levels. In order to do that, let us analyze the modules $W$ more closely. Here, we are using the description given in \cite{Gruenberg-Weiss}. We start by focusing on one level (say $K_1$).
Let $\overline{G_v}=\langle F \rangle$ be the Galois group of the residue field extension of $K_{1,v}/F_u$, where $F$ is the appropriate Frobenius automorphism. Then, we have
$$W_v=\{(x,y)\in \Delta G_v\bigoplus \mathbb{Z}[\overline{G_v}]\mid \Bar{x}=(F-1)y\},$$ where $\Bar{x}$ is the image of $x$ in $\mathbb{Z}[\overline{G_v}]$. (See loc.cit.) Clearly, $W_v$ is a free $\mathbb{Z}$--module. A $\mathbb{Z}$--basis is given by $$\{w_g=(g-1,\sum_{i=0}^{a(g)-1}F^i) \mid  g\in G_v\},$$ where $a(g)$ defined such that for each $g\in G_v, \Bar{g}=F^{a(g)}$ and $0<a(g)\leq f_1:=|\overline{G_v}|$. Under the notation of the short exact sequence \ref{ZWG-SES}, we have $i(1)=w_1$ and $j(w_g)=g-1$. Now, the $G_v$--action on this basis is given by $$g\cdot w_h=w_{gh}-w_g+a_{g,h}w_1,$$ for each $g,h\in G_v$, where $a_{g,h}$ is defined by $$a(g)+a(h)=a(gh)+f_1 a_{g,h}.$$ Observe that $a_{1,h}=1$ for each $h\in G_v.$\\\\
Now, we recall the following technical result from \cite{gambheera-popescu}. (See Proposition 5.29 in loc.cit.)
\begin{lemma}\label{formulas}
Let $\Tilde{h}\in G_w$ is a lift of $h\in G_v$ and $g\in G_v$. Let $e$ be the ramification index of the extension $K_{2,w}/K_{1,v}$ and $f,f_1$ and $f_2$ be the residual degrees of the extensions $K_{2,w}/K_{1,v}$, $K_{1,v}/F_u$ and $K_{2,w}/F_u$, respectively. Then, the followings are true.\\\\
(1) $a(\Tilde{h})=a(h)+k_{\Tilde{h}}f_1$ for some $k_{\Tilde{h}}\in\{0,1,2,...\text{ }  f-1\}$.\\
(2)$\sum_{\Tilde{g}\xrightarrow[]{} g} a_{\Tilde{g},\Tilde{h}}=e(a_{g,h}+k_{\Tilde{h}})$
\end{lemma}
Observe that the $G_w$-module map $f_w$ in  Proposition \ref{compatibility} induces a map $$f_w':V_w^2/\mathbf{O}_w^{\times}=W_w^2\xrightarrow[]{} V_v^1/\mathbf{O}_v^{\times}=W_v^1.$$
This map is a local component at primes $v\in S'\setminus S$ of the map $W_2\to W_1$ in diagram \ref{JVW-diagram}. Now, we give an explicit description of this map in terms of the above mentioned $\mathbb{Z}$-basis elements.  
\begin{proposition} The map
$f_w':W_w^2\xrightarrow[]{} W_v^1$ is given by $f_w'(w_{\Tilde{g}})=w_g+k_{\Tilde{g}}\cdot w_1$, for all $\Tilde{g}\in G_w$. Here, $g=\pi(\Tilde{g})$.
\end{proposition}
\begin{proof}
Proposition \ref{compatibility} gives maps as described below between the short exact sequences \ref{ZWG-SES} at the $K_2$ and $K_1$ levels. Here ``$\times f$'' denotes multiplication by the residual degree of the extension $K_{2,w}/K_{1,v}$.

\begin{equation}\label{ZWG-diagram}
\begin{tikzcd}
0\arrow{r} & \mathbb{Z}\arrow{r}{i_2}\arrow[]{d}{\times f} & W_w^2\arrow{r}{j_2}\arrow[]{d}{f_w'} & \Delta G_w \arrow{r}\arrow[]{d}{\pi} & 0 \\%
0\arrow{r} & \mathbb{Z}\arrow{r}{i_1} & W_v^1\arrow{r}{j_1} & \Delta G_v \arrow{r} & 0
\end{tikzcd}
\end{equation}
Observe that for any $g\in G_v$, by the commutativity of the right square, we have,
$$j_1(f_w'(w_{\Tilde{g}}))=\pi(j_2(w_{\Tilde{g}}))=\pi(\Tilde{g}-1)=g-1.$$
On the other hand, we know that $j_1(w_g)=g-1$. Therefore, by the exactness of the upper row, we have $f_w'(w_{\Tilde{g}})=w_g+e_{\Tilde{g}}\cdot w_1$ for some $e_{\Tilde{g}}\in\mathbb{Z}$. Moreover, from the left commutative square, we have $f_w'(w_{\Tilde{1}})=f\cdot w_1$, were $\Tilde{1}$ is the identity of $G_w$.

Now, we are left to prove that $e_{\Tilde{g}}=k_{\Tilde{g}}$ for all $\Tilde{g}\in G_w$. Since we also know that $f_w'$ is $G_w$- equivariant, for all $\Tilde{g}, \Tilde{h}\in G_w$ above $g,h\in G_v$ we have, $h\cdot (f_w'(w_{\Tilde{g}}))=f_w'(\Tilde{h}\cdot w_{\Tilde{g}})$. Observe that
\begin{eqnarray*}
    h\cdot (f_w'(w_{\Tilde{g}}))&=&h\cdot (w_g+e_{\Tilde{g}}\cdot w_1)\\
    &=&w_{gh}-w_h+a_{a,h}w_1+e_{\Tilde{g}}\cdot w_1\\
    &=&w_{gh}-w_h+(a_{g,h}+e_{\Tilde{g}})\cdot w_1.
\end{eqnarray*}
On the other hand, we have 
\begin{eqnarray*}
  f_w'(\Tilde{h}\cdot w_{\Tilde{g}})&=&f_w'(w_{\Tilde{g}\Tilde{h}}-w_{\Tilde{h}}+a_{\Tilde{g},\Tilde{h}}\cdot w_{\Tilde{1}})\\
  &=&w_{gh}+e_{\Tilde{g}\Tilde{h}}\cdot w_1-w_h-e_{\Tilde{h}}\cdot w_1+a_{\Tilde{g},\Tilde{h}}f\cdot w_{1}\\
  &=&w_{gh}-w_h+(a_{\Tilde{g},\Tilde{h}}f+e_{\Tilde{g}\Tilde{h}}-e_{\Tilde{h}})w_1  
\end{eqnarray*}
Hence, for all $\Tilde{g}, \Tilde{h}\in G_w$ above $g,h\in G_v$ we have
$$a_{g,h}+e_{\Tilde{g}}=e_{\Tilde{gh}}-e_{\Tilde{h}}+a_{\Tilde{g},\Tilde{h}}f.$$
Now, by taking the summation of above equation when $\Tilde{h}\in G_w$ varies such that $\pi(\Tilde{h})=h$ and applying Lemma \ref{formulas} (b), we get
$$ef(a_{g,h}+e_{\Tilde{g}})=\sum_{\pi(\Tilde{h})=h}(e_{\Tilde{gh}}-e_{\Tilde{h}})+ef(a_{g,h}+k_{\Tilde{g}}).$$
finally, by taking the summation when $h$ varies through all the elements in $G_w$, the first term of the right hand side vanishes. Then, we easily get $e_{\Tilde{g}}=k_{\Tilde{g}}$, as desired. 
\end{proof}

Next, we review the map $\gamma:W\xrightarrow[]{} B$ in the commutative diagram \ref{Big-diagram} from \cite{Dasgupta-Kakde}. Let us look at its component-wise definition at the $K_1$ level. 
\begin{itemize}
    \item For $v\in S,$ $\gamma_v$ is induced by the inclusion $\Delta G_v\subseteq\mathbb{Z}[G_v]$
    \item For $v\in S'\setminus S$, $\gamma_v$ is induced by the map $s$, described as follows.
$$s(w_g)=\sum_{h\in G_v}(r(g)+1-a_{g^{-1},h})h$$ for all $g\in G_v$, where $r$ is given by 
    \begin{equation*}
r(g) =  \left\{
        \begin{array}{ll}
            1 & if \quad g\in I_v \\
            0 & if \quad g\notin I_v
        \end{array}
    \right.
\end{equation*}
and $I_v$ is the corresponding ramification group. This description is from \cite{Gruenberg-Weiss}.
\end{itemize}

\begin{lemma}\label{lemma}
For all $\Tilde{g}\in G_w$ we have $k_{\Tilde{g}^{-1}}+k_{\Tilde{g}}=(r(\Tilde{g})+1)f-(r(g)+1)$. Here $g=\pi(\Tilde{g})$.
\end{lemma}
\begin{proof}
We split the proof into three cases.

$\bullet$ Case (i): $g\notin I_{1,v}$ where $I_{1,v}$ is the ramification group of the extension $K_{1,v}/F_u.$
In this case we also have $\Tilde{g}\notin I_{2,w}$. Here, $I_{2,w}$ is the ramification group of the extension $K_{2,w}/F_u$. Let, $f_1$ and $f_2$ are the residual degrees of the extensions $K_{1,v}/F_u$ and $K_{2,w}/F_u$ respectively. Therefore,
$$k_{\Tilde{g}^{-1}}+k_{\Tilde{g}}=\frac{1}{f_1}(a(\Tilde{g})+a(\Tilde{g}^{-1})-a(g)-a(g^{-1}))=\frac{f_2-f_1}{f_1}=f-1$$
\\

$\bullet$ Case (ii): $g\in I_{1,v}$ and $\Tilde{g}\notin I_{2,w}.$ We have:
$$k_{\Tilde{g}^{-1}}+k_{\Tilde{g}}=\frac{f_2-2f_1}{f_1}=f-2.$$
\\

$\bullet$  Case (iii): $\Tilde{g}\in I_{2,w}.$ In this case we also have $g\in I_{1,v}$ and 
$$k_{\Tilde{g}^{-1}}+k_{\Tilde{g}}=\frac{2f_2-2f_1}{f_1}=2f-2.$$
\\

In all three cases we have obtained he desired equality.
\end{proof}
From the right square of the commutative diagram \ref{CWG-diagram}, the map $f_w'$ is compatible with the maps $j$ and $\pi$. Now we prove a similar result for the maps $s$ at the $K_{1,v}$ and $K_{2,w}$ levels.

\begin{proposition}
Suppose that $w$ is unramified in $K_{2,w}/K_{1,v}$. Then, we have the following commutative diagram of $G_w$-modules.
\[\begin{tikzcd}
W_w^2\arrow{r}{s_2}\arrow{d}{f_w'} & \mathbb{Z}[G_w]\arrow{d}{\pi} \\%
W_v^1\arrow{r}{s_1} & \mathbb{Z}[G_v]
\end{tikzcd}
\]
\end{proposition}
\begin{proof}
Observe that, for all $\Tilde{g}\in G_w$, we have 
$$s_1(f_w'(w_{\Tilde{g}}))=s_1(w_g+k_{\Tilde{g}}\cdot w_1)=\sum_{h\in G_v}(r(g)+1-a_{g^{-1},h}+k_{\Tilde{g}})h,$$ 
where $g=\pi(\Tilde{g})$. On the other hand using Lemma \ref{formulas} (b), we have 
$$\pi(s_2(w_{\Tilde{g}}))=\pi(\sum_{\Tilde{h}\in G_w}(r(\Tilde{g})+1-a_{\Tilde{g}^{-1},\Tilde{h}})\Tilde{h})=\sum_{h\in G_v}(f(r(\Tilde{g})+1)-(a_{g^{-1},h}+k_{\Tilde{g}^{-1}}))h$$
Now, by applying Lemma \ref{lemma}, we have  $s_1(f_w'(w_{\Tilde{g}}))=\pi(s_2(w_{\Tilde{g}}))$. This completes the proof.
\end{proof}
As a consequence, we have the following global result.
\begin{proposition}\label{WB-compatibility}
The following diagram of $G_2$-modules commutes.
\[\begin{tikzcd}
W_2\arrow{r}{\gamma_2}\arrow{d}{f'} & B_2\arrow{d}{\pi} \\%
W_1\arrow{r}{\gamma_1} & B_1
\end{tikzcd}
\]
\end{proposition}
\begin{proof}
This is obtained by gluing the local diagrams at each prime. At primes in $S$, the local diagrams are obviously commutative. Since the primes in $S'\setminus S$ are unramified, the previous Proposition  gives the commutativity of the corresponding local diagrams as well.
\end{proof}
We recall the definition of the map $\theta_B: B\to \Bbb Z[G]$ in diagram \ref{Big-diagram} from \cite{Dasgupta-Kakde}. Component-wise this is defined as follows.
\begin{itemize}
    \item For $v\in S$, the $v$--component of $\theta_B$ is the identity.
    \item For $v\in S'\setminus S$, the $v$--component of $\theta_B$ is given by $\theta_B(x)=(\sigma_v-1)x$, fo all $x\in\Bbb Z[G]$, where $\sigma_v$ is the corresponding Frobenius automorphism.
\end{itemize}

\begin{proposition}\label{B-compatibility}
We have a commutative diagram of $G_2$--modules
\[\begin{tikzcd}
0\ar{r} &B_2^\theta\ar{r}\ar{d}{\pi}& B_2\arrow{r}{\theta_{B_2}}\arrow{d}{\pi} &\Bbb Z[G_2]\ar{d}{\pi}\ar{r} &0\\
0\ar{r} &B_1^\theta\ar{r}& B_1\arrow{r}{\theta_{B_1}} &\Bbb Z[G_1]\ar{r} &0
\end{tikzcd}
\]
with exact rows and surjective vertical maps.
\end{proposition}
\begin{proof} Commutativity follows immediately from the definitions and the fact that if $v$ is a prime in $K_1$ lying in $S'\setminus S$ and $w$ is a prime above it in $K_2$, then 
$\pi(\sigma_w)=\sigma_v$, where $\sigma_w$ and $\sigma_v$ are the corresponding Frobenius morphisms in $G_2$ and $G_1$, respectively.

The surjectivity of the middle and right vertical maps is obvious from the definitions. The surjectivity of the left vertical map is a consequence of the snake lemma is the observation that since $S$ is not empty, $\theta_{B_2}$ maps $\ker(B_2\overset{\pi}\to B_1)$ to $\ker(\Bbb Z[G_2]\overset{\pi}\to \Bbb Z[G_1])$ surjectively.
\end{proof}
We need one more compatibility result before achieving the final goal of this section.
\begin{proposition}\label{WB-kernels}
The following diagram of $G_2$-modules is commutative.
\begin{equation}\label{WB-equation}
\begin{tikzcd}
W_2^{\theta}\arrow{r}{\gamma_2}\arrow{d}{f'} & B_2^{\theta}\arrow{d}{\pi} \\%
W_1^{\theta}\arrow{r}{\gamma_1} & B_1^{\theta}
\end{tikzcd}
\end{equation}
\end{proposition}
\begin{proof}
Observe that we have the following  commutative diagram. 
\[
 \begin{tikzcd}[row sep=1.5em, column sep = 1.5em]
    W_2\arrow[rr] \arrow[dr, swap] \arrow[dd,swap] &&
    W_1 \arrow[dd] \arrow[dr] \\
    & B_2 \arrow[rr] &&
    B_1 \arrow[dd] \\
    \Delta G_2 \arrow[rr,] \arrow[dr] && \Delta G_1 \arrow[dr] \\
    & \mathbb{Z}[G_2] \arrow[rr] \arrow[uu,leftarrow]&& 
    \mathbb{Z}[G_1]
    \end{tikzcd}
\]
Here, the upper face is defined and proved to be commutative in Proposition \ref{WB-compatibility}. The bottom face is defined in the obvious way and clearly commutative. The left and right faces are the left squares of diagram \ref{Big-diagram} at levels $K_2$ and $K_1$, respectively. The back face is the right face in the diagram of Proposition \ref{commutative-cuboid}. The front face is given in the last Proposition.

Now, by taking kernels of the vertical maps in the diagram above, we get our result.
\end{proof}
Now, we are ready to define the transition maps between Ritter-Weiss modules at the $K_2$ and $K_1$ levels.
\begin{definition}
The $G_2$-module morphism $\lambda: \nabla_S^T(K_2)\xrightarrow[]{} \nabla_S^T(K_1)$ is uniquely defined by the following commutative diagram of $G_2$--modules with exact rows
\[\begin{tikzcd}
V_2^{\theta}\arrow{r}{}\arrow{d}{f} & B_2^{\theta}\arrow{d}{\pi}\arrow{r} & \nabla_S^T(K_2)\arrow{d}{\lambda}\ar{r} &0  \\%
V_1^{\theta}\arrow{r}{} & B_1^{\theta}\arrow{r} & \nabla_S^T(K_1)\ar{r} &0,
\end{tikzcd}
\]
where the left square is obtained by connecting diagram \ref{WB-equation} and the middle square of diagram \ref{OVW-diagram}.
\end{definition}

\begin{remark}\label{surjective-lambda} Since $\pi$ is surjective (see Proposition \ref{B-compatibility}), $\lambda$ is surjective as well.
\end{remark}

 Now, as the main result of this section, we construct the desired meaningful map between the sequences \ref{VBNabla-SES} at levels $K_2$ and $K_1$. For that purpose, we remind the reader that, as remarked before,  $Z_\ast$ can be identified with ${\rm Div}_S(K_\ast)$ (the $\Bbb Z[G_2]$--module of divisors supported at primes above $S$ in $K_\ast$), for all $\ast=1,2$. Consequently, from the definition of $\theta$ (which becomes the divisor degree map), this leads to an identification of $\Bbb Z_\ast^\theta$ with the modules $X_{S, K_\ast}$ of $S$--supported divisors of degree $0$ at the $K_\ast$ level.
 
 \begin{theorem} \label{transition}
From the definitions, we have a commutative diagram of $\Bbb Z[G_2]$--modules 
 \[\begin{tikzcd}
0\arrow{r} & Cl_S^T(K_2)\arrow{r}\arrow{d}{N} & \nabla_S^T(K_2)\arrow{r}\arrow{d}{\lambda} & X_{S,K_2}\arrow{r}\arrow{d}{\Tilde{\pi}} & 0 \\%
0\arrow{r} & Cl_S^T(K_1)\arrow{r} & \nabla_S^T(K_1)\arrow{r} & X_{S,K_1}\arrow{r} & 0
\end{tikzcd}
\]
where $N$ is norm map at the level of ideal classes  and $\Tilde{\pi}$ is induced by the map (which we also call $\Tilde{\pi}$) $\sum_w x_w\cdot w\mapsto \sum_v(\sum_{w|v}x_w)\cdot v$ on $S$--supported divisors.
 \end{theorem}
 \begin{proof}
 We have the following commutative diagram.
 \[
 \begin{tikzcd}[row sep=1.5em, column sep = 1.5em]
    B_2\arrow[rr] \arrow[dr, swap] \arrow[dd,swap] &&
    Z_2 \arrow{dd}{\theta_2} \ar{dr}{\Tilde{\pi}} \\
    & B_1 \arrow[rr] &&
    Z_1 \arrow{dd}{\theta_1} \\
    \mathbb{Z}[G_2] \arrow[rr,] \arrow[dr] && \mathbb{Z} \ar[equal]{dr} \\
    & \mathbb{Z}[G_1] \arrow[rr] \arrow[uu,leftarrow]&& 
    \mathbb{Z}
    \end{tikzcd}
\]
where the left face is given by Proposition \ref{B-compatibility}, the back and front faces are the right square in the diagram \ref{Big-diagram} at levels $K_2$ and $K_1$, respectively, the bottom face is given by the augmentation and restriction maps, and the upper face is obvious. Now, by taking kernels of the vertical maps, we get the following commutative diagram.
  
\[
\begin{tikzcd}
B_2^{\theta}\arrow{r}\arrow{d} & Z_2^{\theta}\arrow{d} \\%
B_1^{\theta}\arrow{r} & Z_1^{\theta}
\end{tikzcd}
\]
When combined with Proposition \ref{WB-kernels}, this gives us the following commutative diagram with exact rows.
\begin{equation}\label{WBZ-combatibility}
\begin{tikzcd}
0\arrow{r} & W_2^{\theta}\arrow{r}{}\arrow{d}{f'} & B_2^{\theta}\arrow{d}{\pi}\arrow{r} & Z_2^{\theta}\arrow{d}{\Tilde{\pi}}\arrow{r} & 0   \\%
0\arrow{r} & W_1^{\theta}\arrow{r}{} & B_1^{\theta}\arrow{r} & Z_1^{\theta}\arrow{r} & 0
\end{tikzcd}
\end{equation}
Observe that this gives a map between the exact sequences \eqref{WBZ-SES} at levels $K_2$ and $K_1$. Now, by combining diagrams \ref{OVW-diagram} and \ref{WBZ-combatibility}, we obtain the desired result.
 \end{proof}

\section{An Equivariant Main Conjecture for the Iwasawa--Ritter--Weiss module}

In this section, we use Theorem \ref{transition} to define the Iwasawa--Ritter--Weiss modules. Further, we use the results in \cite{gambheera-popescu} to study their module theoretic properties and prove an equivariant main conjecture type result for them. As an application, in the last subsection we compute the Fitting ideal over a certain equivariant Iwasawa algebra of a module which is closely related to the classical cyclotomic unramified Iwasawa module associated to a CM number field. The latter computation builds in an essential way upon results on shifted Fitting ideals developed in the Appendix.

\subsection{The Iwasawa theoretic set up.} Let $p$ be an odd prime and let $H/F$ be a finite, abelian CM extension of a totally real number field $F$, of Galois group $G:=Gal(H/F)$. Let $H_{\infty}$ be the cyclotomic $\mathbb{Z}_p-$extension of $H$ and let $\mathcal{G}:=Gal(H_{\infty}/F)$. Let $S$ and $T$ be two nonempty, disjoint sets of places in $F$ such that 
\begin{equation}\label{hyp-ST}S_{\infty}\cup 
S_p\subseteq S,\quad T\not\subseteq S_{ram}(H_{\infty}/F), \quad S_{ram}(H_{\infty}/F)\subseteq S\cup T.\end{equation}
As usual, $S_{ram}(H_\infty/F)$ denotes the ramification locus of $H_\infty/F$. Note that $S_p\subseteq S_{ram}(H_\infty/F)$. Further, we let $H_n$ denote the $n$--th layer of the cyclotomic extension $H_\infty/H$, and let $G_n:=Gal(H_n/F)$. As usual, we define the topological, profinite group algebra (the $G$--equivariant Iwasawa algebra)
$$\Bbb Z_p[[\mathcal G]]:=\varprojlim\limits_n\Bbb Z_p[G_n],$$
where the transition maps $\pi_n:{\Bbb Z}_p[G_{n+1}]\to\Bbb Z_p[G_n]$ in the projective limit are  induced by Galois restriction.
\\

Observe that since $H$ is a CM field, $H_n$ and $H_\infty$ are also CM fields, and therefore there is a unique complex conjugation automorphism $j\in \mathcal G$, which restricts to the unique complex conjugation in $G_n$ (abusively denoted $j$), for all $n\geq 0$. Throughout this chapter, for any $\Bbb Z_p[[\mathcal G]]-$module $M$, we define,
$$M^-:=\frac{1}{2}(1-j)\cdot M.$$
Observe that this is a $\Bbb Z_p[[\mathcal G]]^-$--module, 
where 
$$\Bbb Z_p[[\mathcal G]]^-:=\frac{1}{2}(1-j)\cdot\Bbb Z_p[[\mathcal G]]\cong \Bbb Z_p[[\mathcal G]]/(1+j).$$ 
The same notations apply to $\Bbb Z_p[G_n]$--modules, for all $n\geq 0$. The functor $M\to M^{-}$ is exact in the category of modules over any of these $p$--adic group rings. 
Further, in order to simplify notations, we let $$N_p:=N\otimes_{\Bbb Z}\Bbb Z_p,$$ for any $\Bbb Z$--module $N$.  
\\

As discussed in the introduction, the so--called main conjectures in Iwasawa theory relate algebraic objects to $p$--adic analytic objects. Next, we recall the definitions of the relevant analytic objects ($L$--functions). The reader can consult \cite{gambheera-popescu} for more details. We use the same notations as in loc.cit.
For a place $v$ of $F$, we let $G_v$ and $I_v$ denote its decomposition and inertia groups in $G$, respectively, and fix $\sigma_v\in G_v$ a Frobenius element for $v$.
The idempotent element associated to the trivial character of $I_v$ in $\Bbb Q[I_v]$ is given by
$$e_v:=\frac{1}{|I_v|} N_{I_v}:=\frac{1}{|I_v|}\sum_{\sigma\in I_v}\sigma.$$
Then $e_v\sigma_v^{-1}\in \Bbb Q[G]$ is independent of the choice of $\sigma_v$.
As in \cite{Dasgupta-Kakde}, the $\Bbb C[G]$--valued ($G$--equivariant, $S$--depleted, $T$--smoothed) Artin $L$--function associated to $(H/F, S, T)$ of complex variable $s$ is given by the meromorphic continuation to the entire complex plane of the following absolutely and compact--uniformly convergent infinite Euler product
$$\Theta_{S, H/F}^T(s):=\prod_{v\not\in S}(1-e_v\sigma_v^{-1}\cdot Nv^{-s})^{-1}\cdot\prod_{v\in T}(1-e_v\sigma_v^{-1}\cdot Nv^{1-s}), \qquad \text{Re}(s)>0.$$
The resulting meromorphic continuation, also denoted by $\Theta_{S, H/F}^T(s)$, is holomorphic on $\Bbb C\setminus\{1\}$. We are interested in its special value at $0$, denoted by $$\Theta_S^T(H/F):=\Theta_{S, H/F}^T(0).$$
The same construction applies to the extension $H_n/F$ and leads to elements $\Theta_S^T(H_n/F)$ in the complex group rings $\Bbb C[G_n]$, for all $n\geq 0$.
An important Theorem due independently to Deligne--Ribet and Pi. Cassou--Nogu\'es, combined with a Lemma of Kurihara, imply that, under hypotheses \eqref{hyp-ST} on our sets $S$ and $T$, we have the remarkable $p$--adic integrality phenomenon:
$$\Theta_S^T(H_n/F)\in \mathbb{Z}_{(p)}[G_n]^-\subseteq\mathbb Z_p[G_n]^-,\qquad\text{ for all }n\geq 0.$$
(See \cite{gambheera-popescu}, Lemma 2.5 and the preceding paragraph.)
Now, due to a well known restriction--inflation property of Artin $L$--functions, we have 
$$\pi_n(\Theta_S^T(H_{n+1}/F)=\Theta_S^T(H_n/F), \qquad \text{ for all } n\geq 0.$$
This property allows us to define the $(S,T)-$modified equivariant $p-$adic L-function:
$$\Theta_S^T(H_{\infty}/F):=(\Theta_S^T(H_n/F))_n\,\in\, \varprojlim\limits_n\Bbb Z_p[G_n]^-=\mathbb{Z}_p[[\mathcal{G}]]^-,$$
which will play a very important role in the considerations which follow.

\subsection{The Iwasawa--Ritter--Weiss Module.}
In this section we define the Iwasawa-Ritter-Weiss module $\nabla_S^T(H_\infty)_p$, at the infinite level $H_\infty$ of the cyclotomic Iwasawa tower $H_\infty/H$, for the data $(H/F, S, T, p)$. We work with the notations and under the hypotheses of the previous subsection.

\begin{definition}[the Iwasawa--Ritter--Weiss module]
Define the $\Bbb Z_p[[\mathcal G]]$--module
$$\nabla_S^T(H_{\infty})_p:=\underset{n}{\varprojlim}\text{ }\nabla_S^T(H_n)_p,$$ 
where the transition maps $\lambda_{n,p}: \nabla_S^T(H_{n+1})_p\to \nabla_S^T(H_{n+1})_p$ in the projetive limit are the maps $\lambda\otimes{\mathbf{1}}_{\Bbb Z_p}$, with $\lambda$ is given by Theorem \ref{transition} 
applied to the field extension $H_{n+1}/H_n$.
\end{definition}
\begin{definition}
Let $A_S^T(H_n):=Cl_S^T(H_n)_p$ and define the $\Bbb Z_p[[\mathcal G]]$--module
$$X_S^T:=\underset{n}{\varprojlim}\text{ }A_S^T(H_n),$$ where the projective limit is taken with respect to the norm maps at the level of ideal class--groups. 
\end{definition}
\begin{proposition}\label{SES}
We have a short exact sequence of $\mathbb{Z}_p[[\mathcal{G}]]^-$ modules,
$$0\xrightarrow[]{} X_S^{T,-}\xrightarrow[]{} \nabla_S^T(H_{\infty})_p^-\xrightarrow[]{} {\rm Div}_S(H_{\infty})_p^-\xrightarrow[]{} 0,$$
where ${\rm Div}_S(H_\infty)$ is the free $\Bbb Z$--module of divisors of $H_\infty$, supported at the primes above those in $S\setminus S_\infty$.
\end{proposition}
\begin{proof}
When tensoring with $\Bbb Z_p$ the diagram in Theorem \ref{transition} applied to the field extension $H_{n+1}/H_n$, followed by applying the $(\ast\to \ast^{-})$--functor, we obtain commutative diagrams of compact $\Bbb Z_p[[\mathcal G]]^-$--modules 
 \begin{equation}\label{lambda-diagram}\begin{tikzcd}
0\arrow{r} & A_S^T(H_{n+1})^-\arrow{r}\arrow{d}{N_{n+1/n}} & \nabla_S^T(H_{n+1})_p^-\arrow{r}\arrow{d}{\lambda_{n, p}} & {\rm Div}^0_S(H_{n+1})_p^-\arrow{r}\arrow{d}{\widetilde{\pi}_n} & 0 \\%
0\arrow{r} & A_S^T(H_{n})^-\arrow{r} & \nabla_S^T(H_{n})_p^-\arrow{r} & {\rm Div}^0_S(H_n)_p^-\arrow[r] & 0
\end{tikzcd}
\end{equation}
with exact rows, for every $n\geq 0$. Here, compactness is viewed in the $p$--adic topology and follows from the fact that all the modules involved are finitely generated over $\Bbb Z_p$. The vertical maps in the commutative diagram above are precisely defined in Theorem \ref{transition}. In particular, $N_{n+1/n}$ denote the norm maps at the level of ideal classes. 
Now, observe that we have equalities
$${\rm Div}_S^0(H_n)_p^-={\rm Div}_S(H_n)_p^-,$$
for all $n\geq 0$. Since when taking projective limits of exact sequences in the category of topological, compact $\Bbb Z_p[[\mathcal G]]^-$--modules, exactness is preserved (see \textcolor{red}{\cite{Washington}}) and
$$\varprojlim\limits_n{\rm Div}_S(H_n)_p^-={\rm Div}_S(H_\infty)_p^-$$
(as a consequence of the definition of the maps $\lambda_{n,p}$), the statement of the Proposition follows by taking a projective limit with respect to $n$ of the exact sequences above.\end{proof}
\begin{lemma}\label{surjective-lambda}
The maps $\lambda_{n,p}:\nabla_S^T(H_{n+1})_p^-\to \nabla_S^T(H_{n})_p^-$ are surjective, for all $n\gg0$.
\end{lemma}
\begin{proof} By class field theory, for all $i\geq0$, the class--group $A_S^T(H_{i})$ is isomorphic (via the Artin reciprocity map) to the Galois group of the maximal abelian pro--$p$ extension $\mathcal H_{i, S}^T$ of $H_i$, which is split at all primes in $S_{H_i}$, at most tamely ramified at all primes in $T_{H_i}$ and unramified outside of $T_{H_i}$. Now, since for all $n\gg 0$ the extensions $H_{n+1}/H_n$ are wildly and totally ramified at all $p$--adic primes and since $S_p\subseteq S\cap T$, we have 
$$\mathcal H_{n, S}^T\cap H_{n+1}=H_n, \qquad \text{ for all }n\gg 0.$$
By the compatibility of Artin maps with the norm maps at the level of ideal classes, this implies that the norm maps $N_{n+1/n}: A_S^T(H_{n+1})\to A_S^T(H_n)$ are surjective, for all $n\gg 0$. Now, since the maps $\widetilde{\pi}_n$ at the level of divisors are surjective by definition, the snake lemma applied to the commutative diagram \eqref{lambda-diagram} implies that the maps $\lambda_{n,p}$ are surjective, for all $n\gg 0$.
\end{proof}

Next, we refer to the considerations in \S\ref{Fitt-proj-section}, applied to our current abelian extension $H_\infty/H/F$. As in loc.cit., we pick a subgroup $\Gamma\subseteq\mathcal G$, such that the group isomorphisms \eqref{group-iso} hold. As in loc.cit., we denote by $\Lambda:=\Bbb Z_p[[\Gamma]].$ Since $p>2$, we have $j\in G'$ and therefore
$\Bbb Z_p[[\mathcal G]]^-=\Lambda[G_{[p]}\times G']/(1+j)$ contains $\Lambda$ as a subring and it is a finitely generated as a $\Lambda$--module.

\begin{proposition}\label{torsion}
With notations as above, $\nabla_S^T(H_{\infty})_p^-$ is a finitely generated, torsion $\Lambda$--module. 
\end{proposition}
\begin{proof}
According to Proposition \ref{SES}, the statement in the proposition is equivalent to ${\rm Div}_S(H_\infty)^-$ and $X_S^{T, -}$ being finitely generated and torsion as $\Lambda$--modules. 

Since $S(H_\infty)$ (the set of primes in $H_\infty$ sitting above primes in $S$) is a finite set, ${\rm Div}_S(H_\infty)^-$ is finitely generated as a $\Bbb Z_p$--module and therefore finitely generated and torsion as a $\Lambda$--module.

In order to deal with $X_S^{T, -}$, let $H':=H^{\Gamma}$ be the fixed field under the action of $\Gamma$. Observe that $H_\infty$ is also the cyclotomic $\Bbb Z_p$--extension of $H'$. It is not hard to see that $X_S^{T, -}$ can be constructed via the same procedure as before, starting with $H'$ rather than $H$ as a base field of the cyclotomic tower. Now, let us allow $S$ and $T$ to be two arbitrary finite, disjoint sets of primes in $F$, possibly empty, and apply the $X_S^{T,-}$ construction to this general situation and the Iwasawa tower $H_\infty/H'$.  From the definitions, one obtains an exact sequence and a surjective map of $\Lambda$--modules  
\[\begin{tikzcd}
& & & X_{\emptyset}^{\emptyset, -}\arrow[d, two heads] & \\%
& \Bbb Z_p(1)^{\delta_T}\arrow[r] & X_S^{T,-} \arrow[r] & X_S^{\emptyset,-}\arrow[r] & 0
\end{tikzcd}
\]
for some $\delta_T\in\Bbb Z_{\geq 0}$. (See \cite{Greither-Popescu} for more details on the exact sequence above.) Now, a fundamental result of Iwasawa shows that $X^{\emptyset, -}_{\emptyset}$ (which is just the classical, unramified Iwasawa module for $H_\infty/H'$) is a finitely generated, torsion $\Lambda$--module. This implies that $X_S^{\emptyset, -}$ is finitely generated and torsion over $\Lambda$. Since so is $\Bbb Z_p(1)^{\delta_T}$ (as a finitely generated module over $\Bbb Z_p$), the bottom exact sequence above gives the desired result for $X_S^{T,-}$. This concludes the proof of the proposition.
\end{proof}

The following Proposition expresses the  Fitting ideals of the Ritter-Weiss modules at the finite levels $H_n/F$ in terms of special values of the appropriate equivariant Artin $L$--values. 
\begin{proposition}\label{Fitt-trans}
Under the above hypotheses for $(H/F, S, T, p)$, we have equalities of $\Bbb Z_p[G_n]^-$--ideals  
$${\rm Fitt}_{\mathbb{Z}_p[G_n]^-}(\nabla_S^T(H_n)_p^-)=(\Theta_S^T(H_n/F)), \qquad \text{ for all }n.$$
\end{proposition}
\begin{proof} In Theorem 6.4 of \cite{gambheera-popescu}, we proved that 
$${\rm Fitt}_{\mathbb{Z}_p[G_n]^-}({\rm Sel}_S^T(H_n)_p^-)=(\Theta_S^T(H_n/F)), \qquad \text{ for all }n.$$
The statement in the proposition follows from the equality above, combined with Theorem \ref{nabla-sel-theorem}, applied to the data $(H_n/F, G_n, p,S, S', T,  R_n:=\Bbb Z_p[G_n]^-)$, for a set $S'$ satisfying the appropriate hypotheses. Note that since $S_p\cap T=\emptyset$, the algebra $R_n$ satisfies the hypotheses in Proposition \ref{nabla-quad} for the extension $H_n/F$.
\end{proof}
Now, we are ready to state our main Iwasawa theoretic results. In order to simplify notations, we let $S_{ram}:=S_{ram}(H_\infty/F)$, an $S_\infty$ and $S_p$ be the sets of infinite, respectively $p$--adic primes in $F$.

\begin{theorem}\label{EMC}
Let $(H_\infty/H/F, p, S, T)$ be as above. Suppose that $S$ and $T$ satisfy the following. 

\begin{equation*} S_{\infty}\cup S_p\subseteq S,\quad  S_{ram}\subseteq S\cup T,\qquad  T\not\subseteq S_{ram}\end{equation*} Then, the following hold.
\begin{enumerate}
\item We have an equality of $\Bbb Z_p[[\mathcal G]]^-$--ideals
$${\rm Fitt}_{\mathbb{Z}_p[[\mathcal{G}]]^-}(\nabla_S^T(H_{\infty})_p^-)=(\Theta_S^T(H_\infty/F)).$$
\item The $\Bbb Z_p[[\mathcal G]]^-$--module $Sel_S^T(H_{\infty})_p^-$ sits in a short exact sequence
$$0\longrightarrow(\Bbb Z_p[[\mathcal G]]^-)^k\longrightarrow (\Bbb Z_p[[\mathcal G]]^-)^k\longrightarrow \nabla_S^T(H_{\infty})_p^-\longrightarrow 0,$$
for some $k>0$. In particular, ${\rm pd}_{\Bbb Z_p[[\mathcal G]]^-} (\nabla_S^T(H_{\infty})_p^-)=1.$
\end{enumerate}
\end{theorem}
\begin{proof}
(1) By Lemma \ref{surjective-lambda} and Proposition \ref{torsion}, the projective system 
$(\nabla_S^T(H_n)^-)_n$ of $\Bbb Z_p[G_n]^-)_n$--modules satisfies the hypotheses of Corollary \ref{GK}. Therefore, we have an equality
$${\rm Fitt}_{\mathbb{Z}_p[[\mathcal{G}]]^-}(\nabla_S^T(H_{\infty})_p^-)=\varprojlim\limits_n{\rm Fitt}_{\mathbb{Z}_p[G_n]^-}(\nabla_S^T(H_{n})_p^-).$$
Now, part (1) follows from the above equality and Proposition \ref{Fitt-trans}.
\\

(2) In \cite{gambheera-popescu} (see Proposition 4.6 in loc.cit.) we proved that, under the current hypotheses, $\Theta_S^T(H_\infty/F)$ is a non zero--divisor in $\Bbb Z_p[[\mathcal G]]^-$.
Now, the existence of the exact sequence in part (2) follows from this fact combined with part (1) and Proposition 4.9 in \cite{gambheera-popescu}. This concludes the proof of part (2).
\end{proof}
\medskip

Next, we will use the above results to compute the Fitting ideal over $\Bbb Z_p[[\mathcal G]]^-$ of the $(S,T)$--modified Iwasawa module $X_S^{T,-}$, which is finitely generated and torsion, as a submodule of $\nabla_S^T(H_\infty)_p^-$. In what follows, we let $S_f:=S\setminus S_\infty$ be the set of finite places in $S$.
\begin{theorem}
Suppose $(H_\infty/H/F,S,T,p)$ are as above. Then, the following equality holds.

\begin{equation*}
   \begin{split}
       {\rm Fitt}_{\mathbb{Z}_p[[\mathcal{G}]]^-}(X_S^{T,-})=(\Theta^T_{S\cap S_{ram}}(H_{\infty}/F))\cdot \prod_{v\in (S_f\cap S_{ram})\setminus S_p}\left(1,\frac{N(\mathcal I_v)}{\sigma_v-1}\right)\cdot  
    \prod_{v\in S_p}{\rm Fitt}_{\mathbb{Z}_p[[\mathcal{G}]]^-}^{[1]}(\mathbb{Z}_p[\mathcal{G}/\mathcal{G}_v]^-)
   \end{split}
\end{equation*}
Here, $\mathcal{G}_v$ and $\mathcal I_v$ are the decomposition and inertia groups of the prime $v$ in the extension $H_{\infty}/F$. Further, $N(\mathcal I_v):=\sum_{g\in\mathcal I_v}g$ and $\sigma_v$ is any choice of Frobenius for $v$ in $\mathcal G_v$.
\end{theorem}
\begin{proof}
First, note that we have $\Bbb Z_p[[\mathcal G]]^-$--module isomorphisms 
$${\rm Div}_S(H_\infty)_p^-=\bigoplus_{v\in S}(\Bbb Z_p[[\mathcal G]]w)^-\simeq \bigoplus_{v\in S_f}\Bbb Z_p[\mathcal G/\mathcal G_v]^-,$$
where $w$ is a fixed prime in $H_\infty$ sitting over $v$. (To see this, note that $\mathcal G_v$ is an open subgroup of $\mathcal G$, for every $v\in S_f$. Also note that $\Bbb Z_p[[\mathcal G/\mathcal G_v]]^-=0$, for every $v\in S_\infty$, as $j\in\mathcal G_v$ in that case.)

Now, combine the above isomorphism with the exact sequence in Proposition \ref{SES}, with Proposition \ref{Shifted-Fitting} and Theorem \ref{EMC} to obtain the following. 
\begin{equation}\label{one}\begin{aligned}
{\rm Fitt}_{\mathbb{Z}_p[[\mathcal{G}]]^-}(X_S^{T,-})&= {\rm Fitt}_{\mathbb{Z}_p[[\mathcal{G}]]^-}(\nabla_S^T(H_\infty)_p^-)\cdot {\rm Fitt}_{\mathbb{Z}_p[[\mathcal{G}]]^-}^{[1]}(\rm Div_S(H_\infty)_p^-)\\
&=(\Theta^T_{S}(H_{\infty}/F))\cdot \prod_{v\in S_f}{\rm Fitt}_{\mathbb{Z}_p[[\mathcal{G}]]^-}^{[1]}(\mathbb{Z}_p[\mathcal{G}/\mathcal{G}_v]^-)\end{aligned}\end{equation}
When applying Proposition \ref{Shifted-Fitting} above, it is important to keep in mind that all the $\Bbb Z_p[[\mathcal G]]^-$--modules in the exact sequence of Proposition \ref{SES} are finitely generated and torsion, by Proposition \ref{torsion}.
\\

Now, note that by the definition of the equivariant Artin $L$--functions we have 
$$\Theta^T_{S}(H_{n}/F)=\Theta^T_{S\cap S_{ram}}(H_{n}/F)\cdot \prod_{v\in S\setminus S_{ram}}(1-\sigma_v^{-1}),$$
for all $n\geq 0$. By taking a projective limit,  we obtain 
\begin{equation}\label{two}\Theta^T_{S}(H_{\infty}/F)=\Theta^T_{S\cap S_{ram}}(H_{\infty}/F)\cdot \prod_{v\in S\setminus S_{ram}}(1-\sigma_v^{-1}).\end{equation}
For primes $v\in S\setminus S_{ram}$, we have an isomorphism 
of $\Bbb Z_p[[\mathcal G]]^-$--modules
$$\Bbb Z_p[\mathcal G/\mathcal G_v]^-\simeq\Bbb Z_p[\mathcal G]^-/(1-\sigma_v^{-1}).$$
Since $(1-\sigma_v^{-1})$ is obviously a non zero--divisor in $\Bbb Z_p[[\mathcal G]]$, the modules in the isomorphism above are torsion, of projective dimension 1. Therefore, by definition,  we have
\begin{equation}\label{three}{\rm Fitt}_{\mathbb{Z}_p[[\mathcal{G}]]^-}^{[1]}(\mathbb{Z}_p[\mathcal{G}/\mathcal{G}_v]^-)=\left(\frac{1}{1-\sigma_v^{-1}}\right).\end{equation}
Moreover, by Proposition 1.8 of \cite{Greither-Kataoka-Kurihara}, for finite primes $v\in S_{ram}\setminus S_p$, we have  
\begin{equation}\label{four}{\rm Fitt}_{\mathbb{Z}_p[[\mathcal{G}]]^-}^{[1]}(\mathbb{Z}_p[\mathcal{G}/\mathcal{G}_v]^-)=\left(1,\frac{N(I_v)}{1-\sigma_v}\right).\end{equation}
The equality in the Theorem follows by combining \eqref{one}, \eqref{two}, \eqref{three} and \eqref{four}. 
\end{proof}
As a consequence of the theorem above and the explicit computations done in the Appendix, we have the following main theorem.
\begin{theorem}\label{S,T-modified} Let $(H_\infty/H/F, S, T, p)$ be as in the previous Theorem. Then, we have the following.
$$Fitt_{\mathbb{Z}_p[[\mathcal{G}]]^-}(X_S^{T,-})=(\Theta_T^{S\cap S_{ram}}(H_{\infty}/F))\prod_{v\in (S_f\cap S_{ram})}(N(A)\Delta B^{r_B-2}\text{ ; } \mathcal{G}_v=A\times B ,  A- torsion)$$
where $\mathcal{G}_v$ is the decomposition group of the prime $v$ in the extension $H_{\infty}/F$, for any finite group $A$, $N(A)=\sum_{g\in A}g$ and $r_B$ is the minimum number of generators of $B$. Moreover, for each prime $v\in (S\cap S_{ram})$, $A$ and $B$ varies through all the possibilities such that $A$ is torsion and $\mathcal{G}_v=A\times B$.
\end{theorem}
\begin{proof}
This is an easy consequence of the above theorem and Theorem \ref{shifted-Fitt-intrinsic} in the Appendix.
\end{proof}

\begin{remark} Observe that the right hand side of the equality in the statement of the last Theorem does not depend on the primes in $S\setminus S_{ram}.$ It is an easy exercise in class field theory to show that the module $X_S^T$ itself (therefore the left hand side of the equality) is independent of the primes in $S\setminus S_{ram}.$
\end{remark} 
\bigskip

\section{On the Equivariant Tamagawa Number Conjecture}

Let $H$ be a finite, abelian CM extension of a totally real number field $F$, and let $G:=G(H/F):=Gal(H/F)$. In this section, we use the results of \S4 to give an Iwasawa theoretic proof of the minus part of the Equivariant Tamagawa Number Conjecture for the Artin motive with $\Bbb Z[G]$--coefficients associated to $H/F$, away from the prime $2$. Proofs of this statement were given independently by Bullach--Burns--Daoud--Seo \cite{Bullach-Burns-Daoud-Seo}, away from 2, and by Dasgupta--Kakde--Silliman \cite{Dasgupta-Kakde-Silliman-ETNC}, at all primes. Both of these proofs rely on a statement about the scarcity of Euler Systems. Our proof does not involve Euler Systems. Instead, it uses our Theorem \ref{EMC} above and Iwasawa co-descent, followed by an application of the method of Taylor--Wiles primes.
\\

Assume that the data $(H/F, S, T, S')$ satisfy the hypotheses $P(H/F, S, T, S')$ listed after Definition \ref{sel-def} and, in addition, assume that $S_{ram}(H/F)\subseteq S$. Let us denote this set of slightly stronger hypotheses by $\overline P(H/F, S, T, S').$ It turns out that for all intermediate fields $F\subseteq E\subseteq H$, we have an implication
\begin{equation}\label{larger-implication}\overline P(H/F, S, T, S') \longrightarrow \overline P(E/F, S, T, S').\end{equation} (See \cite{Weiss}, Chapter 2, Lemma 1.) 

By Proposition \ref{nabla-quad} and Remark \ref{freeness-remark}, we have the following canonical quadratic presentation of rank $(|S'|-1)$ for the minus part of the associated Ritter--Weiss $\Bbb Z[G]^-$--module:
$$V_{S'}^T(H)^{\theta, -}\xrightarrow{f_{S, S'}^T(H)} B_{S'}(H)^{\theta, -}\xrightarrow{} \nabla_S^T(H)^{-}\xrightarrow{} 0.$$
Here, we have decorated the first two modules and the map between them to emphasize their dependence on $S'$, $S$ and $T$, respectively. In order to simplify the notations, in what follows we will set
$$\overline{V}_{S'}^T(H):=V_{S'}^T(H)^{\theta, -}, \qquad \overline{B}_{S'}(H):=B_{S'}(H)^{\theta, -}.$$
In what follows, we fix once and for all an infinite place $v_\infty\in S_\infty(F)$ and view the free modules $\overline{B}_{S'}(H)$ as endowed with their standard bases with respect to this choice. (See Remark \ref{freeness-remark}.) So, if one picks a basis for the $\Bbb Z[G]^-$--module $\overline{V}_{S'}^T(H)$, one can talk about the determinant ${\rm det}(f_{S,S'}^T(H))\in \Bbb Z[G]^-$. Since the bases for the modules $\overline V$ are not canonical, three important remarks are in order.

\begin{remark}\label{enlarging S'} It is not difficult to show that if $S'\subseteq S''$ and properties $\overline P(H/F, S, T, S')$ and $\overline P(H/F, S, T, S'')$ are satisfied, then there is a canonical extension of any $\Bbb Z[G]^-$--basis of $\overline{V}_{S'}^T(H)$ to a $\Bbb Z[G]^-$--basis of the larger module $\overline{V}_{S''}^T(H)$, such that with respect to those bases we have
$${\rm det}(f_{S,S'}^T(H))={\rm det}(f_{S,S''}^T(H)).$$
(See the proof of Lemma A3 in \cite{Dasgupta-Kakde}.) Incidentally, the standard basis of $\overline B_{S'}(H)$ canonically extends the standard basis of the larger module $\overline B_{S''}(H)$. (See loc. cit.)
\end{remark}

\begin{remark}\label{coinvariants-remark}
Assume that $F\subseteq H\subseteq \tilde H$, with $\tilde H$ CM, $\tilde H/F$ abelian of Galois group $\tilde G$, such that hypotheses  $\overline P(H/F, S, T, S')$ and 
$\overline P(\tilde H/F, S, T, S')$ are satisfied. Denote by $X:=G(\tilde H/H)$ and by $M_X$ the module of $X$--coinvariants, for any $\Bbb Z[\tilde G]$--module $M$. Then, there is a canonical commutative diagram of $\Bbb Z[\tilde G]$--modules
\[
\begin{tikzcd}   
     \overline V_{S'}^T(\tilde H)\quad \arrow{r}{\quad f_{S, S'}^T(\tilde H)\quad}\arrow[two heads]{d} & \quad \overline B_{S'}(\tilde H)\arrow[two heads]{d}\\%
\overline V_{S'}^T(\tilde H)_X \quad \arrow{r}{\quad f_{S, S'}^T(\tilde H)_X\quad }\arrow{d}{\psi}[swap]{\simeq} & \quad \overline B_{S'}(\tilde H)_X\arrow{d}{\phi}[swap]{\simeq}\\%
\overline V_{S'}^T(H) \quad \arrow{r}{\quad f_{S,S'}^T(H)\quad } & \quad \overline B_{S'}(H)   
\end{tikzcd}
 \]
where the vertical surjective arrows are the canonical projections $M\to M_X$, $\phi$ is induced by the canonical isomorphism $\Bbb Z[\tilde G]_X\simeq \Bbb Z[G]$ given by Galois restriction and $\psi$ is a canonical (norm--like) isomorphism. (See \cite{Dasgupta-Kakde}, Lemmas B.1, B.2 and B.3) In particular, this shows that any $\Bbb Z[\tilde G]^-$--basis for $\overline V_{S'}^T(\tilde H)$ canonically induces a $\Bbb Z[G]^-$--basis for $\overline V_{S'}^T(H)$ and, with respect to these bases, one has an equality
$$\pi_{\tilde H/H}(det(f_{S,S'}^T(\tilde H)))=det(f_{S, S'}^T(H)),$$ where $\pi_{\tilde H/H}: \mathbb{Z}[\tilde G]\xrightarrow{}\mathbb{Z}[G]$ is the algebra morphism induced by Galois restriction.
\end{remark}

\begin{remark}\label{enlarging T} Assume that $F\subseteq E\subseteq H$, with $E$ CM, and consider sets of $F$--primes $S, S', T\subseteq T'$, such that 
$\overline P(H/F, S, T, S')$ and $\overline P(H/F, S, T', S')$ are satisfied. Then, if one picks $\Bbb Z[G]^-$ bases for $\overline V_{S'}^T(H)$ and $\overline V_{S'}^{T'}(H)$, there exists $x_H\in(\Bbb Z[G]^-)^\times$, independent of $E$, such that 
$$det(f_{S,S'}^{T'}(E))=\pi_{H/E}(x_H)\cdot det(f_{S,S'}^{T}(E))\cdot\prod_{v\in T'\setminus T}(1-\sigma_v^{-1}Nv),$$
where the determinants above are computed with respect to the canonically induced bases for $\overline V_{S'}^{T'}(E)$ and $\overline V_{S'}^{T}(E)$ from any chosen bases for  $\overline V_{S'}^{T'}(H)$ and $\overline V_{S'}^{T}(H)$, respectively. (See \cite{Dasgupta-Kakde-Silliman-ETNC}, Proposition 14(3).)
\end{remark}

Of course, all the remarks above remain valid when tensoring the corresponding modules with $\Bbb Z_p$, for an odd prime $p$. The next Proposition will play an important role in the considerations which follow.
\begin{proposition}\label{basis-aligning} With the notations and under the assumptions in the previous Remark, let $p>2$ be a prime number and assume further that the order of $X$ is a power of $p$.
     Then, for a fixed $\Bbb Z_p[G]^-$--basis of $\overline V_{S'}^T(H)_p$, there exists a $\Bbb Z_p[G]^-$--basis for $\overline V_{S}^T(\tilde H)_p$, such that, with respect to these bases, we have 
    $$\pi_{\tilde H/H}(det(f_{S,S'}^T(\tilde{H})_p))=det(f_{S,S'}^T(H)_p).$$
\end{proposition}
\begin{proof} First, tensor the commutative diagram in the previous remark with $\Bbb Z_p$ to get a similar commutative diagram of $\Bbb Z_p[\tilde G]^-$--modules.
  Now, given a fixed basis $\bf v$ for $\overline V_{S'}^T(H)_p$, we start by picking an arbitrary basis $\bf{\tilde v}$ for $\overline V_{S'}^T(\tilde H)_p$. Via the left vertical morphisms in the commutative diagram above, $\bf{\tilde v}$ induces a basis $\bf v'$ for $\overline V_{S'}^T(H)_p$. Since the right vertical morphisms carry canonical basis into canonical basis, we have equalities 
  $$\pi_{\tilde H/H}(det_{\bf{\tilde v}}(f_{S, S'}^T(\tilde H)_p))=det_{\bf v'}(f_{S,S'}^T(H)_p)=\alpha\cdot det_{\bf v}(f_{S,S'}^T(H)_p),$$
  where the determinant subscript meaning is the obvious one and $\alpha\in(\Bbb Z_p[G]^-)^\times.$
 However, since $X$ is a $p$--group, by Lemma \ref{unit-lifting-lemma} we can find a lift $\Tilde{\alpha}\in \mathbb{Z}_p[\tilde G]^{-, \times}$ of $\alpha$. Now, a simple change of the basis $\bf{\tilde v}$ for $\overline V_{S}^T(\tilde H)_p$ of determinant $\Tilde{\alpha}^{-1}$ leads to a basis which satisfies the desired equality.
\end{proof}

For any CM extension $K$ of $F$, let $S(K)=S_{\infty}\cup S_{ram}(K/F)$. If $p>2$ is a fixed prime, we let 
$$S^*(K):=S(K)\cup S_p(F).$$ The main goal of this section is to give a proof of the following theorem, which is Kurihara's equivalent formulation  of the minus part of the Equivariant Tamagawa Number Conjecture given in \cite{Kurihara}, 
originally formulated by Burns and Flach in \cite{Burns-Flach}.

\begin{theorem}[${\rm ETNC}(H/F)^-$]\label{ETNC} Let $S'$ and $T$ be two sets of $F$--primes such that $\overline P(H/F, S(H), T, S')$ are satisfied.
    Fix a $\Bbb Z[G]^-$--basis for $\overline V_{S'}^T(H)$. Then, there exists an element $y_H\in (\mathbb{Z}[G]^-)^{\times}$ such that, for all intermediate CM fields $E$ with $F\subseteq E\subseteq H$, we have an equality
    $$det(f_{S(E), S'}^T(E))=\pi_{H/E}(y_H)\cdot\Theta_{S(E)}^T(E/F),$$ where $\pi_{H/E}:\Bbb Z[G(H/F)]^-\to\Bbb Z[G(E/F)]^-$ is given by the Galois restriction and 
    the determinant on the left is taken with respect to the canonically induced $\Bbb Z[G(E/F)]^-$--basis of
    $\overline V_{S'}^T(E)$.
\end{theorem}

We start by making two simple but very useful remarks.

\begin{remark}\label{independence on S' and T} By Remark \ref{enlarging S'}, the statement above is independent of $S'$, as long as 
properties $\overline P(H/F, S(H), T, S')$ are satisfied. Remark \ref{enlarging T} implies that the statement is also independent of $T$. Indeed, if $T\subseteq T'$, then 
$$\Theta_{S(E)}^{T'}(E/F)=\prod_{v\in T'\setminus T}(1-\sigma_v^{-1}Nv)\cdot\Theta_{S(E)}^{T}(E/F) .$$
Since each factor $(1-\sigma_v^{-1}Nv)$ is a non--zero divisor in $\Bbb Z[G(E/F)]$, the displayed formula in Remark \ref{enlarging T} implies independence of $T$ indeed.
\end{remark}

\begin{remark}\label{ETNC p-adic enough}
It is not hard to show that if one proves the above statement tensored with $\Bbb Z_p$ (call the ensuing $p$--adic statement $ETNC_p(H/F)^-$) for each odd prime $p$, then the global statement above follows. (See \cite{Dasgupta-Kakde-Silliman-ETNC}, \S1.2, Lemma 6.) So, we first prove the following $p$--imprimitive variant of the $ETNC_p^-(H/F)$.
\end{remark}

\begin{proposition}[$ETNC^\ast_p(H/F)^-$]\label{ETNC-S*} Fix a prime $p>2$ and assume that $S'$ and $T$ are chosen so that $\overline P(H/F, S^\ast(H), T, S')$ are satisfied. Then
    $ETNC_p(H/F)^-$ is true when $S(E)$ is replaced by $S^\ast(E)$, for all intermediate CM fields $E$.
\end{proposition}
\begin{proof}
Let $H_{\infty}/H$ be the cyclotomic $\mathbb{Z}_p-$extension of $H$ and let $H_n$ be the $n-$th intermediate layer, where $H_0=H$. Observe that $S^*(H_n)=S^*(H)$, for all $n$. Start by fixing a $\Bbb Z_p[G(H_0/F)]^-$--basis $\bf v_0$ of $\overline V_{S'_0}^T(H_0)_p$, where $S'_0:=S'$. Starting with this data and using Remark \ref{enlarging S'}, Remark \ref{coinvariants-remark} and Proposition \ref{basis-aligning}, we construct inductively on $n$ sets $S'_{n}$ of $F$--primes and bases $\bf v_n$ for the $\Bbb Z_p[G(H_n/F)]^-$--modules $\overline V_{S'_n}^T(H_n)_p$, such that $S'_n\subseteq S'_{n+1}$, the properties $\overline P(H_n/F, S(H), T, S'_n)$ are satisfied and 
$$\pi_{H_{n+1}/H_n}(det_{\bf v_{n+1}}(f_{S^*(H), S'_{n+1}}^T(H_{n+1})_p))=det_{\bf v_n}(f_{S^*(H), S'_n}^T(H_{n})_p),
$$
for all $n\geq 0$. The reader should note that the inductive process is a bit subtle: once $S'_n$ and $\bf v_n$ are constructed, then one chooses $S'_{n+1}$, such that 
$S'_n\subseteq S'_{n+1}$ and $\overline P(H_{n+1}/F, S^\ast(H), T, S'_{n+1})$ are satisfied, then one extends $\bf v_n$ to a basis $\tilde {\bf v}_n$ of 
$\overline V_{S'_{n+1}}^T(H_n)_p$ as in Remark \ref{enlarging S'}, and finally one lifts $\tilde{\bf v}_n$ to a basis $\bf v_{n+1}$
of $\overline V_{S'_{n+1}}^T(H_{n+1})_p$ as in Proposition \ref{basis-aligning}. 

The last displayed equalities  allow us to define the following element in $\Bbb Z_p[[\mathcal G]]^-$
 $$det(f_{S^*(H)}^T(H_{\infty})_p):= (det_{\bf v_n}(f_{S^*(H), S'_n}^T(H_{n})_p))_n.$$
 However, for each $n$, we also have equalities of $\Bbb Z_p[G(H_n/F)]^-$--ideals
 $$Fitt_{\mathbb{Z}_p[G(H_n/F)]^-}(\nabla_{S^*(H), S'_n}^T(H_n)^-_p)=(det_{\bf v_n}(f_{S^*(H), S'_n}^T(H_{n})_p)).$$
 Therefore, when taking projective limits, an argument identical to the proof of Theorem \ref{EMC} (1), gives $${\rm Fitt}_{\mathbb{Z}_p[[\mathcal{G}]]^-}(\nabla_{S^\ast(H)}^T(H_{\infty})_p^-)=(det(f_{S^*(H)}^T(H_{\infty})_p)).$$
 Now, together with Theorem \ref{EMC} (1) and Proposition 4.6 of \cite{gambheera-popescu} this gives an equality
 $$det(f_{S^*(H)}^T(H_{\infty})_p)=y_{\infty}\cdot \Theta_{S^*(H)}^T(H_\infty/F),$$ where $y_{\infty}\in (\mathbb{Z}_p[[G(H_\infty/F)]]^-)^{\times}$.
 We let $$y_H:=\pi_{H_{\infty}/H}(y_{\infty}),$$ 
 and will prove next that $y_H$ satisfies the properties in the statement of $ETNC^\ast_p(H/F)^-$. \\
 
 Let $E$ be an intermediate CM field, $F\subseteq E\subseteq H$. Let $E_{\infty}$ be its cyclotomic $\mathbb{Z}_p-$extension whose $n$--th layer is denoted by $E_n$. Then, when applying implication \eqref{larger-implication} and Remark \ref{coinvariants-remark} to the tower of fields $H_n/E_n/F$, we conclude that the $\Bbb Z_p[G(H_n/F)]^-$--basis $\bf v_n$ of $\overline V_{S'_n}^T(H_n)_p$ canonically induces a $\Bbb Z_p[G(E_n/F)]^-$--basis of $\overline V_{S'_n}^T(E_n)_p$, for each $n$. Remark \ref{coinvariants-remark} shows that, with respect to these bases, we have the following.
 \begin{eqnarray*}
      \pi_{E_{n+1}/E_n}(det(f_{S^*(H), S'_{n+1}}^T(E_{n+1})_p))&=&\pi_{H_{n+1}/E_n}(det(f_{S^*(H), S'_{n+1}}^T(H_{n+1})_p))\\
      &=&\pi_{H_n/E_n}(\pi_{H_{n+1}/H_n}(det(f_{S^*(H), S'_{n+1}}^T(H_{n+1})_p)))\\
      &=&\pi_{H_n/E_n}(det(f_{S^*(H), S'_n}^T(H_{n})_p))\\
      &=&det(f_{S^*(H), S'_n}^T(E_{n})_p)
 \end{eqnarray*}
Therefore, we can define 
$$det(f_{S^*(H)}^T(E_{\infty})_p):=(det(f_{S^*(H), S'_n}^T(E_{n})_p))_n\in \mathbb{Z}_p[[G(E_\infty/F)]]^-.$$
Since the auxiliary $\Bbb Z_p[G(H_n/F)]^-$--basis $\tilde{\bf v}_n$ of $\overline V_{S'_{n+1}}^T(H_n)_p$ canonically induces 
an auxiliary $\Bbb Z_p[G(E_n/F)]^-$--basis $\tilde{\bf v}_n$ of $\overline V_{S'_{n+1}}^T(E_n)_p$, we have similar equalities
\[\pi_{E_{n+1}/E_n}(det(f_{S^*(E), S'_{n+1}}^T(E_{n+1})_p))=det(f_{S^*(E), S'_n}^T(E_{n})_p),
\]
for all $n$. Therefore, we can also define 
$$det(f_{S^*(E)}^T(E_{\infty})_p):=(det(f_{S^*(E), S'_n}^T(E_{n})_p))_n\in \mathbb{Z}_p[[G(E_\infty/F]]^-.$$
Now, by Proposition 13(2) of \cite{Dasgupta-Kakde-Silliman-ETNC}, for each $n\geq 0$, we have
$$det(f_{S^*(H), S'_n}^T(E_{n})_p)=det(f_{S^*(E), S'_n}^T(E_{n})_p)\cdot \prod_{v\in S^*(H)\setminus S^*(E)}(1-\sigma_v(E_n)^{-1}), $$
where $\sigma_v(E_n)$ is the Frobenius of the place $v$ in $G(E_n/F)$. Therefore, by taking a projective limit we have
\begin{equation} \label{determinant-S-change}
    det(f_{S^*(H)}^T(E_{\infty})_p)=det(f_{S^*(E)}^T(E_{\infty})_p)\cdot \prod_{v\in S^*(H)\setminus S^*(E)}(1-\sigma_v(E_{\infty})^{-1}) 
\end{equation}
where $\sigma_v(E_\infty)$ is the Frobenius of $v$ in $G(E_\infty/F)$. On the other hand, we have
\begin{eqnarray*}
    det(f_{S^*(H)}^T(E_{\infty})_p)&=& \pi_{H_{\infty}/E_{\infty}}(det(f_{S^*(H)}^T(H_{\infty})_p)) \\
    &=& \pi_{H_{\infty}/E_{\infty}}(y_{\infty}) \cdot \pi_{H_{\infty}/E_{\infty}}(\Theta_{S^*(H)}^T(H_\infty/F))\\
    &=& \pi_{H_{\infty}/E_{\infty}}(y_{\infty}) \cdot \Theta_{S^*(H)}^T(E_\infty/F)\\
    &=& \pi_{H_{\infty}/E_{\infty}}(y_{\infty}) \cdot \Theta_{S^*(E)}^T(E_\infty/F) \cdot \prod_{v\in S^*(H)\setminus S^*(E)}(1-\sigma_v(E_{\infty})^{-1}).
\end{eqnarray*}
Observe that each term in the equalities above is a nonzero divisor of $\mathbb{Z}_p[[G(E_\infty/F)]]^-$ by Proposition 4.6 of \cite{gambheera-popescu}. Now, if one combines the above equalities with \ref{determinant-S-change}, one obtains the following.
$$det(f_{S^*(E)}^T(E_{\infty})_p)=\pi_{H_{\infty}/E_{\infty}}(y_{\infty}) \cdot \Theta_{S^*(E)}^T(E_\infty/F).$$
By projecting onto to $\mathbb{Z}_p[G(E/F)]$ and using the fact that $\pi_{H_{\infty}/E}(y_{\infty})=\pi_{H/E}(y_H)$, we obtain
$$det(f_{S^*(E)}^T(E)_p)=\pi_{H/E}(y_{H}) \cdot \Theta_{S^*(E)}^T(E/F),$$
which completes the proof.
\end{proof}

In order to prove the full $ETNC_p(H/F)^-$, we need to remove the primes in $S^*(E)\setminus S(E)$ (i.e. the $p$--adic primes in $F$ which are unramified in $E/F$) in the statement of $ETNC^\ast_p(H/F)^-$. We achieve this by applying the method of Taylor-Wiles primes, as used in our proof of the Burns--Kurihara--Sano Conjecture in \S7 of \cite{gambheera-popescu}. For that, we start with the following.
\begin{definition}
    An odd prime $p$ is called $H/F$--bad if there exists $v\in S_p(F)$ such that $j \notin G_v(H/F)$, where $j$ is the unique complex conjugation automorphism of the CM field $H$ and  $G_v(H/F)$ is the decomposition group of $v$ in $H/F$. Otherwise, we say that the prime $p$ is 
    $H/F$--good.
\end{definition}

\begin{proof}[Proof of $ETNC_p(H/F)^-$] Fix a prime $p>2$ and auxiliary sets $S'$ and $T$, such that $\overline P(H/F, S^\ast(H), T, S')$ are satisfied.
By Proposition \ref{ETNC-S*} and Proposition 13(2) of \cite{Dasgupta-Kakde}, there exists $y_H\in (\mathbb{Z}_p[G]^-)^{\times}$ such that the following is true for each intermediate CM field $E$ as above.
$$det(f_{S(E), S'}^T(E)_p) \prod_{v\in S^*(E)\setminus S(E)} (1-\sigma_v^{-1}(E/F))=\pi_{H/E}(y_{H}) \cdot \Theta_{S(E)}^T(E/F)  \prod_{v\in S^*(E)\setminus S(E)} (1-\sigma_v^{-1}(E/F))$$
Now, if the prime $p$ is $H/F$--good, then it is $E/F$--good for each intermediate CM field $E$. Then, by Remark 7.32 of \cite{gambheera-popescu}, $\prod_{v\in S^*(E)\setminus S(E)} (1-\sigma_v^{-1}(E/F))$ is invertible in $\mathbb{Z}_p[G(E/F)]^-$. So, we obtain the desired result by canceling out this term in the above equality. Now, independence of $S'$ and $T$ (see Remark \ref{independence on S' and T} above) settles $ETNC_p(H/F)^-$ for $H/F$--good primes $p>2$. So, from now on let us assume that $p$ is $H/F-$bad. \\

Now, fix integers $N>M>0$, such that $N-M= ord_p(|G|\cdot f)$, where $f$ is the least common multiple of the residual degrees $f(v/p):=[\mathcal{O}_F/v:\mathbb{F}_p]$ for all $v\in S_p(F)$. By Proposition 7.33 in \cite{gambheera-popescu}, due essentially to Greither (see Proposition 4.1 in \cite{Greither-Brumer}), there exists an odd prime $r$ such that the following hold.

\begin{enumerate}
    \item[(TW1)] $r\equiv 1 \pmod{p^N}$. 
    \item[(TW2)] $r$ is $H/F$--good.
    \item[(TW3)] The Frobenius morphism $\sigma_p(K/\mathbb{Q})$ generates $G(K/\mathbb{Q})$, where $K\subseteq \mathbb{Q}(\mu_r)$ such that $[K:\mathbb{Q}]=p^N.$
    \item[(TW4)] $S_r(F)\cap (S(H)\cup T)=\emptyset$, where $S_r(F)$ is the set of all $r$--adic primes in $F$.
    \item[(TW5)] $r\notin S_{ram}(H^c/\mathbb{Q})$, where $H^c$ is the Galois closure of $H$ (relative to $\Bbb Q$).
\end{enumerate}
Note that for any intermediate CM field $E$ as above, $EK$ is an abelian, CM extension of $F$ and 
$$S(EK)=S(E)\overset{\cdot}{\cup} S_r(F), \qquad S^\ast(EK)=S^\ast(E)\overset{\cdot}{\cup}S_r(F).$$
Pick sets of $F$--primes $S'$ and $T$, so that $\overline P(HK/F, S^\ast(HK), T, S')$ are satisfied. Note that \eqref{larger-implication} shows that $\overline P(EK/F, S^\ast(EK), T, S')$ and $\overline P(E/F, S^\ast(E), T, S')$ are also satisfied, for all $E$ as above.
By Proposition \ref{ETNC-S*} applied to the data $(HK/F, p, S', T)$, there exists $y_{HK}\in (\mathbb{Z}_p[G(HK/F)]^-)^\times$ such that
$$det(f_{S^\ast(EK), S'}^T(EK))=\pi_{HK/EK}(y_{HK})\cdot \Theta_{S^\ast(EK)}^T(EK/F),$$
for all $E$ as above. If $u_p(EK/F):=\prod_{v\in S_p(F)\setminus S(E)}(1-\sigma_v^{-1}(EK/F))$, then the above equality combined with Proposition 13(2) of \cite{Dasgupta-Kakde} gives the following for all $E$ as above.
$$det(f_{S(EK), S'}^T(EK))\cdot u_p(EK/F)=\pi_{HK/EK}(y_{HK})\cdot \Theta_{S(EK)}^T(EK/F)\cdot u_p(EK/F).$$
Let $\Delta:=G(K/\Bbb Q)$. Note that Galois restriction induces a group isomorphism $G(EK/F)\simeq G(E/F)\times \Delta$, for all $E$ as above.  This allows us to define the following element in $\Bbb Z_p[\Delta]$.   
$$\nu:=1+\sigma_p(K/\mathbb{Q})^{p^{N-M}}+(\sigma_p(K/\mathbb{Q})^{p^{N-M}})^2+...+ (\sigma_p(K/\mathbb{Q})^{p^{N-M}})^{p^M-1}$$
and view it as an element in $\Bbb Z_p[G(EK/F)]$. 
By Proposition 4.6 of \cite{Greither-Brumer}, $u_p(EK/F)$ becomes a non-zero-divisor in the quotient ring  $\mathbb{Z}_p[G(EK/F)]^-/(\nu)$. Therefore, after canceling out $u_p(EK/F)$, the next to the last displayed equality above becomes 
$$det(f_{S(EK), S'}^T(EK))=\pi_{HK/EK}(y_{HK})\cdot \Theta_{S(EK)}^T(EK/F) \quad \text{ in }\mathbb{Z}_p[(EK/F)]^-/(\nu).$$
Observe that the Galois restriction map $G(EK/F)\rightarrow G(E/F)$ induces a surjective ring morphism: $$\mathbb{Z}_p[G(EK/F)]^-/(\nu)\twoheadrightarrow \mathbb{Z}_p[G(E/F)]^-/p^M.$$
Via this morphism, the last displayed equality combined with Remark \ref{coinvariants-remark} gives 
$$det(f_{S(EK), S'}^T(E))=\pi_{HK/E}(y_{HK})\cdot \Theta_{S(EK)}^T(E/F) \quad \text{ in } \mathbb{Z}_p[G(E/F)]^-/p^M. $$
Now, if we let $u_r(E/F):=\prod_{v\in S_r(F)}(1-\sigma_v(E/F)^{-1})$ and apply Prop. 13(2) in \cite{Dasgupta-Kakde}, the last equality becomes
$$det(f_{S(E), S'}^T(E))\cdot u_r(E/F)=\pi_{HK/E}(y_{HK})\cdot \Theta_{S(E)}^T(E/F)\cdot u_r(E/F) \quad \text{ in } \mathbb{Z}_p[Gal(E/F)]^-/p^M.$$
Since the prime $r$ is $H/F$--good, it is also $E/F$--good. Hence, by Remark 7.32 of \cite{gambheera-popescu}, $u_r(E/F)$ is a unit in $\mathbb{Z}_p[G(E/F)]^-$. Therefore, its image $\mod p^M$ is a unit and we obtain 
$$det(f_{S(E), S'}^T(E))=\pi_{HK/E}(y_{HK})\cdot \Theta_{S(E)}^T(E/F)\quad \text{ in }\mathbb{Z}_p[G(E/F)]^-/p^M.$$
These equalities are valid only for auxiliary sets $S'$ and $T$ such that $\overline P(HK/F, S^\ast(HK), T, S')$ are satisfied. However, an argument identical to that given in Remark \ref{independence on S' and T} shows that they also hold for all auxiliary sets $S'$ and $T$ such that $\overline P(H/F, S(H), T, S')$ are satisfied, with $\pi_{HK/E}(y_{HK})$ replaced by 
$\pi_{HK/E}(y_{HK})\cdot\pi_{H/E}(x_H)$, for some $x_H\in(\Bbb Z_p[G]^{-})^\times.$

Now, we let $M\to\infty$ and for every $M$ we pick a prime $r:=r_M$ and consequently a field $K:=K_M$ as above. We 
let $y_H$ be a limit of a convergent subsequence of the sequence $(\pi_{HK_M/H}(y_{HK_M})\cdot x_H)_M$ in the compact group $(\mathbb{Z}_p[G]^-)^{\times}$. Then, by a standard topological argument, the above equalities  give the following
$$det(f_{S(E), S'}^T(E))=\pi_{H/E}(y_H)\cdot \Theta_{S(E)}^T(E/F) \qquad \text{ in }\Bbb Z_p[G(E/F)]^-,$$
for all intermediate fields $E$ as above. This completes the proof of $ETNC_p(H/F)^-$ for all primes $p>2$. According to Remark \ref{ETNC p-adic enough}, this completes the proof of $ETNC(H/F)^-.$
\end{proof}

\appendix
\section{Shifted Fitting Ideal Computations}

In this section, we will compute explicitly the shifted Fitting ideals  $${\rm Fitt}_{\mathbb{Z}_p[[\mathcal{G}]]}^{[1]}(Y_v), \qquad \text{ where } Y_v:=\Bbb Z_p[[\mathcal G/\mathcal G_v]],$$  for all primes $v\in S_f\cap S_{ram}$. As explained before, this generalizes the results in \cite{Greither-Kataoka-Kurihara}, where this calculation was carried out only for primes $v\in (S_f\cap S_{ram})\setminus S_p.$\\

If $F_\infty/F$ is the cyclotomic $\Bbb Z_p$--extension of $F$, we know that there exists an intermediate CM number field $H'$, with $F\subseteq H'\subseteq H_\infty$, such that $H'F_\infty=H_\infty$ and Galois restriction induces a group isomorphism 
$$\mathcal G\simeq G(H'/F)\times G(F_\infty/F).$$
Observing that $H_\infty$ is the cyclotomic $\Bbb Z_p$--extension of $H'$ as well, in order to simplify notations, we may assume without loss of generality that $H=H'$. Consequently,  $G=Gal(H'/F)$. We let $\Gamma:=G(H_\infty/H)$ and denote by $\gamma$ a topological generator of $\Gamma$. The group isomorphism above becomes
\begin{equation}\label{group-iso-app}\mathcal G\simeq G\times\Gamma.\end{equation}
As before, for each $n\geq 0$, we let $H_n$ be the $n$--th layer of the cyclotomic tower $H_\infty/H$ and $G_n:=Gal(H_n/F)$.

Now, we fix $n\gg 0$ such that all primes in $S_f$ are inert or totally ramified in $H_\infty/H_n$. (This is possible because for all primes $v\in S_f$, the decomposition group $\mathcal G_v$ is an open subgroup in $\mathcal G$.) Then, Galois restriction induces group isomorphisms  
$$\mathcal{G}/\mathcal{G}_v\cong G_n/G_{n,v},\qquad\text{ for all }v\in S_f,$$ 
where $G_{n,v}$ is the decomposition groups of $v$ in the extensions $H_n/F$. In what follows, we let $R:=\Bbb Z_p[G_n]$ and fix a prime $v\in S_f\cap S_{ram}.$ 
\\

By Proposition 4.2 in \cite{Greither-Kataoka-Kurihara}, we have the following.
\begin{proposition}[Greither--Kataoka--Kurihara]\label{GKK-prop-shifted}
Assume that we have an exact sequence
\begin{equation}\label{Y-resolution}R^{t_3}\xrightarrow[]{A} R^{t_2}\xrightarrow[]{} R^{t_1}\xrightarrow[]{} Y_v\xrightarrow[]{} 0,\end{equation}
of $R$--modules, for same $t_1, t_2, t_3\in\Bbb Z_{>0}$ and a matrix $A\in M_{t_2\times t_3}(R)$. Then, we have an equality
\begin{equation}\label{shifted-fitt-equation}
{\rm Fitt}_{\mathbb{Z}_p[[\mathcal{G}]]}^{[1]}(Y_v)=(\gamma^{p^n}-1)^{t_2-t_1}\cdot \sum_{e=0}^{t_2}(\gamma^{p^n}-1)^{-e}{\rm Min}_e(\Tilde{A})
\end{equation}
where $\Tilde{A}\in M_{t_2\times t_3}(\mathbb{Z}_p[[\mathcal{G}]])$ is a lift of $A$ under the natural ring morphism $\Bbb Z_p[[\mathcal G]]\twoheadrightarrow \Bbb Z_p[G_n]$ and ${\rm Min}_e(\Tilde{A})$ is the $\Bbb Z_p[[\mathcal G]]$--ideal generated by the $(e\times e)$--minors of $\Tilde{A}$. 
\end{proposition}

Now, our first goal in this section is to find an explicit resolution of type \eqref{Y-resolution} for $Y_v$.
For that, fix a decomposition of $G_{n,v}$ as a product of cyclic groups: 
$$G_{n,v}=\langle g_1\rangle\times \langle g_2\rangle\times\text{...  }\langle g_r\rangle $$
Let $\beta_i=g_i-1$ and $\alpha_i=N(\langle g_i \rangle):=\sum_{k=1}^{ord(g_i)}g_i^k$, for all $i$. Then, we have isomorphisms of $R$--modules 
$$Y_v\cong \mathbb{Z}_p[G_n/G_{n,v}]\cong R/(\beta_1, \beta_2, \text{...  }\beta_r).$$ 
Consequently, we have the following exact sequence of $R$--modules
$$0\xrightarrow[]{} K\xrightarrow[]{} R^r\xrightarrow[]{f} R\xrightarrow[]{} Y_v\xrightarrow[]{}0,$$
where $f$ is defined on the standard $R$--basis $(e_i)_i$ of $R^r$ as $f(e_i)=\beta_i$, for all $i$, and  
$$K:=\ker(f)=\{(y_1, y_2,\text{...  } ,y_r)\in R^r\, \mid\, \sum_{i=1}^{r}y_i\beta_i=0  \}.$$ In what follows, we will give an explicit set of generators for the $R$--module $K$. More precisely, for any $k=1,\dots, r$, we will give an explicit set of generators for the $R$--module 
$$K_k:=\{(y_1, y_2,\dots ,y_k)\in R^k\,\mid\, \sum_{i=1}^{k}y_i\beta_i=0  \}.$$
Note that, under this definition, we have $K=K_r$ and that the $R$--module embedding
$$R^{k-1}\to R^k, \qquad x=(x_1, \dots, x_{k-1})\to x^\ast:=(x_1, \dots, x_{k-1}, 0),$$
sends $K_{k-1}$ to a submodule $K_{k-1}^\ast$ of $K_k$.
\begin{proposition}\label{image}
Let $1\leq k\leq r$ and let $\{e^k_i\}_{i=1}^k$ be the standard $R$--basis of $R^k$. Then the $R$--module $K_k$ is generated by the following set of $k(k+1)/2$ vectors 
$$V_k:=\{v_i^k:=\alpha_i\cdot e_{i}^k\mid i=1, \dots, k\}\, \bigcup\, \{q_{i,j}^k:=-\beta_j e_i^k+\beta_i e_j^k\mid 1\leq i<j\leq k\}.$$
\end{proposition}
\begin{proof} First, note that $V_k\subseteq K_k$, as $\alpha_i\cdot\beta_i=0$, for all $i=1,\dots, r.$ Now, we will proceed to prove the statement in the Proposition by induction on $k$. 

For $k=1$, first note $V_1=\{\alpha_1\}\subseteq R$. Second, note that since $R$  is an induced $\langle g_1\rangle$--module, $R$ is $\langle g_1\rangle$--cohomologically trivial. In particular, $\widehat H^{0}(\langle g_1\rangle, R)=0$, where $\widehat{H}^\ast$ denotes Tate cohomology. By definition, this amounts to the exactness of the following sequence
$$R\overset{\times\alpha_1}\longrightarrow R\overset{\times\beta_1}\longrightarrow R,$$
which gives the desired statement $K_1=R\cdot\alpha_1.$  

Now, assume that $V_{k-1}$ generates $K_{k-1}$, for some $k\geq 2$. Noting that we have the obvious equalities
$$(e_i^{k-1})^\ast=e_i^k, \qquad (q_{i,j}^{k-1})^\ast=q_{i,j}^k, \qquad \text{ for all }1\leq i<j<k,$$
in order to prove that $V_k$ generates $K_k$, it suffices to prove that
\begin{equation}\label{K_k+K_{k-1}}K_k=K_{k-1}^\ast + \langle \alpha_k\cdot e_k^k,\quad (q_{i,k}^k)_{i=1}^{k-1} \rangle.\end{equation}
The inclusion of the right hand side in the left hand side is obvious.   Let $(y_1, \dots, y_k)\in K_k.$ Then, we have 
\begin{equation}\label{K_k}\beta_1\cdot y_1+\dots+\beta_k\cdot y_k=0.\end{equation}
Let $\overline{G_n}:=G_n/\langle g_1\rangle\times \langle g_2\rangle\times\text{...  }\langle g_{k-1}\rangle$ and $\overline R=R/(\beta_1, \beta_2,\text{...  }\beta_{k-1})\cong \mathbb{Z}_p[\overline{G_n}]$. Let us view the last displayed equality
in $\overline R$, via the canonical ring morphism $R\twoheadrightarrow \overline R$ taking $x\to\overline x$.
Then, we have $\overbar{y_k}\cdot\overbar{\beta_k}=0$. Now, observe that $\overline{\alpha_k}=\overline{N(\langle g_k\rangle)}=N(\langle \overbar{g_k}\rangle)$. So, the cohomological argument used in the case $k=1$, applied now to $\overline R$ and $\langle\overline{g_k}\rangle$, shows that $\overbar{y_k}=\overbar{\alpha_k}\cdot \overbar{\alpha'}$, for some $\alpha'\in R$. 
Therefore, there exist $\theta_1, \dots, \theta_{k-1}\in R$, such that 
\begin{equation}\label{y_k}y_k=\alpha_k\cdot\alpha'+\sum_{i=1}^{k-1}\beta_i\theta_i.\end{equation}
Now, combine the equality above with \eqref{K_k} and with the fact that $\alpha_k\beta_k=0$ to obtain
$$\sum_{i=1}^{k-1}(y_i+\theta_i\beta_k)\beta_i=0.$$
This shows that $(y_i+\theta_i\beta_k)_{i=1}^{k-1}\in K_{k-1}$. So $v:=((y_i+\theta_i\beta_k)_{i=1}^{k-1})^\ast\in K_{k-1}^\ast.$ By \eqref{y_k}, we obtain:
$$\begin{aligned} (y_1, \dots, y_{k-1}, y_k) &=v-(\theta_1\beta_k, \dots, \theta_{k-1}\beta_k, -\alpha'\alpha_k-\theta_1\beta_1-\dots -\theta_{k-1}\beta_{k-1})\\
&=v+\theta_1\cdot q_{1,k}^k+\dots +\theta_{k-1}\cdot q_{k-1,k}^k+ \alpha'\cdot \alpha_ke_k^k.
\end{aligned}$$
This proves that the left hand side of \eqref{K_k+K_{k-1}} is included in the right hand side, which concludes the proof. \end{proof}
\medskip

\begin{definition}\label{matrix A Q} For every $k=1, 2, \dots, r$, we define the following matrix in $M_{k\times k(k-1)/2}(R)$: 
$$Q_k=\begin{pmatrix} q_{1,2}^k & q_{1,3}^k &\dots &q_{k-1, k}^k\end{pmatrix},$$
where the vectors $\{q_{i,j}^k\mid 1\leq i<j\leq k\}$ are viewed as column vectors in $R^k$ and ordered with respect the lexicographic order 
$(i,j)< (i',j')$ if $i<i'$ or $i=i'$ and $j<j'$. Note that $Q_1$ is the empty matrix.

Further, we define the matrices $A_k\in M_{k\times k(k+1)/2}(R)$, for all $k=1,2,\dots, r$ by
$$A_k=\begin{pmatrix}\alpha_1 e_1^k& \dots &\alpha_k e_k^k & Q_k\end{pmatrix}= \begin{pmatrix}\alpha_1 e_1^k& \dots &\alpha_k e_k^k & q_{1,2}^k & q_{1,3}^k &\dots &q_{k-1, k}^k\end{pmatrix},$$
where the vectors $e_i^k$ and $q_{i,j}^k$ are again viewed as column vectors in $R^k$.
\end{definition}
\medskip

Under the above definition, an immediate consequence of the above Proposition is the following.

\begin{corollary}\label{Y_v-resolution}
With notations as above , we have an exact sequence of $R$--modules 
$$R^{r(r+1)/2}\xrightarrow[]{A_r}R^r\xrightarrow[]{f} R\xrightarrow[]{}Y_v\xrightarrow[]{}0,$$
where $A_r$ is the matrix in Definition \ref{matrix A Q}.
\end{corollary}

Now, by Proposition \ref{GKK-prop-shifted}, in order to compute the shifted Fitting ideals of interest we need to pick a lift $\Tilde{A_r}$ of $A_r$ in $M_{r\times r(r+1)/2}(\Bbb Z_p[[\mathcal G]])$. For that, since $A_r$ depends on the chosen $n\gg 0$ and generators $\{g_i\}_{i=1}^r$ of $G_{n, v}$, we need to make those choices a bit more carefully.\\

We revisit the group isomorphism \eqref{group-iso-app} and observe that ${\rm Tor}(\mathcal G)\cong G$, where ${\rm Tor}(\ast)$ denotes the torsion subgroup of the abelian group $\ast$.
Further, since $\mathcal G_v$ is an open subgroup of $\mathcal G$, its $\Bbb Z_p$--rank is one and its torsion subgroup is finite, and therefore we have a direct product decomposition
$$\mathcal G_v={\rm Tor}(\mathcal G_v)\times \Gamma_v,$$
where $\Gamma_v\simeq \Bbb Z_p$ and ${\rm Tor}(\mathcal G_v)\subseteq G$. It is easy to see that one can pick a topological generator $y=(g, \gamma^t)$ of $\Gamma_v$, where $g\in G$ is of $p$--power order and $t$ is a power of $p$.

Now, we fix $n\gg 0$ satisfying the previous conditions (i.e. all primes in $S_f$ are either inert or totally ramified in $H_\infty/H_n$) and, in addition, satisfying
$$\gamma^{p^n}\in\langle y\rangle\subseteq \Gamma_v.$$
For this $n$, observe that if $\pi_n:\mathcal G\to G_{n}$ is the Galois restriction, then we have 
$$\pi_n:\mathcal{G}\cong G\times \Gamma\xrightarrow[]{} G_n\cong G\times (\Gamma/\Gamma^{p^n}), \qquad \pi_n(G)=G,\qquad \pi_n(\Gamma)=\Gamma/\Gamma^{p^n}$$
$$G_{n,v}=\pi_n(\mathcal{G}_{v})\cong \pi_n({\rm Tor}(\mathcal{G}_{v}))\times (\overbar{\langle y \rangle}/\Gamma^{p^n}).$$
Observe that $\pi_n({\rm Tor}(\mathcal{G}_{v}))\subseteq G$ is independent of $n$ and choose  
generators $\{g_i\}_{i=1}^{r-1}$ of $\pi_n({\rm Tor}(\mathcal{G}_{v}))$, such that
$$\pi_n({\rm Tor}(\mathcal{G}_{v}))=\langle g_1\rangle\times \langle g_2\rangle\times\text{...  }\langle g_{r-1}\rangle.$$
Further, let $g_r:=\pi_n(y)$. Then, we have a direct product decomposition
$$G_{n,v}=\langle g_1\rangle\times \langle g_2\rangle\times\text{...  }\langle g_r\rangle.$$

For this choice of generators, we construct the corresponding elements $\alpha_i, \beta_i\in R=\Bbb Z_p[G_{n}]$, for all $i=1,\dots, r$. (see the paragraphs leading into Proposition \ref{image}.) Observe that out of these elements, only $\alpha_r, \beta_r$ depend on $n$. 
Now, we choose lifts $\widetilde{\alpha_i}$ and $\widetilde{\beta_i}$ of these elements to $\Bbb Z_p[[\mathcal G]]$ in the following manner.
\begin{gather*}\widetilde{\alpha_i}=\alpha_i, \quad \widetilde{\beta_i}=\beta_i\qquad\text{ for } 1\leq i\leq r-1\\
\widetilde{\alpha_r}=\sum_{i=0}^{{\rm ord}_{G_n}(g_r)-1}y^i, \quad \widetilde{\beta_r}=y-1.\end{gather*}
With these lifts, we construct lifts $\widetilde Q_k$ and  $\widetilde{A_k}$ for the matrices given in Definition \ref{matrix A Q}, for all $k=1,\dots, r$, in the obvious way. Observe that in the matrix $\widetilde{A_r}$ only the column $\widetilde{\alpha_r}e_r^r$ depends on the chosen $n$ and that
\begin{equation}\label{lifts}\widetilde{\alpha_i}\cdot\widetilde{\beta_i}=0 \text{ for }i=1,\dots, r-1; \qquad \widetilde{\alpha_r}\cdot\widetilde{\beta_r}\ne 0.\end{equation}
\\

Next, we prove some technical results regarding the minors of the matrix $\widetilde{A_r}$ which, via Proposition \ref{GKK-prop-shifted} and Corollary \ref{Y_v-resolution} will eventually lead to the calculation of the desired shifted Fitting ideal. In what follows, $x_1, \dots, x_r, y_1, \dots, y_r$ denote algebraically independent variables. Working over the polynomial ring $\mathcal R:=\Bbb Z_p[x_1, \dots, x_r, y_1, \dots, y_r]$, we repeat the construction in Definition \ref{matrix A Q} to obtain matrices $\mathcal Q_k$ and $\mathcal A_k$:

\begin{gather*}\mathcal Q_k:=\begin{pmatrix} q_{1,2}^k & q_{1,3}^k &\dots &q_{k-1, k}^k\end{pmatrix}\in M_{k\times k(k-1)/2}(\mathcal R), \qquad\text{ where }q_{i,j}^k:=-x_je_i^k+x_ie_j^k\\
\mathcal A_k:=\begin{pmatrix}y_1 e_1^k & \dots & y_k e_k^k & \mathcal Q_k\end{pmatrix} \in M_{k\times k(k+1)/2}(\mathcal R),\end{gather*}
for all $k=1, 2, \dots, r$ an all $1\leq i<j\leq k$.

\begin{lemma}\label{k-minor}
With notations as above, we have ${\rm Min}_k(\mathcal Q_k)=(0)$, for all $1\leq k\leq r$.
\end{lemma}
\begin{proof}
Observe that $\mathcal R$ is an integral domain and that, by the definition of $\mathcal Q_k$, we have
$$\begin{pmatrix} x_1 & x_2 & ... & x_k \end{pmatrix}\cdot\mathcal Q_k=0.$$
Therefore, all the $k-$minors of $\mathcal Q_k$ must be zero. This completes the proof.
\end{proof}
\begin{lemma}\label{monomials} Let $k=1, \dots, r$. For any $k$--minor of $\mathcal A_k$ (viewed as a polynomial in $\mathcal R$) and an arbitrary non--zero monomial of this $k$--minor, one of the following holds.
\begin{enumerate}
    \item The monomial is equal to $y_1y_2\text{...  }y_k.$
    \item The monomial is divisible by $x_iy_i$, for some $i=1, \dots, k$.
\end{enumerate}
\end{lemma}
\begin{proof}
We use induction on $k$. When $k=1$, clearly (1) occurs because $\mathcal A_1=(y_1)$. Now, let us assume the result for all the integers less than a given $k\geq 2$ and prove it for $k$. 

So, pick a $k\times k$ submatrix $X$ of $\mathcal A_k$ whose determinant is the minor in question.
If the chosen monomial is not divisible by $y_i$, for any $i\leq k$, then it must occur as a monomial in a $k$--minor of $\mathcal Q_k$ and therefore vanish, according to the previous Lemma. Since the chosen monomial is not $0$, it must be divisible by $y_i$, for some $i$. This means that the matrix $X$ contains $y_ie_i^k$ as a column. When expanding $\det(X)$ with respect to that column, one obtains $$\det(X)=y_i\det(X'),$$ where $X'$ is obtained from $X$ by eliminating its $y_ie_i^k$--column and its $i$--th row. 

Now, observe that if $X'$ contains any entry equal to $\pm x_i$, then that will be the unique non--zero entry in the corresponding column of $X'$. Therefore, $\det(X')$ is divisible by $x_i$ and from the last displayed equality we conclude that the chosen monomial is divisible by $x_iy_i$, satisfying property (2).

On the other hand, if $X'$ contains no entry equal to $\pm x_i$, then $X'$ is a $(k-1)\times(k-1)$ submatrix of a type $\mathcal A_{k-1}$  matrix constructed out of the variables 
$\{x_1, \dots, x_k\}\setminus\{x_i\}$ and $\{y_1, \dots, y_k\}\setminus\{y_i\}$. By the induction hypothesis, any non-zero monomial in $\det(X')$ satisfies either (1) or (2) with respect to the new sets of variables. Therefore, any non--zero monomial in the product $y_i\det(X')$ will satisfy (1) or (2) with respect to the old sets of variables. \end{proof}

\begin{remark}\label{matrix-symmetry} It is not very difficult to see that if $\sigma$ is a permutation of $\{1, 2, \dots, k\}$, then 
$$\mathcal A_k^\sigma=\mathcal A_k(x_{\sigma(1)}, \dots, x_{\sigma(k)};y_{\sigma(1)}, \dots, y_{\sigma(k)})\, \sim\,  \mathcal A_k=\mathcal A(x_1, \dots, x_k;y_1, \dots, y_k)$$
where $\sim$ is the equivalence relation of $k\times k$ matrices induced by permuting rows or columns and multiplying rows or columns by $\pm 1$. Since determinants are invariant under $\sim$ up to signs, we can conclude that 
$${\rm Min}_k(\mathcal A_k^\sigma)={\rm Min}_k(\mathcal A_k),$$
for all $\sigma$ as above.
\end{remark}
Now, we are ready to compute the terms in equation \ref{shifted-fitt-equation}. 
\begin{proposition}\label{r-minor}
The following equality of $\Bbb Z_p[[\mathcal G]]$--ideals holds.
\begin{equation*}
        {\rm Min}_r(\Tilde{A_r})=\left(\prod_{i=1}^r\Tilde{\alpha_i}, \quad \Tilde{\alpha_r}\Tilde{\beta_r}^{t_r}\prod_{i\in L}\Tilde{\alpha_i}\prod_{i\in L^c}\Tilde{\beta_i}^{t_i}\quad\bigg|\quad  L\subsetneqq \{1,2,\text{ ... }, r-1\},\,
        \sum_{i\in L^c\cup\{r\}} t_i=r-|L|-1,\, t_i\geq0,\, t_r\geq 1\right),
\end{equation*}
where $L$ runs through the strict subsets of $\{1, \dots, r-1\}$ and $L^c$ is the complement of $L$ inside $\{1, \dots, r-1\}$.
\end{proposition}
\begin{proof}
Note that the matrix of $\widetilde{A_r}$ is the image of the matrix $\mathcal A_r$ under the $\Bbb Z_p$--algebra morphism 
$$h:\mathbb{Z}_p[x_1,x_2,\dots, x_r, y_1,y_2,\dots, y_r]\xrightarrow[]{} \mathbb{Z}_p[[\mathcal{G}]], \qquad h(x_i)=\widetilde{\beta_i}, \quad h(y_i)= \widetilde{\alpha_i}, \text{ for all } i.$$
Therefore, the $r$--minors of $\widetilde A_r$ are the images of the $r$--minors of $\mathcal A_r$ via $h$.
Now, by Lemma \ref{monomials} and the fact that $\widetilde{\alpha_i}\widetilde{\beta_i}=0$, for all $1\leq i\leq r-1$ (see \eqref{lifts}), the only non--zero images of monomials in the $r$--minors of $\mathcal A_r$ via $h$ are the explicit generators of the right side ideal given in the statement of the Proposition we are trying to prove. 
Therefore, the left side ideal is included in the right side ideal.\\\\
Next, we prove the reverse inclusion. First, observe that  $$\prod_{i=1}^r\widetilde{\alpha_i}=\det\begin{pmatrix}\widetilde{\alpha_1}e_1^r & \dots & \widetilde{\alpha_r} e_r^r\end{pmatrix} \in {\rm Min}_r(\widetilde{A_r}).$$ 

Now, let $L$ and $(t_1, \dots, t_r)$ be as above, giving rise to the generator $$u=\widetilde{\alpha_r}\widetilde{\beta_r}^{t_r}\prod_{i\in L}\widetilde{\alpha_i}\prod_{i\in L^c}\widetilde{\beta_i}^{t_i},$$ of the right side ideal in the Proposition. 
Observe that, due to Remark \ref{matrix-symmetry}, in order to prove that $u\in{\rm Min}(\widetilde A_r)$, we may assume without loss of generality that 
$$L=\{1,2,\dots, a\}, \qquad t_{r-1}\geq t_{r-2}\geq \dots\geq t_{a+1}, \qquad \sum_{i=a+1}^r t_i=r-a-1.$$ 
We will construct an $r\times r$ submatrix $E$ (obtained by eliminating rows and columns) of $\widetilde{A_r}$, such that $$\det(E)=\pm u=\pm \widetilde{\alpha_r}\widetilde{\beta_r}^{t_r}\prod_{i=1}^a\widetilde{\alpha_i}\prod_{i=a+1}^{r-1}\widetilde{\beta_i}^{t_i}.$$ 
First, we consider the following submatrices of $\widetilde{A_r}$:
$$A:=\begin{pmatrix} \widetilde{\alpha_1}e_1^r&\dots &\widetilde{\alpha_a}e_a^r
\end{pmatrix}\qquad  B:=\begin{pmatrix}\widetilde{\alpha_r}e_r^r
\end{pmatrix}.$$
Now, consider the following quantities
$$\theta_0:=\max\{\theta\mid 0\leq \theta\leq (r-1),\,\, t_{r-\theta}>0\}, \qquad T_{\theta}:=t_{r-\theta_0}+t_{r-(\theta_0-1)}+\dots + t_{r-\theta}, \text{ for all }0\leq\theta\leq\theta_0.$$
Observe that, for every $0\leq\theta\leq\theta_0$, we have the following
$$a+T_\theta=(r-1)-(t_r+t_{r-1}+\dots+t_{r-(\theta-1)})\leq (r-1)-\theta < r-\theta.$$
The last displayed inequality, permits us to construct the following submatrices
of $\widetilde{A_r}$:
\begin{gather*}C_{r-\theta_0}:=\begin{pmatrix}{\tilde q_{a+1, r-\theta_0}^r}& {\tilde q^r_{a+2, r-\theta_0}}&\dots&{\tilde q^r_{a+t_{r-\theta_0}, r-\theta_0}}\end{pmatrix}\\
C_{r-\theta}:=\begin{pmatrix}{\tilde q_{a+T_{\theta+1}+1, r-\theta}^r}& {\tilde q}^r_{a+T_{\theta+1}+2, r-\theta}&\dots&{\tilde q^r_{a+T_{\theta}, r-\theta}}\end{pmatrix}, \text{ for all } 0\leq\theta<\theta_0,
\end{gather*}
where $\tilde q_{i,j}^r: =h(q_{i,j}^r)=-\widetilde{\beta_j}e_i^r+\widetilde{\beta_i}e_j^r$, for all $0<i<j\leq r$.  Now, define the following $r\times r$ matrix 
\[E'=\begin{pmatrix}A&C_{r-\theta_0} &C_{r-\theta_0+1} &\dots & C_{r}& B \end{pmatrix}\]
From the definitions, it is clear that $E'$ is lower triangular   and ${\rm det}(E')=\pm u.$
However, $E'$ is not a  submatrix of $\widetilde A_r$. Nevertheless, if we perform the following column permutation within $E'$, we get a sumbatrix of $\widetilde A_r$: 
$$E:=\begin{pmatrix}A&B&C_{r-\theta_0} &C_{r-\theta_0+1} &\dots & C_{r}\end{pmatrix}$$
and $\det E=\pm \det E'=\pm u$, as desired. This concludes the proof.
\end{proof}
Now, we are ready to give an explicit description of ${\rm Fitt}_{\mathbb{Z}_p[[\mathcal{G}]]}^{[1]}(Y_v)$ in terms of the choices made so far ($n\gg 0$, generators for $G_{n, v}$ and lifts of those generators to $\mathcal G$.) Before we state the next result, remark that
$$\widetilde{\alpha_r}\cdot\widetilde{\beta_r}=\gamma^{p^n}-1,$$
and $(\gamma^{p^n}-1)$ is clearly not a zero divisor in $\Bbb Z_p[[\mathcal G]]$ (as via the ring isomorphism $\Bbb Z_p[G][[\Gamma]]\simeq \Bbb Z_p[G][[T]]$ sending $\gamma\to (T+1)$, we have 
$\gamma^{p^n}-1\to (T+1)^{p^n}-1$.) Therefore, the elements $\widetilde{\alpha_r}$ and $\widetilde{\beta_r}$ are not zero divisors in $\Bbb Z_p[[\mathcal G]]$ either. Consequently,  $1/\widetilde{\beta_r}$, $1/\widetilde{\alpha_r}$ and $1/(\gamma^{p^n}-1)$ are well defined elements in the total ring of fractions $\mathcal Q(\Bbb Z_p[[\mathcal G]])$ of $\Bbb Z_p[[\mathcal G]]$.\\

\begin{theorem}\label{Shifted-Fitt}
Under the above assumptions, we have an equality of fractional $\Bbb Z_p[[\mathcal G]]$--ideals in $\mathcal Q(\Bbb Z_p[[\mathcal G]])$:
    \begin{align*}
        {\rm Fitt}_{\mathbb{Z}_p[[\mathcal{G}]]}^{[1]}(Y_v)&=\frac{1}{\gamma^{p^n}-1}{\rm Min}_r(\widetilde{A_r})\\
        &=\left(\frac{1}{\Tilde{\beta_r}}\prod_{i=1}^{r-1}\Tilde{\alpha_i}\text{ , } \prod_{i\in L}\Tilde{\alpha_i}\cdot \prod_{i\in L^c\cup \{r\}}\Tilde{\beta_i}^{t_i}\,
\bigg\mid\,  L\subseteq \{1,2,\text{ ... }, r-1\}, \quad \sum_{i\in L^c\cup\{r\}} t_i=r-|L|-2\text{ , }t_i\geq 0  \right). \end{align*}
\end{theorem}
\begin{proof} Observe that the second equality in the statement above is Proposition \ref{r-minor}. In order to prove the first equality, we simplify notations and let:
$$I:= \frac{1}{\gamma^{p^n}-1}{\rm Min}_r(\widetilde{A_r})=\left(\frac{1}{\Tilde{\beta_r}}\prod_{i=1}^{r-1}\Tilde{\alpha_i}, \quad  \prod_{i\in L}\Tilde{\alpha_i}\cdot \prod_{i\in L^c\cup \{r\}}\Tilde{\beta_i}^{t_i}\,
\bigg\mid\,  L\subseteq \{1,2,\text{ ... }, r-1\}, \quad \sum_{i=1}^r t_i=r-|L|-2\text{ , }t_i\geq 0  \right).$$
Observe that $I$ is precisely the $e=r$ term in equality \ref{shifted-fitt-equation} applied to our matrix lift $\widetilde{A_r}$:
\begin{equation}\label{Fitt-shit-sum}{\rm Fitt}_{\mathbb{Z}_p[[\mathcal{G}]]}^{[1]}(Y_v)=\sum_{e=0}^{r}(\gamma^{p^n}-1)^{r-1-e}{\rm Min}_e(\Tilde{A_r}).\end{equation}
So, by Proposition \ref{GKK-prop-shifted}, we have an inclusion of fractional ideals
$$I\subseteq {\rm Fitt}_{\mathbb{Z}_p[[\mathcal{G}]]}^{[1]}(Y_v).$$

In order to prove equality, we start by looking at the $e=r-1$ term in \eqref{Fitt-shit-sum}. Now, it is easy to see that the  $(r-1)$--minors of $\widetilde{A_r}$ are sums of monomials of the general form 
$$\Tilde{\alpha_r}^b\prod_{i\in L}\Tilde{\alpha_i}\cdot \prod_{i\in L^c\cup \{r\}}\Tilde{\beta_i}^{t_i}, \quad\text{ where }L\subseteq\{1,\dots, r-1\}, \quad \sum_{i=1}^r t_i=r-|L|-b-1, \text{ and }b=0, 1.$$
Clearly, each of these monomials are multiples of the generators of the second type in $I$. Therefore, we have 
$${\rm Min}_{r-1}(\widetilde{A_r})\subseteq I.$$
Now, if $e<r-1$, the $e$--term in \eqref{Fitt-shit-sum} is a $\Bbb Z_p[[\mathcal G]]$--ideal contained in the ideal generated by $(\gamma^{p^n}-1)$. Therefore, we have an inclusion of fractional $\Bbb Z_p[[\mathcal G]]$--ideals 
$$
I\subseteq {\rm Fitt}_{\mathbb{Z}_p[[\mathcal{G}]]}^{[1]}(Y_v) \subseteq I+(\gamma^{p^n}-1).
    $$
Since $I$ is independent of $\widetilde{\alpha_r}$, $I$ is independent of $n$. Therefore, the inclusion above is true if we replace $n$ with any $m\geq n$. Since we also know that $(\gamma^{p^n}-1)I={\rm Min}_r(\Tilde{A_r})$ is an ideal of $\mathbb{Z}_p[[\mathcal{G}]]$, we have the following inclusions of closed ideals of the Noetherian, compact, topological ring $\mathbb{Z}_p[[\mathcal{G}]]$:
$$(\gamma^{p^n}-1)I\subseteq (\gamma^{p^n}-1){\rm Fitt}_{\mathbb{Z}_p[[\mathcal{G}]]}^{[1]}(Y_v)\subseteq (\gamma^{p^n}-1)I+(\gamma^{p^n}-1)(\gamma^{p^m}-1).$$
Now, $\lim_{m \to \infty}(\gamma^{p^m}-1)=0$ in the standard topology of $\Bbb Z_p[[\mathcal G]]$. Therefore, due to compactness, we have 
$$(\gamma^{p^n}-1)I\subseteq (\gamma^{p^n}-1){\rm Fitt}_{\mathbb{Z}_p[[\mathcal{G}]]}^{[1]}(Y_v)\subseteq (\gamma^{p^n}-1)I.$$
Now, canceling out the  nonzero divisor $\gamma^{p^n}-1$ yields the desired equality of fractional $\Bbb Z_p[[\mathcal G]]$--ideals, which concludes the proof.
\end{proof}
Finally, we are able to give the following intrinsic description for ${\rm Fitt}_{\mathbb{Z}_p[[\mathcal{G}]]}^{[1]}(Y_v)$, independent of any choices. 
\begin{theorem}\label{shifted-Fitt-intrinsic}
$${\rm Fitt}_{\mathbb{Z}_p[[\mathcal{G}]]}^{[1]}(Y_v)=\left(N(A)\Delta B^{r_B-2}\,\bigg\mid\, A, B\text{ subgroups of } \mathcal{G}_v,\text{ with }\mathcal G_v=A\times B \text{ and }  A \text{ finite }\right).$$
Here, $N(A):=\sum_{a\in A}a$ is the norm element associated to $A$ in $\Bbb Z_p[[\mathcal G]]$, $\Delta B=(g-1\,\mid\,g\in B)$ is the augmentation ideal of $B$ in $\Bbb Z_p[[\mathcal G]]$, and $r_B$ is the minimum number of generators of $B$.
\end{theorem}
\begin{proof} Let $\mathcal L$ and $\mathcal R$ be the left and right side of the equality in the Theorem, respectively.
Observe that, with the choice of generators $g_1, \dots, g_r$ for $\mathcal G_v$ made above, we are permitted to choose $A$ and $B$ such that  
$$A=\prod_{i\in L}\langle g_i \rangle,\quad  B=\prod_{i\in L^c}\langle g_i \rangle\times \overbar{\langle y \rangle}, \text{ where }L\subseteq \{1, 2, \text{ ... } r-1\}. $$
In the case where $L^c\ne\emptyset$, we have 
$$N(A)=\prod_{i\in L}\Tilde{\alpha_i}, \qquad \Delta B^{r_B-2}=(\prod_{i\in L^c\cup \{r\}}\Tilde{\beta_i}^{t_i}\text{ ; }  \sum t_i=r-|L|-2, t_i\geq0).$$ 
On the other hand, if $L^c=\emptyset$, then we have 
$$N(A)\Delta B^{r_B-2}=(\prod_{i\in L}\Tilde{\alpha_i}) (y-1)^{-1}=(\frac{1}{\Tilde{\beta_r}}\prod_{i=1}^{r-1}\Tilde{\alpha_i}).$$ 
Therefore, by Theorem \ref{Shifted-Fitt}, we have $\mathcal L\subseteq \mathcal R$.\\

Now, let $\mathcal G_v=A\times B$ be a decomposition as above. We have 
$${\rm Tor}(\mathcal G_v)=A\times{\rm Tor}(B)\subseteq G, \qquad B={\rm Tor}(B)\times\Gamma_v,$$
where $\Gamma_v\simeq\Bbb Z_p$. Therefore, we can pick an $n\gg 0$, such that $\Gamma^{p^n}\subseteq \Gamma_v$, pick generators $g_1, \dots, g_{r-1}$ of $A\times{\rm Tor}(B)$ and a generator $y$ of $\Gamma_v$ as in the paragraphs after Corollary \ref{Y_v-resolution} and apply Theorem \ref{Shifted-Fitt} to the corresponding lifts $\widetilde{\alpha_i}$ and $\widetilde{\beta_i}$, for $i=1, 2, \dots, r$. As explained above, this shows that 
$$N(A)\Delta B^{r_B-2}\subseteq {\rm Fitt}_{\mathbb{Z}_p[[\mathcal{G}]]}^{[1]}(Y_v).$$
As a consequence, we have $\mathcal R\subseteq \mathcal L$. This completes the proof.
\end{proof}


\begin{thebibliography}{10}

\bibitem{bley-popescu}
Werner Bley and Cristian~D. Popescu, \emph{Geometric main conjectures in function fields}, to appear in Jour. f\"ur die Reine und Angew. Math. (2024), https://arxiv.org/abs/2209.02440.

\bibitem{Bullach-Burns-Daoud-Seo}
Dominik Bullach, David Burns, Alexandre Daoud, and Soogil Seo, \emph{Dirichlet {$L$}-series at $s=0$ and the scarcity of {E}uler systems}, https://arxiv.org/abs/2111.14689 (2023).

\bibitem{Burns-Flach}
D.~Burns and M.~Flach, \emph{Tamagawa numbers for motives with (non-commutative) coefficients}, Doc. Math. \textbf{6} (2001), 501--570. \MR{1884523}

\bibitem{Burns-Kurihara-Sano}
David Burns, Masato Kurihara, and Takamichi Sano, \emph{On zeta elements for {$\Bbb G_m$}}, Doc. Math. \textbf{21} (2016), 555--626. \MR{3522250}

\bibitem{Dasgupta-Kakde}
Samit Dasgupta and Mahesh Kakde, \emph{On the {B}rumer--{S}tark conjecture}, Ann. of Math. (2) \textbf{197} (2023), no.~1, 289--388. \MR{4513146}

\bibitem{Dasgupta-Kakde-Silliman-ETNC}
Samit Dasgupta, Mahesh Kakde, and Jesse Silliman, \emph{On the {E}quivariant {T}amagawa {N}umber {C}onjecture}, https://arxiv.org/abs/2312.09849 (2023).

\bibitem{gambheera-popescu}
Rusiru Gambheera and Cristian~D. Popescu, \emph{An unconditional main conjecture in {I}wasawa theory and applications}, 2023.

\bibitem{Greither-Brumer}
Cornelius Greither, \emph{Some cases of {B}rumer's conjecture for abelian {CM} extensions of totally real fields}, Math. Z. \textbf{233} (2000), no.~3, 515--534. \MR{1750935}

\bibitem{Greither-Kataoka-Kurihara}
Cornelius Greither, Takenori Kataoka, and Masato Kurihara, \emph{Fitting ideals of {$p$}-ramified {I}wasawa modules over totally real fields}, Selecta Math. (N.S.) \textbf{28} (2022), no.~1, Paper No. 14, 48. \MR{4350205}

\bibitem{Greither-Kurihara}
Cornelius Greither and Masato Kurihara, \emph{Stickelberger elements, {F}itting ideals of class groups of {CM}-fields, and dualisation}, Math. Z. \textbf{260} (2008), no.~4, 905--930. \MR{2443336}

\bibitem{Greither-Popescu-Picard}
Cornelius Greither and Cristian~D. Popescu, \emph{The {G}alois module structure of {$\ell$}-adic realizations of {P}icard 1-motives and applications}, Int. Math. Res. Not. IMRN (2012), no.~5, 986--1036. \MR{2899958}

\bibitem{Greither-Popescu}
\bysame, \emph{An equivariant main conjecture in {I}wasawa theory and applications}, J. Algebraic Geom. \textbf{24} (2015), no.~4, 629--692. \MR{3383600}

\bibitem{Gruenberg-Weiss}
K.~W. Gruenberg and A.~Weiss, \emph{Galois invariants for local units}, Quart. J. Math. Oxford Ser. (2) \textbf{47} (1996), no.~185, 25--39. \MR{1380948}

\bibitem{Jannsen}
Uwe Jannsen, \emph{Iwasawa modules up to isomorphism}, Algebraic number theory, Adv. Stud. Pure Math., vol.~17, Academic Press, Boston, MA, 1989, pp.~171--207. \MR{1097615}

\bibitem{Kataoka}
Takenori Kataoka, \emph{Fitting invariants in equivariant {I}wasawa theory (research announcement)}, Algebraic number theory and related topics 2017, RIMS K\^{o}ky\^{u}roku Bessatsu, vol. B83, Res. Inst. Math. Sci. (RIMS), Kyoto, 2020, pp.~129--140. \MR{4279703}

\bibitem{Kurihara}
Masato Kurihara, \emph{Notes on the dual of the ideal class groups of {CM}-fields}, J. Th\'eor. Nombres Bordeaux \textbf{33} (2021), no.~3, 971--996. \MR{4402386}

\bibitem{Popescu-Stark}
Cristian~D. Popescu, \emph{Integral and {$p$}-adic refinements of the abelian {S}tark conjecture}, Arithmetic of {$L$}-functions, IAS/Park City Math. Ser., vol.~18, Amer. Math. Soc., Providence, RI, 2011, pp.~45--101. \MR{2882687}

\bibitem{Popescu-Yin}
Cristian~D. Popescu and Wei Yin, \emph{Fitting ideals of projective limits of modules over non--{N}oetherian {I}wasawa algebras}, https://arxiv.org/abs/2409.11562 (2024), 31 pages.

\bibitem{Ritter-Weiss}
J\"{u}rgen Ritter and Alfred Weiss, \emph{A {T}ate sequence for global units}, Compositio Math. \textbf{102} (1996), no.~2, 147--178. \MR{1394524}

\bibitem{Shalev}
Aner Shalev, \emph{On the number of generators of ideals in local rings}, Adv. in Math. \textbf{59} (1986), no.~1, 82--94. \MR{825089}

\bibitem{Swan}
Richard~G. Swan, \emph{Induced representations and projective modules}, Ann. of Math. (2) \textbf{71} (1960), 552--578. \MR{138688}

\bibitem{Tate-CFT}
J.~Tate, \emph{Number theoretic background}, Automorphic forms, representations and {$L$}-functions ({P}roc. {S}ympos. {P}ure {M}ath., {O}regon {S}tate {U}niv., {C}orvallis, {O}re., 1977), {P}art 2, Proc. Sympos. Pure Math., vol. XXXIII, Amer. Math. Soc., Providence, RI, 1979, pp.~3--26. \MR{546607}

\bibitem{Washington}
Lawrence~C. Washington, \emph{Introduction to cyclotomic fields}, second ed., Graduate Texts in Mathematics, vol.~83, Springer-Verlag, New York, 1997. \MR{1421575}

\bibitem{Weiss}
A.~Weiss, \emph{Multiplicative {G}alois module structure}, Fields Institute Monographs, vol.~5, American Mathematical Society, Providence, RI, 1996. \MR{1386895}

\bibitem{wiles}
A.~Wiles, \emph{The {I}wasawa conjecture for totally real fields}, Ann. of Math. (2) \textbf{131} (1990), no.~3, 493--540. \MR{1053488}

\end{thebibliography}
\end{document}